\documentclass[11pt]{amsart}

\usepackage{epsfig,color}
\usepackage{esint}
\usepackage{MnSymbol}
\usepackage{amsmath}
\usepackage{amsthm}
\usepackage{hyperref,tikz}
\usepackage[margin=1.0in]{geometry}

\numberwithin{equation}{section}

\newtheorem{prop}{Proposition}
\newtheorem{lemma}[prop]{Lemma}

\newtheorem{thm}[prop]{Theorem}
\newtheorem{cor}[prop]{Corollary}

\numberwithin{prop}{section}

\theoremstyle{definition}
\newtheorem{defn}[prop]{Definition}

\newtheorem{rmk}[prop]{Remark}

\newcommand{\del}{\partial}
\newcommand{\dt}{\frac{\partial}{\partial t}}
\newcommand{\brs}[1]{\left| #1 \right|}

\newcommand{\gG}{\Gamma}

\newcommand{\gD}{\Delta}
\newcommand{\gd}{\delta}
\newcommand{\gs}{\sigma}

\newcommand{\gl}{\lambda}
\newcommand{\gL}{\Lambda}

\newcommand{\ga}{\alpha}
\newcommand{\gb}{\beta}
\renewcommand{\ge}{\epsilon}
\newcommand{\N}{\nabla}
\newcommand{\FF}{\mathcal F}

\newcommand{\OO}{\mathcal O}

\newcommand{\LL}{\mathcal L}
\newcommand{\til}[1]{\widetilde{#1}}

\renewcommand{\bar}[1]{\overline{#1}}

\newcommand{\IP}[1]{\left<#1\right>}

\DeclareMathOperator{\Sym}{Sym}
\DeclareMathOperator{\Ric}{Ric}
\DeclareMathOperator{\Rm}{Rm}

\DeclareMathOperator{\tr}{tr}

\newcommand{\gU}{\Upsilon}

\newcommand{\dtu}{\frac{\del u}{\del t}}

\newcommand{\C}{\mathcal{C}^{+}}

\begin{document}

\title[Uniqueness for the $\sigma_2$-Yamabe problem]{A formal Riemannian structure on conformal classes and uniqueness for the
\texorpdfstring{$\gs_{2}$}{sigmak}-Yamabe problem}

\author{Matthew Gursky}
\address{Department of Mathematics
         University of Notre Dame\\
         Notre Dame, IN 46556}
\email{\href{mailto:mgursky@nd.edu}{mgursky@nd.edu}}

\author{Jeffrey Streets}
\address{Department of Mathematics\\
         University of California\\
         Irvine, CA  92617}
\email{\href{mailto:jstreets@uci.edu}{jstreets@uci.edu}}

\date{\today}

\begin{abstract} We define a new formal Riemannian metric on a
conformal classes of four-manifolds in the context of the $\gs_2$-Yamabe problem.  Exploiting this new variational structure we
show that solutions are unique unless the manifold is conformally equivalent to the round sphere.
\end{abstract}

\maketitle

\section{Introduction}
\subsection{Background}   In \cite{GS_Surfaces}, we defined a formal Riemannian metric on the space of conformal metrics on surfaces of positive (or negative) Gauss curvature.   Our goal in this paper is to show that one can extend this definition to conformal classes of metrics on four-manifolds, and to explore the geometric properties of this metric and their applications. The definition we give can be extended to higher (even) dimensions, but this will be pursued in a subsequent article since there are technical issues that do not arise in two or four dimensions \cite{GS_RVC}.

In addition to verifying the formal properties of this metric we prove a remarkable geometric consequence: namely, solutions of the $\sigma_2$-Yamabe problem -- whose existence follows from our positivity assumption and \cite{CGY1} -- are \emph{unique}, unless the manifold is conformally equivalent to the sphere.  This is a surprising departure from the classical (or $\sigma_1$-)Yamabe problem, where explicit examples of non-uniqueness are known (see Remarks \ref{nonuniqueremark} and \ref{nonUexamples} below).  Thus, positive conformal classes on four-manifolds have a unique conformal representative whose $\sigma_2$-curvature is constant; moreover the value of this constant (after normalizing the volume) can be expressed in terms of the Euler characteristic and the $L^2$-norm of the Weyl tensor (see the introduction of \cite{CGYAnnals}).  We also remark that this representative has positive Ricci curvature.

To give a more detailed description it will be helpful to return to the setting of surfaces.
Let $(M, g_0)$ be a compact Riemannian
surface with positive Gauss curvature $K_0 > 0$, and let $[g_0]$ denote the conformal class of $g_0$.  Define
\begin{align} \label{conedefintro}
\mathcal{C}^{+} = \{ g_u = e^{2u}g_0 \in [g_0] \ :\ K_u = K_{g_{u}} > 0 \}.
\end{align}
Formally, the tangent space to $[g_0]$ at any metric $g_u \in [g_0]$ is given by $C^{\infty}(M)$.
For $\phi,\psi \in C^{\infty}(M) \cong T_u([g_0])$ we define
\begin{align} \label{gsmetric}
 \llangle \phi, \psi \rrangle_u = \int_M \phi \psi K_u dA_u,
\end{align}
where $K_u$ is the Gauss curvature and $dA_u$ is the area form of $g_u$.

The definition in (\ref{gsmetric}) is inspired by the
Mabuchi-Semmes-Donaldson \cite{MabuchiSymp, Semmes, Donaldson} metric of
K\"ahler geometry, wherein a formal
Riemann metric is put on a K\"ahler class by imposing on the tangent space to a
given K\"ahler potential the $L^2$ metric with respect to the associated
K\"ahler metric.  As observed in \cite{MabuchiSymp}, this metric enjoys many
nice formal
properties, for instance nonpositive sectional curvature.  Moreover, it has a
profound relationship to natural functionals in K\"ahler geometry such as the
Mabuchi $K$-energy and the Calabi energy, as well as their gradient flow, the
Calabi flow.

In \cite{GS_Surfaces} we established a number of analogous properties for the metric
defined by (\ref{gsmetric}).  For example, $\mathcal{C}^{+}$ endowed with the metric in (\ref{gskmetric}) has
non-positive curvature in the sense of Alexandrov.  We also showed that the normalized Liouville energy $F : W^{1,2} \rightarrow \mathbb{R}$, defined by
\begin{align} \label{LEdef}
F[u] = \int_M |\nabla_0 u|^2 dA_0 + 2 \int_M K_0 u dA_0 - \big( \int_M K_0 dA_0 \big) \log \Big( \fint_M e^{2u} dA_0  \Big),
\end{align}
is {\em geodesically convex}.  Recall that critical points of $F$, which are precisely the conformal metrics of constant Gauss curvature, are minimizers and unique up to M\"oebius transformation.  Many of these global geometric properties are based on existence and partial
regularity results for geodesics in $\mathcal{C}^{+}$ (see Section 4 of \cite{GS_Surfaces} for precise statements).

In this paper we study a natural generalization of the inner product (\ref{gskmetric}).  For an $n$-dimensional Riemannian manifold ($n \geq 3$), we denote the Schouten tensor by
\begin{align*}
A = \frac{1}{(n-2)} \big( Ric - \frac{1}{2(n-1)} R g \big),
\end{align*}
where $Ric$ is the Ricci tensor and $R$ is the scalar curvature.  Let $\sigma_k(g^{-1} A)$ denote the $k^{th}$-symmetric function of the eigenvalues of the $(1,1)$ tensor obtained by raising an index of $A$; i.e.,
\begin{align*}
A_i^j = g^{jk} A_{ik}.
\end{align*}
The quantity $\sigma_k(g^{-1}A)$ is called the {\em $\sigma_k$-curvature} or the {\em $k$-scalar curvature}.  For example,
\begin{align} \label{s1A}
\sigma_1(g^{-1} A) = \dfrac{R}{2(n-1)}.
\end{align}
For $1 \leq k \leq n$, we write $A = A_g \in \Gamma_k^{+}$ if $\sigma_j(g^{-1} A) > 0$ on $M^n$ for all $1 \leq j \leq k$.  By (\ref{s1A}), we have $A_g \in \Gamma_1^{+}$ if $g$ has positive scalar curvature, while $A_g \in \Gamma_n^{+}$ if the Schouten tensor of $g$ is positive definite.

We will be interested in the case where $n = 4$ and $k =2$.  To this end, let $(M^4, g_0)$ be a compact Riemannian four-manifold such that $A_{g_0} \in \gG_{2}^+$.  Given $u \in C^{\infty}(M)$, let $A_u$ denote
the Schouten tensor of the conformal metric $g_u = e^{-2 u} g_0$.  We will say that $u$ is {\em admissible} if $A_u \in \Gamma_2^{+}$.    Let
\begin{align*}
\mathcal{C}^{+} = \mathcal{C}^{+}([g_0]) =  \big\{ g_u \in [g_0] \ \vert \ A_u \in \Gamma_{2}^{+} \big\}.
\end{align*}
By a result of Guan-Viaclovsky, \cite{GuanVia1}, if $g_u \in \mathcal{C}^{+}$ then $g_u$ has positive Ricci curvature.  As noted above, the tangent space to $\mathcal{C}^{+}$ at any point is given by $C^{\infty}(M)$.
 Thus, in analogy with (\ref{gskmetric}) we define for $\phi,\psi \in C^{\infty}(M)$
\begin{align} \label{gskmetric}
 \IP{ \phi, \psi}_u =&\ \int_M \phi \psi \gs_{2}(g^{-1}_u A_u) dV_u.
\end{align}

\begin{rmk} To simplify the notation we will write $\sigma_2(A)$ instead of $\sigma_2(g^{-1}A)$.  Since we will be working with conformal metrics, we will also need to distinguish between $g^{-1}A_u$ and $g_u^{-1}A_u$; i.e., whether we are using $g$ or $g_u$ to raise an index.  Therefore, we will adopt the usual convention that $\sigma_2(A_u) = \sigma_2( g^{-1} A_u)$, but write $\sigma_2(g_u^{-1}A_u)$ when we are using $g_u$ to raise an index.  Note that
\begin{align} \label{note}
\sigma_2(g_u^{-1} A_u) = e^{4u} \sigma_2(A_u).
\end{align}
In particular,
\begin{align*}
\sigma_2(g_u^{-1} A_u) dV_u = \sigma_2(A_u) dV.
\end{align*}
\end{rmk}

\begin{rmk} There is a sharp characterization of
conformal classes for which $\mathcal{C}^{+}$ is non-empty.  In view of the conformal invariance of the integral
\begin{align*}
\sigma := \int \sigma_{2}(g^{-1} A_g) dV_g,
\end{align*}
a necessary condition for $[g]$ to admit a metric $g_u \in [g]$ with $A_u \in \Gamma_2^{+}$ is the positivity of the Yamabe invariant and the positivity of $\sigma$.  In \cite{CGYAnnals} these conditions were shown to be sufficient.  Thus we have an exact parallel with the case
of two dimensions, since a conformal class of metrics on a surface admits a metric of positive Gauss curvature if and only if the total Gauss curvature is positive.
\end{rmk}

%
%
%
\subsection{Formal metric properties}

We begin by establishing in \S \ref{formalgeometry} some
fundamental formal properties of the metric defined in (\ref{gskmetric}).  We first introduce
a formal path derivative which can be regarded as the Levi-Civita connection
associated to the metric.  Using this we compute the curvature tensor, and
furthermore show that the curvature is nonpositive: \\

\begin{thm} \label{npcthm} Given $(M^4, g)$ a compact Riemannian manifold, with
$A_g \in \gG_{2}^+$.  Then (\ref{gskmetric}) defines a metric with
nonpositive sectional curvature on $\mathcal{C}^{+}$.
\end{thm}

Next, we derive the geodesic equation.  Formal calculations derived using either
the path derivative or variations of the length functional yield that a
one-parameter family of conformal factors is a geodesic if and only if
\begin{align} \label{geodesics}
u_{tt} - \frac{1}{\gs_2(A_u)} \IP{T_{1} (A_u), \N u_t \otimes \N u_t} = 0,
\end{align}
where $T_{1}$ is the Newton transform and $\langle \cdot, \cdot \rangle$ denotes the inner product on
tensor bundles induced by $g$ (the background metric).
This is a degenerate fully nonlinear equation, which is related to a
$\sigma_{2}$-type problem for the spacetime Hessian of $u$, in direct
analogy to the $(n+1)$-dimensional degenerate Monge-Ampere interpretation of the
Mabuchi geodesic equation in K\"ahler geometry.
We also show that one parameter families of conformal
transformations are automatically geodesics (Proposition
\ref{conformalpathsgeodesics}).  This is again in analogy with the fact that
one-parameter families of biholomorphisms generate families of K\"ahler
potentials which are Mabuchi geodesics.

In the K\"ahler setting, the Mabuchi metric and its geodesics are intimately
related to Mabuchi's $K$-energy functional.  This is a ``relative functional''
defined
via path integration of a closed $1$-form on a K\"ahler class.  It was shown in
\cite{Mabuchi,MabuchiSymp} that this functional is geodesically convex, leading
to the conjecture
that extremal K\"ahler metrics are unique up to biholomorphism in a fixed
K\"ahler class.  Confirming this conjecture requires extensive existence and
regularity results for the geodesic equation.  An initial theory of $C^{1,1}$
was developed in \cite{ChenSOKM,CalabiChen,Blocki}, and eventually a more
refined regularity theory was developed and the conjecture finally confirmed in
\cite{ChenTian}.

In our setting there is a natural analogue of Mabuchi's functional.  For surfaces it is given
by the Liouville energy, or regularized determinant (\ref{LEdef}).    In four dimensions this functional was written down by Chang-Yang in \cite{ChangYangMoserVol} (although it appears implicitly in \cite{CGYAnnals}):
\begin{align} \label{CYP} \begin{split}
F[u] &= \int \Big\{ 2 \Delta u |\nabla u|^2 - |\nabla u|^4 - 2 Ric(\nabla u,\nabla u) + R|\nabla u|^2 - 8 u \sigma_2(A_g) \Big\}dV \\
& \ \ \ \ - 2 \big( \int \sigma_2(A_g) dV \big) \log \Big( \fint e^{-4u} dV \Big).
\end{split}
\end{align}
After this, Brendle-Viaclovsky \cite{BV} give a path-integration derivation of this functional which makes clearer the analogy between it and the Mabuchi functional in K\"ahler geometry.  We will not need the precise formula, only the fact that it provides a conformal primitive for $\sigma_2(A)$; i.e., if $u_s$ is a path with $\frac{d}{ds}u_s\vert_{s=0} = u'$, then
\begin{align} \label{CYvar}
\frac{d}{ds}F[u_s]\big\vert_{s=0} = \int u' \big[ - \sigma_2(g_u^{-1} A_u) + \bar{\sigma} \big] dV_u.
\end{align}
Consequently, $u$ is a critical point of $F$ if and only if $g_u = e^{-2u}g$ is a solution of the $\sigma_2$-Yamabe problem:
\begin{align}  \label{skYamabe}
\sigma_2(g_u^{-1} A_u) \equiv const.
\end{align}

In four dimensions the existence of solutions to (\ref{skYamabe}) in conformal classes with $\mathcal{C}^{+} \neq \emptyset$ was first proved by Chang-Gursky-Yang \cite{CGY1} (for surveys on solving the $\sigma_k$-Yamabe problem for general $2 \leq k \leq n$ see \cite{JeffVSurvey} and \cite{STWSurvey}).  In particular, if $\mathcal{C}^{+}([g])$ is non-empty, then $[g]$ always admits a critical point of $F$.  Our next result gives us deeper insight into the variational structure of $F$:

\begin{thm}  \label{FconThm}  The functional $F$ in (\ref{CYP}) is geodesically convex. \end{thm}

The proof of this theorem requires the use of a sharp curvature-weighted Poincar\'e inequality due to Andrews \cite{Andrews}.
In fact, it follows from Andrews' inequality that $F$ is strictly convex, up to to one-parameter families of conformal automorphisms on the round sphere.
This sharp characterization naturally leads one to conjecture that critical points of $F$ are unique, except in the case of the sphere.
 We are able to confirm this surprising fact:

%
%
%
%

\begin{thm} \label{unique} Let $(M^4, g)$ be a compact Riemannian manifold such
that $\mathcal{C}^{+}([g]) \neq \emptyset$.
\begin{enumerate}
\item If $(M^4, g)$ is not conformal to $(S^4, g_{S^4})$, then there exists a
unique solution to the $\sigma_2$-Yamabe problem in $[g]$.
\item In $[g_{S^4}]$, all solutions to the $\sigma_2$-Yamabe problem are round
metrics.
\end{enumerate}
\end{thm}

\begin{rmk} \label{nonuniqueremark} This uniqueness property is in stark contrast to the
Yamabe problem, in which generic conformal classes admit arbitrarily many distinct solutions (see \cite{Pollack}).
In dimensions $n \geq 25$ the solution space may even be non-compact \cite{Brendle},\cite{BM}.   \end{rmk}

\begin{rmk}  \label{nonUexamples} Explicit examples of non-uniqueness for the Yamabe problem were constructed by Schoen in \cite{SchoenNU}, in which
he constructed Delaunay-type solutions on $S^{n-1} \times S^1$.  By lifting to the universal cover $S^{n-1} \times \mathbb{R}$
and imposing symmetry, he reduced the Yamabe equation to an ODE and studied the phase portrait.  Interestingly, Viaclovsky \cite{JeffAMS} carried out a similar construction for
solutions of the $\sigma_k$-Yamabe problem when $k < n/2$.  However, once $k \geq n/2$ the construction fails, since the admissibility
condition implies the Ricci curvature of any solution would have to be positive, and $S^{n-1} \times S^1$ does not admit a metric
with positive Ricci curvature.
\end{rmk}


The proof of Theorem \ref{unique} consists of two main phases.  First we develop
a weak existence/regularity theory for the geodesic equation (\ref{geodesics}).
In general for degenerate Monge-Ampere equations one typically expects at best
$C^{1,1}$ control, and indeed this is verified in the K\"ahler setting by Chen
(with complements due to Blocki)
\cite{ChenSOKM, Blocki}.  Where Mabuchi
geodesics can be interpreted as solutions of a degenerate complex Monge-Ampere
equation, our geodesics are solutions to a degenerate $\gs_2$-equation
(Proposition \ref{geodMA}), and so one at best again expects $C^{1,1}$
regularity.
However, due to some technical issues arising from the presence of first order
terms in the Schouten tensor, we are not able to establish such estimates.
Rather we are forced to regularize the equation by rendering the right hand side
positive (which is a standard trick), but also perturbing the coefficients on
the time direction term, to further break the nondegeneracy.  This leads to full
$C^{\infty}$ regularity, but only the $C^1$-estimates persist as the
regularization parameters go to zero.

Given this, one cannot directly rigorously establish properties of $F$ related
to the geodesic convexity.  Nonetheless we are able to improve the regularity of
an approximate geodesic connecting any two solutions to the $\sigma_2$-problem
by smoothing via the parabolic flow introduced by Guan-Wang \cite{GW}.  In
particular we are able to take a sequence of approximate geodesics connecting
two critical points for $F$, smooth them for a short time with this flow, and
then show that this process yields a path of critical points for $F$, although
not necessarily a geodesic.  Combining this with arguments using the geodesic
convexity shows that the existence of this path implies that the critical points
are all round metrics on $S^4$, finishing the proof.

\subsection{Outline}
In \S \ref{bckgrnd} we establish notation and record some basic properties of
the Schouten tensor and of elementary symmetric polynomials.  Next in \S
\ref{formalgeometry} we establish the basic properties of the $\sigma_2$-metric defined in (\ref{gskmetric}).
 In particular we prove Theorem \ref{npcthm} and establish the geodesic
convexity of the $F$ functional.  Then in \S \ref{geeodestimates} we develop
estimates for approximate solutions to the geodesic equation, leading to a weak
existence theory.  In \S \ref{smoothingsec} we show a short-time smoothing
result which we will use to improve the regularity of approximate geodesics
connecting any two critical points of the $F$-functional.  We combine these two
main technical tools in \S \ref{uniquesec} to establish Theorem \ref{unique}.

\section{Background} \label{bckgrnd}

In this section we establish our notation and some basic formulas.  Although we are primarily interested in four dimensions, we
will state most of the standard results for symmetric functions we will need for general $n$ and $k$.


%
%
\subsection{The Schouten tensor}

Given a Riemannian manifold $(M^n,g)$ let $A$ denote the Schouten tensor of $g$.
 Given a conformal metric $g_u = e^{-2u}g$, the tensor $A$ transforms according
to
\begin{align} \label{Au}
A_u = A + \N^2 u + \N u \otimes \N u - \frac{1}{2} \brs{\N u}^2 g.
\end{align}

Let $g_u = e^{-2u(t)}g$ be a $1$-parameter family of conformal metrics.  Then
using formula (\ref{Au}) it follows that
\begin{align} \label{Ap}
\frac{\partial}{\partial t} (g_u^{-1} A_u)_i^j = 2 ( \frac{\partial u}{\partial t})
(g_u^{-1} A_u)_i^j + (\nabla_u^2 \frac{\partial u}{\partial t})_i^j,
\end{align}
where the Hessian is with respect to $g_u$. A direct calculation (\cite{Reilly})
yields
\begin{align} \label{dsig} \begin{split}
\frac{\partial}{\partial t} \sigma_k(g_u^{-1} A_u) &= \langle T_{k-1}(g_u^{-1} A_u), \nabla^2_u
\frac{\partial u}{\partial t} \rangle_{g_u} + 2k
\frac{\partial u}{\partial t} \sigma_k(g_u^{-1}A_u),
\end{split}
\end{align}
where $T_{k-1}$ is the Newton transform. Since the Newton transform is a $(1,1)$-tensor, for the pairing in (\ref{dsig}) we {\em lower} an index of $T_{k-1}(g_u^{-1}A_u)$ and view it as a $(0,2)$-tensor, and use the inner product induced by $g_u$.  For example, if $n=4$ and $k=2$,
\begin{align} \label{Tk}
T_1(g_u A_u) = - A_u + \sigma_1(g_u^{-1} A_u) g_u.
\end{align}

Combining (\ref{dsig}) with the variation of the volume form yields
\begin{align}  \label{densep}
\frac{\partial}{\partial t} \big[ \sigma_k(g_u^{-1} A_u) dV_u \big] = \langle
T_{k-1}(g_u^{-1} A_u), \nabla_u^2 \frac{\partial u}{\partial t} \rangle_{g_u} dV_u + (n - 2k)
\frac{\partial u}{\partial t} \sigma_k(g_{u}^{-1} A_u) dV_u.
\end{align}

A key property we will use throughout is the following:

\begin{lemma} \label{divT} If $k = 2$ or if the manifold is locally conformally flat, then $T_{k-1}(g^{-1}A)$ is divergence-free.
\end{lemma}

\begin{rmk}  This was proved in \cite{JeffThesis}.  The essential idea also appears in \cite{Reilly}, where the Schouten tensor is replaced
with the second fundamental form of a hypersurface of a space of constant curvature.  In both cases one needs that the tensor is
Codazzi; i.e.,
\begin{align*}
\N_k A_{ij} = \N_j A_{ik}.
\end{align*}
\end{rmk}

Note that the conformal invariance of the integral
\begin{align*}
\gs = \int_M \gs_2(g_u^{-1} A_u) dV_u
\end{align*}
follows from the variational formula (\ref{densep}) and Lemma \ref{divT}.  We denote the average value
\begin{align}
 \bar{\gs} = \gs V_u^{-1}.
\end{align}

\subsection{Properties of elementary symmetric polynomials}

We record some lemmas concerning elementary symmetric polynomials and
Newton transforms.  To begin we record basic facts which are well-known from
Garding's theory of hyperbolic polynomials \cite{Garding}.  We use these to
derive some further properties of generalized Newton transforms required for our
estimates of the geodesic equation.  First, given $A \in \gG_k^+$ we let
$\gs_k(A)$ denote the $k$-th elementary polynomial in the eigenvalues of $A$.
Moreover, given $A_1,\dots,A_k$ we define the generalized Newton transformation
by
\begin{align*}
\left[ T_k \right]_{ij}(A_1,\dots,A_k) := \frac{1}{k!}
\gd^{i,i_1,\dots,i_k}_{j,j_1,\dots,j_k} (A_1)_{i_1j_1} \dots (A_k)_{i_k j_k},
\end{align*}
where here $\gd$ denotes the generalized Kronecker delta function.  Moreover we
set
\begin{align*}
\Sigma_k(A_1,\dots,A_k) = \frac{1}{(k-1)!} \gd^{i_1,\dots,i_k}_{j_1,\dots,j_k}
(A_1)_{i_1j_1} \dots (A_k)_{i_k j_k}.
\end{align*}

\begin{lemma} \label{newtonprops} One has
\begin{enumerate}
\item Given $A_1,\dots,A_k \in \gG_k^+$, then $[T_k]_{ij} (A_1,\dots,A_k) > 0$.
\item Given $A_1,\dots,A_k \in \gG_k^+$, then $\Sigma(A_1,\dots,A_k) > 0$.
\item If $A - B \in \gG_k^+$ and $A_2,\dots,A_k \in \gG_k^+$ then
$\Sigma(B,A_2,\dots,A_k) < \Sigma(A,A_2,\dots,A_k)$.
\end{enumerate}
\end{lemma}

\begin{lemma} \label{newtonmonotone} Given $A,B \in \gG_k^+$ $A < B$, one has
$T_{k-1}(A) < T_{k-1}(B)$.
 \begin{proof} From Lemma \ref{newtonprops}, for $A_i \in \gG_k$ one has
$T_k(A_1,\dots,A_k) > 0$.  Now consider $M_t = A +
t(B - A)$.  Since $B - A$ is positive definite certainly it lies in $\gG_k^+$.
It follows that
\begin{align*}
 \frac{d}{dt} T_k(M_t) =&\ \frac{d}{dt} [T_k](M_t,\dots,M_t)\\
 =&\ \sum_{j=1}^k [T_k] (M_t,\dots,B - A,\dots M_t)\\
 \geq&\ 0.
\end{align*}
The result follows.
 \end{proof}
\end{lemma}

\begin{lemma} \label{LAlemma} Given $A$ a symmetric matrix and $X$ a vector, one
has for $k \geq
1$,
\begin{align*}
 \IP{T_k(A - X \otimes X), X \otimes X} =&\ \IP{T_k(A), X \otimes X},\\
\gs_k(A - X \otimes X) =&\ \gs_k(A) - \IP{T_{k-1}(A),X \otimes X}.
\end{align*}
\begin{proof}
If we express the matrix $B_t = A - t X \otimes X$ in a basis where $X$ is the
first basis vector, it is clear that the function
\begin{align*}
f(t) = \gs_k(B_t)
\end{align*}
is a linear function of $t$.  It follows that its time derivative is constant,
hence
\begin{align*}
C = f'(t) = - \IP{T_{k-1}(A - t X \otimes X), X \otimes X}.
\end{align*}
Hence
\begin{align*}
\IP{T_{k-1}(A),X \otimes X} = -f'(0) = -f'(1) = \IP{T_{k-1}(A - X \otimes X), X
\otimes X}.
\end{align*}
Moreover, this shows that
\begin{align*}
\gs_k(A - X \otimes X) = f(1) = f(0) + \int_0^1 f'(s) ds = \gs_k(A) -
\IP{T_{k-1}(A),X \otimes X}.
\end{align*}
\end{proof}
\end{lemma}

\begin{lemma} \label{convexityinequality} Given $A,B \in \Sym^2(\mathbb R^4)$,
$A,B \in \gG_2^+$ one has
\begin{align*}
\IP{T_1(B),A}^2 \geq 4 \gs_2(A) \gs_2(B).
\end{align*}
\begin{proof} We compute that
\begin{align*}
 \frac{\gs_1(A)}{\gs_1(B)} \IP{T_1(B),A} =&\ - \frac{\gs_1(A)}{\gs_1(B)}
\IP{B,A} + \gs_1(A)^2\\
 \geq&\ - \frac{1}{2} \left[ \frac{\gs_1(A)}{\gs_1(B)} \right]^2 \brs{B}^2 -
\frac{1}{2} \brs{A}^2 + \left[ \gs_1(A) \right]^2\\
 =&\ - \frac{1}{2} \gs_1(A)^2 \left[ \frac{\brs{B}^2 - \gs_1(B)^2 +
\gs_1(B)^2}{\gs_1(B)^2} \right] + \gs_2(A) + \frac{1}{2} \gs_1(A)^2\\
=&\ \frac{\gs_1(A)^2}{\gs_1(B)^2} \gs_2(B) + \gs_2(A).
\end{align*}
Rearranging this and applying Cauchy-Schwarz yields
\begin{align*}
 \gs_2(A) \leq&\ \frac{\gs_1(A)}{\gs_1(B)} \IP{T_1(B),A} -
\frac{\gs_1(A)^2}{\gs_1(B)^2} \gs_2(B)\\
 \leq&\ \frac{1}{4 \gs_2(B)} \IP{T_1(B),A}^2,
\end{align*}
as required.
\end{proof}
\end{lemma}

\section{The \texorpdfstring{$\sigma_2$}{sigmak}-metric} \label{formalgeometry}

In this section we define the $\sigma_2$-metric and establish fundamental
properties of this metric concerning connections, torsion, curvature and
distance.  We end by showing the crucial
geodesic convexity property of the functional $F$ of Chang-Yang.

\subsection{Metric, connection, and curvature}

As in the Introduction, let
\begin{align*}
\mathcal{C}^{+} = \mathcal{C}^{+}([g]) = \big\{ g_u = e^{-2u}g\ :\ A_u \in \Gamma_{2}^{+} \big\}.
\end{align*}

\begin{defn} \label{sigmakmetric}  Let $(M^4, g)$ be
a compact Riemannian four-manifold.  The \emph{$\gs_k$-metric} is the formal
Riemannian metric defined for $g_u \in \mathcal{C}^{+}([g]) = \mathcal{C}^{+}$, $\ga,\gb \in T_u \mathcal{C}^{+} \cong
C^{\infty}(M)$ via
\begin{align*}
\IP{\ga,\gb}_{u} = \frac{1}{\gs} \int_M \ga \gb \gs_2(g_u^{-1} A_u) dV_u.
\end{align*}
Moreover, given $u_t$ a path in $\mathcal{C}^{+}$ and $\ga_t$ a one-parameter family of
tangent vectors with $\ga_t \in T_{u_t} \C$, we define the directional derivative along the path $u_t$ by
\begin{align} \label{connectiondef} \begin{split}
\frac{D}{\del t} \ga &:= \ga_t - \gs_2(g_u^{-1} A_u)^{-1} \IP{T_{1}(g_u^{-1} A_u), \N \ga
\otimes \N u_t}_{g_u} \\
&= \ga_t - \gs_2(A_u)^{-1} \IP{T_{1}(A_u), \N \ga
\otimes \N u_t},
\end{split}
\end{align}
where we have used (\ref{note}), and the convention that $T_{1}(g^{-1}A_u) = T_{1}(A_u)$.
\end{defn}

\begin{lemma}  \label{metcomp} The connection
defined by (\ref{connectiondef}) is metric compatible and torsion free.
\begin{proof} First we check metric compatibility.  We compute, using (\ref{densep}) and Lemma \ref{divT},
\begin{align*}
\frac{d}{dt} \IP{\ga_t, \gb_t}_{u_t} =&\ \frac{d}{dt} \int_M \ga \gb \gs_2(g_u^{-1} A_u)
dV_u\\
=&\ \IP{\dot{\ga}, \gb} + \IP{\ga,\dot{\gb}} + \int_M \ga \gb
\IP{T_{1}(g_u^{-1} A_u),\N_u^2 \dtu} dV_u\\
=&\ \IP{\dot{\ga}, \gb} + \IP{\ga,\dot{\gb}} - \int_M \IP{T_{1}(g_u^{-1} A_u), \left(
\ga \N \gb + \gb \N \ga \right) \otimes \N_u \dtu} dV_u\\
=&\ \IP{ \frac{D}{\del t} \ga,\gb} + \IP{\ga,\frac{D}{\del t} \gb}.
\end{align*}
Next, to compute the torsion, let $u_{s,t}$ be a two parameter family of
conformal factors.
Then
\begin{align*}
\frac{D}{\del s} \frac{\del u}{\del t} - \frac{D}{\del t} \frac{\del u}{\del
s} =&\ \frac{\del^2 u}{\del s \del t} - \gs_2(g_u^{-1} A_u)^{-1} \IP{T_{1}(g_u^{-1} A_u), \N
\frac{\del u}{\del s} \otimes \N \frac{\del u}{\del t}}_u \\
&\ - \frac{\del^2 u}{\del s
\del t} + \gs_2(g_u^{-1} A_u)^{-1} \IP{ T_{1}(g_u^{-1} A_u), \N \frac{\del u}{\del t} \otimes \N
\frac{\del u}{\del s} }_u \\
=&\ 0.
\end{align*}
The lemma follows.
\end{proof}
\end{lemma}

Next we compute the sectional curvature, and conclude that it is non-positive. We first record an
integral identity in Lemma \ref{curvlemma1} and a certain general quadratic
inequality in Lemma \ref{sgLemma}. We then obtain the curvature inequality by
exploiting these identities.

\begin{lemma} \label{curvlemma1} If $\phi, \psi \in C^{\infty}(M)$, then
\begin{align*}  \begin{split}
& \int \Big\{ \nabla^2 \phi (\nabla \psi, \nabla \psi)  - \Delta \phi \vert
\nabla \psi \vert^2  - \nabla^2 \psi (\nabla \psi, \nabla \phi ) + \Delta \psi
\langle \nabla \psi, \nabla \phi \rangle \Big\} \phi dV \\
&\quad \quad = \int \Big\{ - \vert \langle \nabla \phi, \nabla \psi \rangle
\vert^2 + \vert \nabla \phi \vert^2 \vert \nabla \psi \vert^2 \Big\} dV.
\end{split}
\end{align*}
\end{lemma}

\begin{proof}  Consider the vector field
\begin{align*}
X_i = \langle \nabla \phi, \nabla \psi \rangle \nabla_i \psi - |\nabla \psi|^2
\nabla_i \phi.
\end{align*}
Taking the divergence gives
\begin{align*} \begin{split}
\delta X &= \nabla_i X_i \\
&= \nabla^2 \phi(\nabla \psi, \nabla \psi) + \nabla^2 \psi(\nabla \phi, \nabla
\psi) + \Delta \psi \langle \nabla \phi, \nabla \psi \rangle  \\
& \quad -2 \nabla^2 \psi(\nabla \psi, \nabla \phi) - \Delta \phi |\nabla \psi|^2
\\
&= \nabla^2 \phi (\nabla \psi, \nabla \psi)  - \Delta \phi \vert \nabla \psi
\vert^2  - \nabla^2 \psi (\nabla \psi, \nabla \phi ) + \Delta \psi \langle
\nabla \psi, \nabla \phi \rangle.
\end{split}
\end{align*}
Therefore,
\begin{align*} \begin{split}
I &\equiv \int \Big\{ \nabla^2 \phi (\nabla \psi, \nabla \psi)  - \Delta \phi
\vert \nabla \psi \vert^2  - \nabla^2 \psi (\nabla \psi, \nabla \phi ) + \Delta
\psi \langle \nabla \psi, \nabla \phi \rangle \Big\} \phi dV \\
&= \int \big( \delta X\big) \phi dV.
\end{split}
\end{align*}
On the other hand, integrating by parts gives
\begin{align*} \begin{split}
I &= \int \big( \delta X\big) \phi dV \\
&= - \int \langle X, \nabla \phi \rangle dV \\
&= \int \Big\{ - \vert \langle \nabla \phi, \nabla \psi \rangle \vert^2 + \vert
\nabla \phi \vert^2 \vert \nabla \psi \vert^2 \Big\} dV,
\end{split}
\end{align*}
as claimed.
\end{proof}

\begin{lemma} \label{sgLemma}  Let $T_1 = T_1(A)$ denote the first Newton
transformation of the symmetric linear map $A : V \rightarrow V$, where $V$ is a
real inner product
space of dimension four. Assume $A \in \Gamma_2^{+}$.  Then for all $X,Y \in V$,
\begin{align*}
-T_1(X,X) T_1(Y,Y) + T_1(X,Y)^2 + \sigma_2(A) \big[ |X|^2 |Y|^2 - \langle X,Y
\rangle^2 \big]  \leq 0.
\end{align*}
\end{lemma}

\begin{proof}  Choose an orthonormal basis for $V$ which diagonalizes $T_1$, and
let $\{ \lambda_1, \dots, \lambda_4\}$ denote the eigenvalues of $T_1$.  Note by
our assumption on $A$
we know that $\lambda_i \geq 0$ for each $i$.  With respect to this orthornormal
basis write $X = (x_1,\dots,x_4)$ and $Y =
(y_1,\dots,y_4)$. Then expanding and collecting terms we get
\begin{align*} \begin{split}
 & -T_1(X,X) T_1(Y,Y)  + T_1(X,Y)^2  \\
& \quad = - \big\{ \lambda_1 x_1^2 + \cdots + \lambda_4 x_4^2 \big\}\big\{
\lambda_1 y_1^2 + \cdots + \lambda_4 y_4^2 \big\} + \big\{ \lambda_1 x_1 y_1 +
\cdots \lambda_4 x_4 y_4 \big\}^2 \\
& \quad = -\lambda_1 \lambda_2 \big( x_1^2 y_2^2 + x_2^2 y_1^2 - 2 x_1 x_2 y_1
y_2 \big) - \lambda_1 \lambda_3 \big( x_1^2 y_3^2 + x_3^2 y_1^2 - 2 x_1 x_3 y_1
y_3 \big) \\
& \quad \quad \quad - \cdots - \lambda_3 \lambda_4 \big( x_3^2 y_4^2 + x_4^2
y_3^2 - 2 x_3 x_4 y_3 y_4 \big).
\end{split}
\end{align*}
Next, let
\begin{align*}
Z = X \wedge Y,
\end{align*}
whose components are
\begin{align*}
z_{ij} = x_i y_j - x_j y_i.
\end{align*}
In terms of $Z$, we can rewrite the above as
\begin{align*}
 & -T_1(X,X) T_1(Y,Y)  + T_1(X,Y)^2 = -\lambda_1 \lambda_2 z_{12}^2 - \lambda_1
\lambda_3 z_{13}^2 - \cdots - \lambda_3 \lambda_4 z_{34}^2.
\end{align*}
At the same time,
\begin{align*} \begin{split}
|X|^2 |Y|^2 - \langle X,Y \rangle^2  &= \frac{1}{2}|Z|^2 \\
&= z_{12}^2 + z_{13}^2 + \cdots + z_{34}^2.
\end{split}
\end{align*}
Therefore,
\begin{align} \label{sPdef} \begin{split}
& -T_1(X,X) T_1(Y,Y) + T_1(X,Y)^2 + \sigma_2(A) \big[ |X|^2 |Y|^2 - \langle X,Y
\rangle^2 \big] \\
&=  -\lambda_1 \lambda_2 z_{12}^2 - \lambda_1 \lambda_3 z_{13}^2 - \cdots -
\lambda_3 \lambda_4 z_{34}^2 + \sigma_2(A) \big[ z_{12}^2 + z_{13}^2 + \cdots +
z_{34}^2 \big].
\end{split}
\end{align}
We need to express $\sigma_2(A)$ in terms of the eigenvalues of $T_1$.   Since
\begin{align} \label{TS}
T_1 = - A + \sigma_1(A) \cdot I,
\end{align}
taking the trace it follows that
\begin{align*}
\lambda_1 + \cdots + \lambda_4 = 3 \sigma_1(A).
\end{align*}
Also, taking the norm-squared in (\ref{TS}),
\begin{align*}
|T_1|^2 = |A|^2 + 2 \sigma_1(A)^2.
\end{align*}
Therefore,
\begin{align*}
\sigma_2(A) = \frac{1}{3} \big( -\lambda_1^2  - \cdots - \lambda_4^2  +
\lambda_1 \lambda_2 + \lambda_1 \lambda_3 + \cdots \lambda_3 \lambda_4 \big).
\end{align*}
Substituting this into (\ref{sPdef}),
\begin{align}  \label{P4} \begin{split}
& -T_1(X,X) T_1(Y,Y) + T_1(X,Y)^2 + \sigma_2(A) \big[ |X|^2 |Y|^2 - \langle X,Y
\rangle^2 \big] \\
&=  -\lambda_1 \lambda_2 z_{12}^2 - \lambda_1 \lambda_3 z_{13}^2 - \cdots -
\lambda_3 \lambda_4 z_{34}^2  \\
& \quad +\frac{1}{3} \big( -\lambda_1^2  - \cdots - \lambda_4^2  + \lambda_1
\lambda_2 + \lambda_1 \lambda_3 + \cdots \lambda_3 \lambda_4 \big) \big[
z_{12}^2 + z_{13}^2 + \cdots + z_{34}^2 \big] \\
&= \frac{1}{3} \big( -\lambda_1^2  - \cdots - \lambda_4^2  - 2 \lambda_1
\lambda_2 + \lambda_1 \lambda_3 + \cdots \lambda_3 \lambda_4 \big) z_{12}^2 \\
& \quad + \frac{1}{3} \big( -\lambda_1^2  - \cdots - \lambda_4^2  + \lambda_1
\lambda_2 - 2 \lambda_1 \lambda_3 + \lambda_1 \lambda_4 + \cdots \lambda_3
\lambda_4 \big) z_{13}^2  \\
& \quad + \cdots + \frac{1}{3} \big( -\lambda_1^2  - \cdots - \lambda_4^2  +
\lambda_1 \lambda_2 + \cdots + \lambda_2 \lambda_4 - 2 \lambda_3 \lambda_4 \big)
z_{34}^2.
\end{split}
\end{align}
We claim that the coefficients of the $z_{ij}^2$-terms are all non-positive.  To
see this, consider the first one:
\begin{align} \label{coe} \begin{split}
& -\lambda_1^2  - \cdots - \lambda_4^2  - 2 \lambda_1 \lambda_2 + \lambda_1
\lambda_3 + \lambda_1 \lambda_4 + \lambda_2 \lambda_3 + \lambda_2 \lambda_4 +
\lambda_3 \lambda_4 \\
&= -(\lambda_1 + \lambda_2)^2 - \lambda_3^2 - \lambda_4^2 + (\lambda_1 +
\lambda_2)\lambda_3 + (\lambda_1 + \lambda_2)\lambda_4 + \lambda_3 \lambda_4 \\
&\leq -(\lambda_1 + \lambda_2)^2 - \lambda_3^2 - \lambda_4^2 +
\frac{1}{2}(\lambda_1 + \lambda_2)^2 + \frac{1}{2}\lambda_3^2  +
\frac{1}{2}(\lambda_1 + \lambda_2)^2 + \frac{1}{2}\lambda_4^2 \\
& \quad  + \frac{1}{2}\lambda_3^2 +  \frac{1}{2}\lambda_4^2 \\
&= 0.
\end{split}
\end{align}
\end{proof}

Finally we prove the required curvature inequality, which is a more precise
statement of Theorem \ref{npcthm}.

\begin{thm} \label{n4curvature} Let $(M^4, g)$ be a compact Riemannian manifold
such that $A_g \in
\gG_2^+$.  Given $u \in \gG_2^+$ and $\phi,\psi \in T_u \gG_2^+$ we have
\begin{align*}
 K(\phi,\psi) =&\ \int \frac{1}{\sigma_2(g_u^{-1} A_u)}\Big\{ - \big\langle T_1(g_u^{-1} A_u),
\nabla
 \phi \otimes \nabla \phi
\big\rangle \big\langle T_1(g_u^{-1} A_u), \nabla  \psi
\otimes \nabla \psi \big\rangle \\
 & \quad + \big\langle T_1(g_u^{-1} A_u), \nabla  \phi \otimes
\nabla \psi \big\rangle^2 + \sigma_2(g_u^{-1} A_u) \big\vert \nabla \phi \vert^2 \vert
\nabla \psi \vert^2 - \sigma_2(g_u^{-1} A_u) \big\vert \langle
\nabla  \phi , \nabla \psi
\rangle \vert^2  \Big\} dV_u\\
\leq&\ 0,
\end{align*}
where the inner products are with respect to $g_u$
\end{thm}

\begin{proof} Let $u(s,t)$ be a 2-parameter family of conformal factors, and
$\alpha = \alpha(s,t) \in T_{u(s,t)}\mathcal{C}^{+}$.
Using the formula for the directional derivative in (\ref{connectiondef}), we have
\begin{align} \begin{split} \label{secD1}
\frac{D}{\partial s} \frac{D}{\partial t} \alpha &= \frac{\partial}{\partial s}
\big( \frac{D}{\partial t} \alpha \big) - \frac{1}{\sigma_2(g_u^{-1} A_u)} \big\langle
T_1(g_u^{-1} A_u), \nabla ( \frac{D}{\partial t}\alpha ) \otimes \nabla (\frac{\partial
u}{\partial s})  \big\rangle_u  \\
&= \frac{\partial}{\partial s} \Big\{ \frac{\partial \alpha}{\partial t} -
\frac{1}{\sigma_2(g_u^{-1} A_u)} \big\langle T_1(g_u^{-1} A_u), \nabla \alpha  \otimes \nabla(
\frac{\partial u}{\partial t}) \big\rangle_u \Big\}  \\
& \quad \quad - \frac{1}{\sigma_2(g_u^{-1} A_u)}\big\langle T_1(g_u^{-1} A_u), \nabla(
\frac{D}{\partial t}\alpha ) \otimes \nabla ( \frac{\partial u}{\partial s})
\big\rangle_u \\
&= \frac{\partial^2 \alpha}{\partial s \partial t} + \frac{1}{\sigma_2(g_u^{-1} A_u)^2}
\big\langle T_1(g_u^{-1} A_u), \nabla^2 (\frac{\partial u}{\partial s}) \big\rangle_u
  \big\langle T_1(g_u^{-1} A_u) ,
\nabla \alpha \otimes \nabla (\frac{\partial u}{\partial t}) \big\rangle_u \\
& \quad \quad - \frac{1}{\sigma_2(g_u^{-1} A_u)} \big\langle \frac{\partial}{\partial s}
T_1(g_u^{-1} A_u), \nabla \alpha \otimes \nabla (\frac{\partial u}{\partial t})
\big\rangle_u \\
&\quad \quad - \frac{1}{\sigma_2(g_u^{-1} A_u)} \big\langle  T_1(g_u^{-1} A_u), \nabla(
\frac{\partial \alpha}{\partial s}) \otimes \nabla(\frac{\partial u}{\partial
t}) + \nabla \alpha \otimes \nabla( \frac{\partial^2 u}{\partial s \partial t})
\big\rangle_u \\
&\quad \quad - \frac{1}{\sigma_2(g_u^{-1} A_u)} \big\langle  T_1(g_u^{-1} A_u), \nabla(
\frac{D}{\partial t}\alpha) \otimes \nabla( \frac{\partial u}{\partial s})
\big\rangle_u.
\end{split}
\end{align}
In the above, we have used the fact that the inner product on symmetric $2$-tensors satisfies
\begin{align*}
\frac{\partial}{\partial s}\langle \ \cdot\ , \  \cdot \ \rangle_u = 4 \frac{\partial u}{\partial s} \langle \ \cdot\ , \ \cdot \ \rangle_u.
\end{align*}
For the last term in (\ref{secD1}),
\begin{align*} \begin{split}
&- \frac{1}{\sigma_2(g_u^{-1} A_u)} \big\langle   T_1(g_u^{-1} A_u), \nabla( \frac{D}{\partial
t}\alpha) \otimes \nabla( \frac{\partial u}{\partial s}) \big\rangle_u = \\
&= - \frac{1}{\sigma_2(g_u^{-1} A_u)} \Big\langle  T_1(g_u^{-1} A_u), \nabla \Big\{
\frac{\partial \alpha}{\partial t} - \frac{1}{\sigma_2(g_u^{-1} A_u)} \big\langle
T_1(g_u^{-1} A_u), \nabla  \alpha \otimes \nabla( \frac{\partial u}{\partial t})
\big\rangle_u \Big\} \otimes \nabla( \frac{\partial u}{\partial s}) \Big\rangle_u \\
&= - \frac{1}{\sigma_2(g_u^{-1} A_u)} \Big\langle  T_1(g_u^{-1} A_u), \nabla ( \frac{\partial
\alpha}{\partial t} ) \otimes \nabla( \frac{\partial u}{\partial s}) \Big\rangle_u
\\
& \quad \quad + \frac{1}{\sigma_2(g_u^{-1} A_u)} \Big\langle  T_1(g_u^{-1} A_u), \nabla \Big\{
\frac{1}{\sigma_2(g_u^{-1} A_u)} \big\langle T_1(g_u^{-1} A_u), \nabla  \alpha \otimes \nabla(
\frac{\partial u}{\partial t}) \big\rangle_u \Big\} \otimes \nabla( \frac{\partial
u}{\partial s}) \Big\rangle_u.
\end{split}
\end{align*}
By (\ref{Tk}) and (\ref{Ap}),
\begin{align*}  \begin{split}
\frac{\partial}{\partial s}T_1(g_u^{-1} A_u) &= \frac{\partial}{\partial s} \big\{ -A_u + \sigma_1(g_u^{-1} A_u) g_u  \big\} \\
&= - \nabla_u^2( \frac{\partial u}{\partial s})
+ \Delta_u ( \frac{\partial u}{\partial s}) g_u.
\end{split}
\end{align*}
Substituting this into (\ref{secD1}), we get
\begin{align*} \begin{split}
\frac{D}{\partial s} \frac{D}{\partial t} \alpha  &= \frac{\partial^2
\alpha}{\partial s \partial t} + \frac{1}{\sigma_2(g_u^{-1} A_u)} \Bigg\{
\frac{1}{\sigma_2(g_u^{-1} A_u)}  \big\langle T_1(g_u^{-1} A_u), \nabla^2 (\frac{\partial
u}{\partial s}) \big\rangle_u \big\langle T_1(g_u^{-1} A_u) , \nabla \alpha \otimes \nabla
(\frac{\partial u}{\partial t}) \big\rangle_u \\
& \quad \quad +  \Big\langle
\nabla_u^2( \frac{\partial u}{\partial s}) - \Delta_u ( \frac{\partial u}{\partial
s}) g_u, \nabla \alpha \otimes \nabla
(\frac{\partial u}{\partial t}) \Big\rangle_u \\
&\quad \quad -  \big\langle  T_1(g_u^{-1} A_u), \nabla( \frac{\partial \alpha}{\partial
s}) \otimes \nabla(\frac{\partial u}{\partial t}) \big\rangle_u  + \big\langle
T_1(g_u^{-1} A_u),\nabla \alpha \otimes \nabla( \frac{\partial^2 u}{\partial s \partial
t}) \big\rangle_u \\
&\quad \quad -   \Big\langle  T_1(g_u^{-1} A_u), \nabla ( \frac{\partial \alpha}{\partial
t} ) \otimes \nabla( \frac{\partial u}{\partial s}) \Big\rangle_u \\
& \quad \quad +   \Big\langle  T_1(g_u^{-1} A_u), \nabla \Big\{  \frac{1}{\sigma_2(g_u^{-1} A_u)}
\big\langle T_1(g_u^{-1} A_u), \nabla  \alpha \otimes \nabla( \frac{\partial u}{\partial
t}) \big\rangle_u \Big\} \otimes \nabla( \frac{\partial u}{\partial s})
\Big\rangle_u \Bigg\}.
\end{split}
\end{align*}
Next, we rearrange the terms into two groups: those symmetric in $s,t$, and those that
are not:
\begin{align*} \begin{split}
\frac{D}{\partial s} \frac{D}{\partial t} \alpha  &= \frac{\partial^2
\alpha}{\partial s \partial t} + \frac{1}{\sigma_2(g_u^{-1} A_u)} \Bigg\{ -  \big\langle
T_1(g_u^{-1} A_u), \nabla( \frac{\partial \alpha}{\partial s}) \otimes
\nabla(\frac{\partial u}{\partial t}) \big\rangle_u -   \Big\langle  T_1(g_u^{-1} A_u),
\nabla ( \frac{\partial \alpha}{\partial t} ) \otimes \nabla( \frac{\partial
u}{\partial s}) \Big\rangle_u \\
&\quad \quad   + \big\langle  T_1(g_u^{-1} A_u),\nabla \alpha \otimes \nabla(
\frac{\partial^2 u}{\partial s \partial t}) \big\rangle_u \Bigg\} \\
& \quad \quad  + \frac{1}{\sigma_2(g_u^{-1} A_u)} \Bigg\{ \frac{1}{\sigma_2(g_u^{-1} A_u)}
\big\langle T_1(g_u^{-1} A_u), \nabla^2 (\frac{\partial u}{\partial s}) \big\rangle_u
\big\langle T_1(g_u^{-1} A_u) , \nabla \alpha \otimes \nabla (\frac{\partial u}{\partial
t}) \big\rangle_u \\
& \quad \quad  +  \Big\langle \nabla_u^2( \frac{\partial u}{\partial s}) - \Delta_u
( \frac{\partial u}{\partial s}) g_u , \nabla \alpha \otimes \nabla
(\frac{\partial u}{\partial t}) \Big\rangle_u \\
& \quad \quad +   \Big\langle  T_1(g_u^{-1} A_u), \nabla \Big\{  \frac{1}{\sigma_2(g_u^{-1} A_u)}
\big\langle T_1(g_u^{-1} A_u), \nabla  \alpha \otimes \nabla( \frac{\partial u}{\partial
t}) \big\rangle_u \Big\} \otimes \nabla( \frac{\partial u}{\partial s})
\Big\rangle_u \Bigg\}.
\end{split}
\end{align*}
Therefore,
\begin{align} \begin{split} \label{secD5}
\Big( \frac{D}{\partial s} \frac{D}{\partial t}  - \frac{D}{\partial t}
\frac{D}{\partial s} \Big) \alpha  &= \frac{1}{\sigma_2( g_u^{-1} A_u)} \Bigg\{
\frac{1}{\sigma_2( g_u^{-1} A_u)}  \big\langle T_1( g_u^{-1} A_u), \nabla^2 (\frac{\partial
u}{\partial s}) \big\rangle_u \big\langle T_1(g_u^{-1} A_u) , \nabla \alpha \otimes \nabla
(\frac{\partial u}{\partial t}) \big\rangle_u \\
& \hskip-.5in - \frac{1}{\sigma_2(g_u^{-1} A_u)}  \big\langle T_1(g_u^{-1} A_u), \nabla^2
(\frac{\partial u}{\partial t}) \big\rangle_u \big\langle T_1(g_u^{-1} A_u) , \nabla \alpha
\otimes \nabla (\frac{\partial u}{\partial s}) \big\rangle_u \\
& \hskip-1in +  \Big\langle \nabla_u^2( \frac{\partial u}{\partial s}) - \Delta_u (
\frac{\partial u}{\partial s}) g_u , \nabla \alpha \otimes \nabla (\frac{\partial
u}{\partial t}) \Big\rangle_u - \Big\langle \nabla_u^2( \frac{\partial u}{\partial
t}) - \Delta_u ( \frac{\partial u}{\partial t}) g_u , \nabla \alpha \otimes \nabla
(\frac{\partial u}{\partial s}) \Big\rangle_u \\
& \hskip-.5in +   \Big\langle  T_1(g_u^{-1} A_u), \nabla \Big\{  \frac{1}{\sigma_2(g_u^{-1} A_u)}
\big\langle T_1(g_u^{-1} A_u), \nabla  \alpha \otimes \nabla( \frac{\partial u}{\partial
t}) \big\rangle_u \Big\} \otimes \nabla( \frac{\partial u}{\partial s})
\Big\rangle_u \\
& \hskip-.5in - \Big\langle  T_1(g_u^{-1} A_u), \nabla \Big\{  \frac{1}{\sigma_2(g_u^{-1} A_u)}
\big\langle T_1(g_u^{-1} A_u), \nabla  \alpha \otimes \nabla( \frac{\partial u}{\partial
s}) \big\rangle_u \Big\} \otimes \nabla( \frac{\partial u}{\partial t})
\Big\rangle_u \Bigg\}.
\end{split}
\end{align}

To compute the sectional curvature of the plane spanned by $\{ \frac{\partial
u}{\partial s}, \frac{\partial u}{\partial t} \}$, we take $\alpha =
\frac{\partial u}{\partial t}$ in the formula above, then take the inner product
with $\frac{\partial u}{\partial s}$:
\begin{align*} \begin{split}
 \Big\langle \Big( \frac{D}{\partial s} \frac{D}{\partial t} -
&\frac{D}{\partial t} \frac{D}{\partial s} \Big) \frac{\partial u}{\partial t},
\frac{\partial u}{\partial s} \Big\rangle_u  = \\ &
\hskip-.6in \int \Big\{ \frac{1}{\sigma_2(g_u^{-1} A_u)}  \big\langle T_1(g_u^{-1} A_u), \nabla_u^2
(\frac{\partial u}{\partial s}) \big\rangle_u \big\langle T_1(g_u^{-1} A_u) , \nabla
(\frac{\partial u}{\partial t}) \otimes \nabla (\frac{\partial u}{\partial t})
\big\rangle_u \frac{\partial u}{\partial s}\\
& \hskip-.5in - \frac{1}{\sigma_2(g_u^{-1} A_u)}  \big\langle T_1(g_u^{-1} A_u), \nabla_u^2
(\frac{\partial u}{\partial t}) \big\rangle_u \big\langle T_1(g_u^{-1} A_u) , \nabla
(\frac{\partial u}{\partial t}) \otimes \nabla (\frac{\partial u}{\partial s})
\big\rangle_u \frac{\partial u}{\partial s}\\
& \hskip-.5in +  \Big\langle \nabla_u^2 (\frac{\partial u}{\partial s})  - \Delta_u
(\frac{\partial u}{\partial s}) g_u , \nabla (\frac{\partial u}{\partial t})
\otimes \nabla (\frac{\partial u}{\partial t}) \Big\rangle_u \frac{\partial
u}{\partial s} - \Big\langle \nabla_u^2( \frac{\partial u}{\partial t}) - \Delta_u (
\frac{\partial u}{\partial t}) g_u , \nabla (\frac{\partial u}{\partial t})
\otimes \nabla (\frac{\partial u}{\partial s}) \Big\rangle_u \frac{\partial
u}{\partial s} \\
& \hskip-.5in +   \Big\langle  T_1(g_u^{-1} A_u), \nabla \Big\{  \frac{1}{\sigma_2(g_u^{-1} A_u)}
\big\langle T_1(g_u^{-1} A_u), \nabla  (\frac{\partial u}{\partial t}) \otimes \nabla(
\frac{\partial u}{\partial t}) \big\rangle_u \Big\} \otimes \nabla (\frac{\partial
u}{\partial s}) \Big\rangle_u \frac{\partial u}{\partial s} \\
& \hskip-.5in - \Big\langle  T_1(g_u^{-1} A_u), \nabla \Big\{  \frac{1}{\sigma_2(g_u^{-1} A_u)}
\big\langle T_1(g_u^{-1} A_u), \nabla  (\frac{\partial u}{\partial t}) \otimes \nabla
(\frac{\partial u}{\partial s}) \big\rangle_u \Big\} \otimes \nabla(
\frac{\partial u}{\partial t}) \Big\rangle_u \frac{\partial u}{\partial s} \Big\}
dV_u.
\end{split}
\end{align*}
Consider the last two lines above.  Integrating by parts and using the fact that
$T_1(g_u^{-1}A_u)$ is divergence-free, we get
\begin{align*} \begin{split}
& \int \Big\{ \Big\langle  T_1(g_u^{-1}A_u), \nabla \Big\{  \frac{1}{\sigma_2(g_u^{-1}A_u)}
\big\langle T_1(g_u^{-1}A_u), \nabla  (\frac{\partial u}{\partial t}) \otimes \nabla(
\frac{\partial u}{\partial t}) \big\rangle_u \Big\} \otimes \nabla (\frac{\partial
u}{\partial s}) \Big\rangle_u \frac{\partial u}{\partial s} \\
& \hskip.25in - \Big\langle  T_1(g_u^{-1}A_u), \nabla \Big\{  \frac{1}{\sigma_2(g_u^{-1}A_u)}
\big\langle T_1(g_u^{-1}A_u), \nabla  (\frac{\partial u}{\partial t}) \otimes \nabla
(\frac{\partial u}{\partial s}) \big\rangle_u \Big\} \otimes \nabla(
\frac{\partial u}{\partial t}) \Big\rangle_u \frac{\partial u}{\partial s} \Big\}
dV_u \\
&= \int \Big\{ - \frac{1}{\sigma_2(g_u^{-1}A_u)}  \big\langle T_1(g_u^{-1}A_u), \nabla_u^2
(\frac{\partial u}{\partial s}) \big\rangle_u \big\langle T_1(g_u^{-1}A_u) , \nabla
(\frac{\partial u}{\partial t}) \otimes \nabla (\frac{\partial u}{\partial t})
\big\rangle_u \frac{\partial u}{\partial s}\\
& \hskip.5in - \frac{1}{\sigma_2(g_u^{-1}A_u)} \big\langle T_1(g_u^{-1}A_u), \nabla
(\frac{\partial u}{\partial s}) \otimes \nabla( \frac{\partial u}{\partial s})
\big\rangle_u \big\langle T_1(g_u^{-1}A_u), \nabla  (\frac{\partial u}{\partial t})
\otimes \nabla( \frac{\partial u}{\partial t}) \big\rangle_u   \\
& \hskip.25in + \frac{1}{\sigma_2(g_u^{-1}A_u)}  \big\langle T_1(g_u^{-1}A_u), \nabla_u^2
(\frac{\partial u}{\partial t}) \big\rangle_u \big\langle T_1(g_u^{-1}A_u) , \nabla
(\frac{\partial u}{\partial t}) \otimes \nabla (\frac{\partial u}{\partial s})
\big\rangle_u \frac{\partial u}{\partial s} \\
& \hskip.5in + \frac{1}{\sigma_2(g_u^{-1}A_u)} \big\langle T_1(g_u^{-1}A_u), \nabla
(\frac{\partial u}{\partial s}) \otimes \nabla( \frac{\partial u}{\partial t})
\big\rangle_u^2 \Big\} dV_u.
\end{split}
\end{align*}
Substituting this into (\ref{secD5}) we find that the the first two lines there
cancel, and we arrive at
\begin{align*} \begin{split}
 \Big\langle \Big( \frac{D}{\partial s} \frac{D}{\partial t} -
&\frac{D}{\partial t} \frac{D}{\partial s} \Big) \frac{\partial u}{\partial s},
\frac{\partial u}{\partial t} \Big\rangle_u  = \\ &
\hskip-.6in \int \Big\{ \Big\langle \nabla_u^2 (\frac{\partial u}{\partial s})  -
\Delta_u (\frac{\partial u}{\partial s}) g_u , \nabla (\frac{\partial u}{\partial
t}) \otimes \nabla (\frac{\partial u}{\partial t}) \Big\rangle_u  \\
& \hskip-.3in - \Big\langle \nabla_u^2( \frac{\partial u}{\partial t}) - \Delta_u (
\frac{\partial u}{\partial t}) g_u , \nabla (\frac{\partial u}{\partial t})
\otimes \nabla (\frac{\partial u}{\partial s}) \Big\rangle_u  \Big\}
\frac{\partial u}{\partial s} dV_u \\
 & \hskip-.7in + \int \frac{1}{\sigma_2(g_u^{-1}A_u)} \Big\{ - \big\langle T_1(g_u^{-1}A_u),
\nabla  (\frac{\partial u}{\partial s}) \otimes \nabla( \frac{\partial
u}{\partial s}) \big\rangle_u \big\langle T_1(g_u^{-1}A_u), \nabla  (\frac{\partial
u}{\partial t}) \otimes \nabla( \frac{\partial u}{\partial t}) \big\rangle_u \\
 & \hskip-.3in + \big\langle T_1(g_u^{-1}A_u), \nabla  (\frac{\partial u}{\partial s})
\otimes \nabla( \frac{\partial u}{\partial t}) \big\rangle_u^2 \Big\} dV_u.
\end{split}
\end{align*}
From Lemmas \ref{curvlemma1} and \ref{sgLemma} we conclude
\begin{align*}
 \Big\langle \Big( \frac{D}{\partial s} \frac{D}{\partial t} & -
\frac{D}{\partial t} \frac{D}{\partial s} \Big) \frac{\partial u}{\partial s},
\frac{\partial u}{\partial t} \Big\rangle_u \\
 =& \int \frac{1}{\sigma_2(g_u^{-1}A_u)}\Big\{ - \big\langle T_1(g_u^{-1}A_u), \nabla
 (\frac{\partial u}{\partial s}) \otimes \nabla( \frac{\partial u}{\partial s})
\big\rangle_u \big\langle T_1(g_u^{-1}A_u), \nabla  (\frac{\partial u}{\partial t})
\otimes \nabla( \frac{\partial u}{\partial t}) \big\rangle_u \\
 & \quad + \big\langle T_1(g_u^{-1}A_u), \nabla  (\frac{\partial u}{\partial s}) \otimes
\nabla( \frac{\partial u}{\partial t}) \big\rangle_u^2\\
& \quad + \sigma_2(g_u^{-1}A_u) \big\vert \nabla \frac{\partial u}{\partial s} \big\vert_u^2
\big\vert \nabla \frac{\partial u}{\partial t} \big\vert_u^2 - \sigma_2(g_u^{-1}A_u) \big\vert
\langle \nabla  \frac{\partial u}{\partial s} , \nabla \frac{\partial
u}{\partial t} \rangle_u \big\vert_u \Big\} dV_u\\
\leq&\ 0,
\end{align*}
as required.
\end{proof}

\begin{rmk} The Mabuchi metric turns out to be formally an
infinite dimensional symmetric space, evidenced by the sectional curvatures
admitting an interpretation as the square norm of the Poisson bracket of the two
tangent vector functions.  There does not seem to be such an interpretation in
this setting.
\end{rmk}

\subsection{Formal metric space structure} \label{MSS}

In this subsection we observe some fundamental properties of lengths of curves
and distances in the $\gs_2$-metric.

\begin{defn}  Given a path $u : [a,b] \to \C$, the
\emph{length of $u$} is
\begin{align*}
\mathcal L(u) := \int_a^b \IP{\ga,\gb}^{\frac{1}{2}} dt = \int_a^b \left[ \int_M
\left(
\frac{\del u}{\del t} \right)^2 \gs_2 (g_u^{-1} A_u) dV_u \right]^{\frac{1}{2}}dt.
\end{align*}
A curve is a \emph{geodesic} if it is a critical point for $L$.
\end{defn}

\begin{lemma}  A curve $u_t \in \C$ is a geodesic if
and only if
\begin{align} \label{pathgeod}
u_{tt} - \frac{1}{\gs_2(A_u)} \IP{T_1(A_u), \N u_t \otimes \N u_t} = 0.
\end{align}
\begin{proof} Formally, by Lemma \ref{metcomp} the connection is indeed the
Riemannian connection and so a curve is a
geodesic if and only if
\begin{align*}
0 =&\ \frac{D}{\del t} \frac{\del u}{\del t} = u_{tt} - \frac{1}{\gs_2(A_u)}
\IP{T_1(A_u), \N u_t \otimes \N u_t}.
\end{align*}
This can also be derived by directly taking the first variation of the length
functional.
\end{proof}
\end{lemma}

\begin{rmk}  We observe a canonical isometric
splitting of $T_u\C$ with respect to
the $\sigma_k$ metric.  In particular, the real line $\mathbb R \subset T_u \C$
given by constant functions is orthogonal to
\begin{align*}
T^0_u \C := \left\{ \ga\ |\ \int_M \ga \gs_2(g_u^{-1} A_u) dV_u = 0 \right\}.
\end{align*}
\end{rmk}

\noindent In the next lemma we show two basic properties of geodesics, namely
that they preserve this isometric splitting, and are automatically parameterized
with constant speed.

\begin{lemma}  \label{metricsplit} Let $u_t$ be a
solution to (\ref{pathgeod}).  Then
\begin{align*}
\frac{d}{dt} \int_M u_t \gs_2(g_u^{-1}A_u) dV_u =&\ 0,\\
\frac{d}{dt} \int_M u_t^2 \gs_2(g_u^{-1}A_u) dV_u =&\ 0.
\end{align*}
\begin{proof} Differentiating and using (\ref{densep}),
\begin{align*}
\frac{d}{dt} \int_M u_t \gs_2(g_u^{-1}A_u) dV_u =&\ \int_M \left( u_{tt}
\gs_2(g_u^{-1}A_u) +
u_t \IP{T_1(g_u^{-1}A_u), \N^2 u_t}_u \right) dV_u\\
=&\ \int_M \left( u_{tt} - \gs_2(g_u^{-1} A_u)^{-1} \IP{T_1(g_u^{-1}A_u), \N u_t
\otimes \N
u_t}_u \right) \gs_2(g_u^{-1}A_u) dV_u\\
=&\ 0.
\end{align*}
Next
\begin{align*}
\frac{d}{dt} \int_M u_t^2 \gs_2(g_u^{-1}A_u) dV_u =&\ \int_M \left[ 2 \gs_2(g_u^{-1}A_u)
u_{tt} u_t + u_t^2 \IP{T_1(g_u^{-1}A_u), \N^2 u_t}_u \right] dV_u\\
=&\ 2 \int_M \gs_2(g_u^{-1}A_u) u_t \left[ u_{tt} -
\frac{1}{\gs_2(g_u^{-1}A_u)}\IP{T_1(g_u^{-1}A_u), \N
u_t \otimes \N u_t}_u \right] dV_u\\
=&\ 0.
\end{align*}
\end{proof}
\end{lemma}

\begin{prop} \label{distancenondegeneracy} Given
$u_0, u_1 \in C^{\infty}(M)$ and $u_t : [0,1] \to \C$ a
geodesic, one has
\begin{align*}
\LL(u) \geq&\ \gs^{-\frac{1}{2}} \max \left\{ \int_{u_1  > u_0} (u_1 - u_0)
\gs_2(g_{u_1}^{-1} A_{u_1})
dV_{u_1}, \int_{u_0 > u_1} (u_0 - u_1) \gs_2(g_{u_0}^{-1} A_{u_0}) dV_{u_0} \right\}.
\end{align*}
\begin{proof} Observe that the geodesic equation implies $u_{tt}\geq 0$, and
so we obtain the pointwise inequality
\begin{align*}
u_t(0) \leq u_1 - u_0 \leq u_t(1).
\end{align*}
Thus using H\"older's inequality we have
\begin{align*}
E(1) =&\ \left( \int_M u_t^2 \gs_2(g_{u_1}^{-1} A_{u_1}) dV_{u_1} \right)^{\frac{1}{2}}\\
\geq&\ \gs^{-\frac{1}{2}} \int_M \brs{u_t} \gs_2(g_{u_1}^{-1} A_{u_1}) dV_{u_1}\\
\geq&\ \gs^{-\frac{1}{2}} \int_{u_1  > u_0} (u_1 - u_0) \gs_2(g_{u_1}^{-1} A_{u_1}) dV_{u_1}.
\end{align*}
A similar argument yields
\begin{align*}
E(0) \geq \gs^{-\frac{1}{2}} \int_{u_0 > u_1} (u_0 - u_1) \gs_2(g_{u_0}^{-1} A_{u_0})
dV_{u_0}.
\end{align*}
Since geodesics are automatically constant speed by Lemma \ref{metricsplit},
the result follows.
\end{proof}
\end{prop}

\subsection{Geodesics and the conformal group of the sphere}  As in the two-dimensional case, we will show that the $1$-parameter family of transformations that generate the conformal
group of the sphere are geodesics.  In anticipation of our forthcoming article on the higher-dimensional case we will prove a more general result.

Let $(S^n,g_0)$ denote the round sphere.  Using stereographic projection $\sigma : S^n \setminus \{ N \} \rightarrow \mathbb{R}^n$,
where $N \in S^n$ denotes the north pole, one can define a one-parameter of conformal maps of $S^n$ by conjugating the dilation map $\delta_{\alpha} : x \mapsto \alpha^{-1} x$ on $\mathbb{R}^n$ with $\sigma$:
\begin{align*}
\varphi_{\alpha} = \sigma^{-1} \circ \delta_{\alpha} \circ \sigma : S^n \rightarrow S^n.
\end{align*}
Taking $\alpha(t) = e^{\lambda t}$, where $\lambda$ is a fixed real number, we can define the path of conformal metrics
\begin{align} \label{gtdef}
g(t) = e^{-2u} g_{0} = \phi_{\ga}^* g_0 = \left[ \frac{2 \ga(t) }{(1 + \xi) + \alpha(t)^2 (1 - \xi)} \right]^{2},
 \end{align}
where $\xi = x^{n+1}$ is the $(n+1)$-coordinate function; i.e., $N = (0,\dots,0,1)$  (see \cite{LeeParker}).

\begin{prop}   \label{conformalpathsgeodesics}  If $k = n/2$, the path $g(t) = e^{-2u(t)}g_0 :  (-\infty , + \infty) \rightarrow \mathcal{C}^{+}$ satisfies
\begin{align} \label{geok}
u_{tt} - \frac{1}{\gs_{k}(A_u)} \IP{T_{k-1}(A_u), \N u_t \otimes \N u_t} = 0.
\end{align}
In particular, when $n=4$ this path defines a geodesic.
\end{prop}

 \begin{proof}
 By (\ref{gtdef}),
 \begin{align*}
  u = u(t) = - \log 2 \ga + \log \left[ (1 + \xi) + \ga^2 (1 - \xi) \right].
 \end{align*}
This yields
\begin{align*}
u_t = - \frac{\dot{\ga}}{\ga} +  \frac{2 \ga \dot{\ga}(1 - \xi)}{(1 + \xi) +
\ga^2(1 - \xi)}
\end{align*}
and hence
\begin{align*}
u_{tt} =&\ - \frac{\ga_{tt}}{\ga} + \left( \frac{\ga_t}{\ga} \right)^2 +
\frac{[(1 + \xi) + \ga^2(1 - \xi)] (2 \ga \ga_{tt} + 2 \ga_t^2)(1-\xi) - 4
\ga^2 \ga_t^2 (1-\xi)^2}{\left[ (1 + \xi) + \ga^2(1-\xi) \right]^2}.
\end{align*}
Since $\alpha(t) = e^{\lambda t}$, we have
\begin{align} \label{udd}
u_{tt} = 4 \gl^2 e^{2 \gl t} \frac{ 1-\xi^2}{\left[ (1 + \xi) +
\ga^2(1-\xi) \right]^2}.
\end{align}
Also,
\begin{align*}
 \N u_t =&\ - \frac{2 \ga \ga_t \N \xi}{(1 + \xi) + \ga^2(1-\xi)} -
\frac{2 \ga \ga_t(1 - \xi)}{ \left[ (1 + \xi) + \ga^2(1-\xi) \right]^2}
\left[ (1-\ga^2) \N \xi \right]\\
=&\ \frac{- 2 \ga \ga_t \N \xi}{\left[ (1 + \xi) + \ga^2(1-\xi) \right]^2}
\left[ (1+ \xi) + \ga^2(1-\xi) + (1-\xi)(1 - \ga^2) \right]\\
=&\ \frac{ - 4\ga \ga_t \N \xi}{\left[ (1 + \xi) + \ga^2(1-\xi) \right]^2} \\
=&\ \frac{ - 4\lambda e^{2\lambda t} \N \xi}{\left[ (1 + \xi) + \ga^2(1-\xi) \right]^2}.
\end{align*}

On $S^n$, the Schouten tensor is a multiple of the identity; in fact $A(g_0) = \frac{1}{2}g_0$. Therefore, using standard identities
for the symmetric functions,
\begin{align*}
\dfrac{1}{\sigma_k(g(t)^{-1} A_{g(t)})} T_1(g(t)^{-1} A_{g(t)}) = \dfrac{2k}{n} g(t) = g(t),
\end{align*}
since $k = n/2$.  Thus
\begin{align} \label{rhsg}
\dfrac{1}{\sigma_k(g(t)^{-1} A_{g(t)})} \langle T_{k-1}(g(t)^{-1} A_{g(t)}) , \N u_t \otimes \N u_t \rangle = 4 \gl^2 e^{2 \gl t} \frac{ |\N \xi|^2 }{\left[ (1 + \xi) +
\ga^2(1-\xi) \right]^2}.
\end{align}
Since $|\N \xi|^2 = 1 - \xi^2$, comparing (\ref{udd}) and (\ref{rhsg}) we see that $u$ satisfies (\ref{geok}).
\end{proof}

\begin{rmk} We do not expect conformal vector fields on general backgrounds to
generate nontrivial geodesics, and thus nonuniqueness of solutions.  It follows from a result of
Lelong-Ferrand/Obata \cite{LF, Obata} that if $(M^n, g)$ is not conformally
equivalent to the round sphere, then any conformal Killing field is a Killing
field for a conformally related metric.  Expressed with respect to this
background metric, pullback by a family of isometries will result in no change
on the level of conformal factors.
\end{rmk}

\subsection{\texorpdfstring{The $F$-functional and geodesic convexity}{The
F-functional and geodesic convexity}} \label{Fsubsec}

We now derive the geodesic convexity of the $F$-functional of Chang-Yang.
The crucial input is a sharp curvature-weighted Poincar\'e inequality due to
Andrews:

\begin{prop} \label{andrewsineq} (Andrews \cite{Andrews}, cf. \cite{CLN} pg.
517) Let $(M^n, g)$ be a closed
Riemannian manifold with
positive Ricci curvature.  Given $\phi \in C^{\infty}(M)$ such that $\int_M \phi
dV = 0$, then
\begin{align*}
\frac{n}{n-1} \int_M \phi^2 dV \leq&\ \int_M \left( \Ric^{-1} \right)^{ij} \N_i
\phi
\N_j \phi dV,
\end{align*}
with equality if and only if $\phi \equiv 0$ or $(M^n, g)$ is isometric to the
round sphere.
\end{prop}

The convexity of $F$ will follow from a weaker form of this
inequality:

\begin{cor} \label{n4andrews} Let $(M^4, g)$ be a closed
Riemannian manifold such that $A_g \in \gG_2^+$.  Given $\phi \in C^{\infty}(M)$
such that $\int_M \phi
dV = 0$, then
\begin{align*}
 \int_M \frac{1}{\gs_2(A_g)} T_1(A_g)^{ij} \N_i \phi \N_j \phi dV_g \geq 4
\int_M \phi^2 dV_g - \left( \frac{4}{\int_M dV_g} \right) \left( \int_M \phi
dV_g \right)^2,
\end{align*}
with equality if and only if $\phi \equiv 0$ or $(M^n, g)$ is isometric to the
round sphere.
\begin{proof} We assume $\int_M \phi dV_g = 0$.  By Andrews' Poincar\'e
inequality we have
\begin{align*}
 \frac{4}{3} \int_M \phi^2 dV_g \leq \int_M \left( \Ric^{-1} \right)^{ij} \N_i
\phi \N_j \phi dV_g.
\end{align*}
To show the claim it suffices to show that
\begin{align*}
 3 \Ric^{-1} (X,X) \leq \frac{1}{\gs_2(A)} T_1(X,X).
\end{align*}
Since $\Ric$ and $T_1(A)$ commute, it suffices to show that $\Ric \circ T_1 \geq
3 \gs_2(A) g$.  Since $\Ric = 2 A + \gs_1(A) g$, this is equivalent to
\begin{align*}
 - 2 A \circ A + \gs_1(A) A + \gs_1(A)^2 g \geq 3 \gs_2(A) g.
\end{align*}
Now let $Z = A - \frac{1}{4} \gs_1(A) g$, then we can rewrite this as
\begin{align*}
 - 2 Z^2 + \frac{9}{8} \gs_1(A)^2 g \geq 3 \gs_2 g.
\end{align*}
Now, a Lagrange multipler argument shows that
\begin{align*}
 Z \circ Z \leq \frac{3}{4} \brs{Z}^2 g.
\end{align*}
Thus
\begin{align*}
 -2 Z^2 + \frac{9}{8} \gs_1(A)^2 g \geq - \frac{3}{2} \brs{Z}^2 g + \frac{9}{8} \gs_1(A)^2 g
= 3 \gs_2(A) g.
\end{align*}
\end{proof}
\end{cor}

\begin{prop} \label{CYgeodconv} The functional $F$ is geodesically convex.

\begin{proof} It follows from \cite{ChangYangMoserVol} that for a path of conformal metrics $u = u(t)$,
\begin{align} \label{Fdot}
\frac{d}{dt} F[u] =&\ \int_M u_t\left[ - \gs_2(g_u^{-1} A_u) + \bar{\gs}
\right] dV_u.
\end{align}
Assuming the path is a geodesic, then differentiating again and using Lemma \ref{metricsplit} we have
\begin{align*}
\frac{d^2}{dt^2} F[u] =&\ \frac{d}{dt} \int_M u_t \left[ - \gs_2(g_u^{-1} A_u) +
\bar{\gs} \right] dV_u\\
=&\  \gs \frac{d}{dt} \int_M  u_t V_u^{-1} dV_u\\
=&\ \gs \int_M \left[u_{tt} V_u^{-1}  + V_u^{-2} u_t \left(
\int_M 4 u_t dV_u \right) - 4 V_u^{-1} u_t^2 \right] dV_u\\
=&\ \gs V_u^{-1} \left[ \int_M \frac{1}{\gs_2(g_u^{-1} A_u)} \IP{T_1 (g_u^{-1} A_u),
\N u_t \otimes \N u_t}_u dV_u \right.\\
&\ \left. \qquad - 4 \left( \int_M u_t^2 dV_u - V_u^{-1} \left( \int_M
u_t dV_u \right)^2 \right) \right]\\
\geq&\ 0,
\end{align*}
where the last line follows from Corollary \ref{n4andrews}.
\end{proof}
\end{prop}

\section{Estimates of the Geodesic Equation} \label{geeodestimates}

In this section we establish several fundamental properties of the geodesic
equation (\ref{pathgeod}).   Once again, for future reference we will consider a more general
equation which reduces to (\ref{pathgeod}) when $n = 4$ and $k=2$:
\begin{align*}
u_{tt} = \frac{1}{\gs_k(A_u)}\IP{T_{k-1}(A_u), \N u_t \otimes \N u_t}.
\end{align*}

To begin, we define a certain regularization of this equation.  In
particular let
\begin{align*}
\Phi (u) := u_{tt} \gs_k(A_u) - \IP{T_{k-1}(A_u), \N
u_t
\otimes \N u_t}.
\end{align*}
Furthermore, let
\begin{align*}
\Phi_{\ge}(u) = (1+\ge) u_{tt} \gs_k(A_u) - \IP{T_{k-1}(A_u), \N u_t \otimes \N
u_t}.
\end{align*}
We will fix two parameters $\ge,s$, and study a priori estimates for
\begin{align*}
 \Phi_{\ge}(u(\cdot,\cdot,s)) = s f.
\end{align*}

To obtain estimates though we will simply fix a function $f \in C^{\infty}(M
\times [0,1])$ and study the equation
\begin{align} \label{regularizedgeod}
\mathcal G_{f}^{\ge}(u) =  \Phi_{\ge}(u) - f = 0. \qquad (\star_{\ge,f}).
\end{align}
As remarked on above, in the setting of Mabuchi geodesics, as observed by Semmes
\cite{Semmes} if one complexifies the time direction the equation admits an
interpretation as a certain modification of the tensor $A$ will show up
naturally in the
linearized operator.  Let
\begin{align*}
E = E^{\ge}_u = (1+\ge) u_{tt} A_u - \N u_t \otimes \N u_t.
\end{align*}

\begin{prop} \label{geodMA} $u \in C^2$ satisfies
($\star_{\ge,f}$) if and only
if
\begin{align*}
\left[ (1+\ge) u_{tt} \right]^{1-k} \gs_k(E^{\ge}_u) = f.
\end{align*}
\begin{proof} Using Lemma \ref{LAlemma} and homogeneity properties of elementary
symmetric polynomials we compute
\begin{align*}
\gs_k(E^{\ge}_u) =&\ \gs_k ( (1+\ge) u_{tt} A_u - \N u_t \otimes \N u_t)\\
=&\ \gs_k( (1+\ge) u_{tt} A_u) - \IP{T_{k-1}( (1 +\ge) u_{tt}A_u), \N u_t
\otimes
\N u_t}\\
=&\ \left[ (1 + \ge) u_{tt} \right]^{k-1} \left[ (1 +\ge) u_{tt} \gs_k(A_u) -
\IP{T_{k-1}(A_u),\N u_t \otimes \N u_t} \right].
\end{align*}
The proposition follows.
\end{proof}
\end{prop}

We will say that a solution $u$ of ($\star_{\ge,f}$) is {\em admissible} if $E^{\ge}_u \in \Gamma_k^+$.  As we will see below, ($\star_{\ge,f}$) is elliptic
for admissible solutions.

\begin{lemma} \label{linearizedgeod} Let $u = u(s,\cdot) \in C^{\infty}(M \times [0,1])$ be
a one-parameter family
of smooth functions such that $\left. \frac{d}{ds} u(s,\cdot) \right|_{s=0} = v$.  Then
\begin{align*}
\left. \frac{d}{ds} u_{tt}^{1-k} \gs_k(E_{u(s,\cdot)})\right|_{s=0} =&\ \LL(v),
\end{align*}
where
\begin{gather} \label{linop}
\begin{split}
\LL(v)  =&\ (1+\ge)^{k-1} u_{tt}^{-1} f v_{tt} \\
&\ + u_{tt}^{1-k} \left<T_{k-1}(E_u^{\ge}), (1 + \ge) u_{tt} \left( \N^2 v + \N
v \otimes \N u + \N u
\otimes \N v - \IP{\N v,\N u} g \right) \right.\\
&\ \qquad \qquad \qquad \qquad \left. - \N v_t \otimes \N u_t - \N u_t \otimes
\N v_t + u_{tt}^{-1} v_{tt} \N u_t \otimes \N u_t \right>.
\end{split}
\end{gather}

\begin{proof} We compute
\begin{gather} \label{lin10}
\begin{split}
\frac{d}{ds}& u_{tt}^{1-k} \gs_k(E^{\ge}_{u_s})\\
=&\ (1-k) u_{tt}^{-k} \gs_k(E^{\ge}_u) v_{tt}
+ u_{tt}^{1-k} \IP{T_{k-1}(E^{\ge}_u), \frac{d}{ds}
E^{\ge}_u}\\
=&\ (1-k) u_{tt}^{-k} \gs_k(E^{\ge}_u) v_{tt}\\
&\ + u_{tt}^{1-k} \IP{T_{k-1}(E^{\ge}_u), (1+\ge)v_{tt}
A_u + (1 +\ge) u_{tt} \frac{d}{ds} A_u
- \N v_t \otimes \N u_t - \N u_t \otimes \N v_t}.
\end{split}
\end{gather}
The second term can be simplified using Lemma \ref{LAlemma} to
\begin{gather} \label{lin20}
\begin{split}
(1+\ge) u_{tt}^{1-k} & \IP{T_{k-1}(E^{\ge}_u), v_{tt} A_u}\\
=&\ v_{tt} (1+\ge) u_{tt}^{1-k} \left[
u_{tt}^{-1} (1+\ge)^{-1}
\IP{T_{k-1}(E^{\ge}_u), E^{\ge}_u + \N u_t \otimes \N u_t} \right]\\
=&\ v_{tt} u_{tt}^{-k} \left[ k \gs_k(E_u^{\ge}) + \IP{T_{k-1}(E_u), \N u_t
\otimes \N
u_t}
\right]\\
=&\ k v_{tt} u_{tt}^{-k} \gs_k(E^{\ge}_u) + v_{tt} u_{tt}^{-1} (1+\ge)^{k-1}\IP{
T_{k-1}(A_u), \N
u_t \otimes \N
u_t}\\
=&\ k v_{tt} u_{tt}^{-k} \gs_k(E^{\ge}_u) + v_{tt} \left[ (1+\ge)^k\gs_k(A_u) -
f (1+\ge)^{k-1} u_{tt}^{-1} \right]\\
=&\ v_{tt} \left[ u_{tt}^{-k} (k - 1) \gs_k(E_u^{\ge}) + (1+\ge)^k \gs_k(A_u)
\right].
\end{split}
\end{gather}
Hence the overall term involving $v_{tt}$ in (\ref{lin10}) is $v_{tt}
(1+\ge)^k\gs_k(A_u)$.  However we can furthermore express, again using the
geodesic equation and Lemma \ref{LAlemma}, that
\begin{align*}
(1+\ge)^k \gs_k(A_u) =&\ (1+\ge)^{k-1} u_{tt}^{-1} f + (1+\ge)^{k-1} u_{tt}^{-1}
\IP{T_{k-1}(A_u), \N u_t \otimes \N u_t}\\
=&\ (1+\ge)^{k-1} u_{tt}^{-1} f + u_{tt}^{-k} \IP{T_{k-1}(E_u^{\ge}), \N u_t
\otimes \N u_t}.
\end{align*}
Likewise we simplify the third term of (\ref{lin10}) as
\begin{align*}
 (u_{tt}+\ge)^{1-k}  \IP{T_{k-1}(E_u), (1 + \ge) u_{tt} \left( \N^2 v + \N v
\otimes \N u + \N u
\otimes \N v - \IP{\N v,\N u} g \right) }.
\end{align*}
Collecting these calculations yields the result.
\end{proof}
\end{lemma}

\begin{lemma} \label{geodesicellipticity} Given $f \geq 0$, equation
$(\star_{\ge,f})$ for admissible $u$ is strictly
elliptic for $\ge > 0$, and
weakly elliptic for $\ge = 0$.
\begin{proof} We compute the principal symbol of $\LL$.  We will ignore
the first term of (\ref{linop}), which has weakly positive symbol.  Now fix a
vector $V = \left(\gl, X \right) \in T [0,1] \times TM$.  It follows from (\ref{linop}) that the principal symbol of $\LL$ acts via
\begin{align*}
L(V,V) =&\ u_{tt}^{1-k} \IP{T_{k-1}(E_u^{\ge}), (1+\ge) u_{tt} X \otimes X - \N u_t \otimes  (\gl X) - (\gl X) \otimes \N u_t + u_{tt}^{-1} \N u_t \otimes \N u_t (\gl^2)}
\end{align*}
It follows from the Cachy-Schwarz inequality that for any $\rho > 0$, as an inequality of matrices one has
\begin{align*}
- \gl X \otimes \N u_t - \gl \N u_t \otimes X \leq&\ \rho X \otimes X +
\rho^{-1} \gl^2 \N u_t \otimes \N u_t
\end{align*}
Applying this inequality with $\rho = (1 + \frac{\ge}{2}) u_{tt}$ yields
\begin{align*}
(1+\ge) u_{tt} X \otimes X - \N u_t \otimes  (\gl X) - (\gl X) \otimes \N u_t + u_{tt}^{-1} \N u_t \otimes \N u_t (\gl^2) \geq \frac{\ge}{2} u_{tt} X \otimes X + \frac{\ge}{2} u_{tt}^{-1} \gl^2.
\end{align*}
Since $u$ is admissible, we have $T_{k-1}(E^{\ge}_u) > 0$, and the result follows.
\end{proof}
\end{lemma}

\subsection{\texorpdfstring{$C^0$ estimate}{C0 estimate}}
To prove a $C^0$-estimate we begin with two technical lemmas:

\begin{lemma} \label{Ltt} Suppose $\phi = \phi(t)$.  Then
\begin{align*}
\LL \phi =&\ \phi_{tt} (1+\ge)^{k} \gs_k(A_u).
\end{align*}
\begin{proof} We directly compute using (\ref{linop}), Lemma \ref{LAlemma}, and
the geodesic equation that
\begin{align*}
\LL \phi =&\ \phi_{tt} \left\{ (1+\ge)^{k-1} u_{tt}^{-1} f + u_{tt}^{1-k}
\IP{T_{k-1}(E_u^{\ge}), u_{tt}^{-1} \N u_t \otimes \N u_t} \right\}\\
=&\ \phi_{tt} \left\{ (1+\ge)^{k-1} u_{tt}^{-1} f + u_{tt}^{1-k}
\IP{T_{k-1}((1+\ge) u_{tt} A_u), u_{tt}^{-1} \N u_t \otimes \N u_t} \right\}\\
=&\ \phi_{tt} (1+\ge)^{k-1} \left\{ u_{tt}^{-1} f + u_{tt}^{-1}
\IP{T_{k-1}(A_u), \N u_t \otimes \N u_t} \right\}\\
=&\ \phi_{tt} (1+\ge)^{k} \gs_k(A_u).
\end{align*}
\end{proof}
\end{lemma}

\begin{lemma} \label{Lucalc} Let $u$ be an admissible solution to
($\star_{\ge,f}$).  Then
\begin{align*}
  \LL u =&\ (k+1) (1+\ge)^{k-1} f + (1+\ge) u^{2-k}_{tt} \IP{T_{k-1}(E_u), - A +
\N u \otimes \N u - \frac{1}{2} \brs{\N u}^2 g}.
\end{align*}
\begin{proof} To begin we directly compute using (\ref{linop}) that
\begin{align*}
 \LL u =&\ (1+\ge)^{k-1} f + u_{tt}^{1-k} \left<T_{k-1}(E_u^{\ge}), (1+\ge)
u_{tt} \left( \N^2 u + 2 \N u \otimes \N u - \brs{\N u}^2 g \right) - \N u_t
\otimes \N u_t \right>.
\end{align*}
For the second term we simplify
\begin{align*}
 (1+\ge) u^{2-k}_{tt}&  \IP{T_{k-1}(E^{\ge}_u),\N^2 u}\\
=&\ (1+\ge) u^{2-k}_{tt} \IP{T_{k-1}(E_u),
A_u - A - \N u \otimes \N u + \frac{1}{2} \brs{\N u}^2 g}\\
 =&\ u_{tt}^{2-k} \IP{T_{k-1}(E_u), u_{tt}^{-1} \left[ E_u + \N
u_t \otimes \N u_t \right] }\\
&\ + (1+\ge) u_{tt}^{2-k}\IP{T_{k-1}(E), - A - \N u \otimes \N u + \frac{1}{2}
\brs{\N u}^2
g}\\
=&\ k u_{tt}^{1-k} \gs_k(E) + u_{tt}^{1-k} \IP{T_{k-1}(E_u), \N u_t \otimes \N
u_t} \\
&\ + (1+\ge) u^{2-k}_{tt}
\IP{T_{k-1}(E_u),
-A - \N u \otimes \N u + \frac{1}{2} \brs{\N u}^2 g}\\
=&\ k (1+\ge)^{k-1} f + u_{tt}^{1-k} \IP{T_{k-1}(E_u), \N u_t \otimes \N u_t}\\
&\ + (1+\ge) u^{2-k}_{tt} \IP{T_{k-1}(E_u),
-A - \N u \otimes \N u + \frac{1}{2} \brs{\N u}^2 g}.
\end{align*}
Combining these calculations yields the result.
\end{proof}
\end{lemma}

\begin{prop} \label{C0est} Let $u$ be an admissible solution to
($\star_{\ge,f}$).  Then
\begin{align*}
\sup_{M \times [0,1]} \brs{u} \leq C(u_{|M \times \{0,1\}}, \max_M f).
\end{align*}
\begin{proof} We first observe that an admissible solution to
(\ref{regularizedgeod})
satisfies $u_{tt} \geq 0$, and hence by convexity one  has $\sup_{M \times
[0,1]} u \leq \sup_{M \times \{0,1\}} u$.  To obtain the lower bound, fix a
constant $\gL$ and let
\begin{align*}
 \Psi = u + \gL t(1-t).
\end{align*}
Observe that at an
interior spacetime minimum of $\Psi$ one has
\begin{align*}
 0 = \N u, \qquad \N^2 u > 0.
\end{align*}
Using this and Lemma \ref{Lucalc} yields, at such a spacetime minimum,
\begin{align*}
 \LL \Psi =&\ (k+1) (1+\ge)^{k-1} f - (1+\ge) u_{tt}^{2-k} \IP{T_{k-1}(E_u),
A}\\
 &\  - 2 \gL \left[ (1+\ge)^{k-1} u_{tt}^{-1} f + u_{tt}^{1-k}
\IP{T_{k-1}(E_u^{\ge}), u_{tt}^{-1} \N u_t \otimes \N u_t} \right].
\end{align*}
Next we claim
\begin{align*}
\Psi_{tt} \N^2 \Psi - \N \Psi_t \otimes \N \Psi_t \geq 0.
\end{align*}
Since we are at a minimum for $\Psi$, $\Psi_{tt} \N^2 \Psi$ is a positive
semidefinite matrix.  The expression above is thus the difference between a
positive semidefinite
matrix and a negative definite rank $1$ matrix.  The lemma follows if we
establish positivity in the nondegenerate direction of the rank $1$ matrix we
subtracted, i.e. $\N \Psi_t$.  In particular it then suffices to show
\begin{align*}
\Psi_{tt} \N^{k} \N^l \Psi \N_k \Psi_t \N_l \Psi_t - \brs{\N \Psi_t}^4 \geq 0.
\end{align*}
To establish this we use that $\Psi$ is actually a spacetime minimum.  This
implies
that the spacetime Hessian is positive semidefinite.  Testing this condition
against the vector $- \sqrt{\Psi_{tt}} \N \Psi_t \oplus \frac{\brs{\N
\Psi_t}^2}{\sqrt{\Psi_{tt}}} \frac{\del}{\del t}$ yields
\begin{align*}
0 \leq&\ \Psi_{tt} \N^k \N^l \Psi \N_k \Psi_t \N_l \Psi_t - 2 \brs{\N \Psi_t}^4
+ \brs{\N
\Psi_t}^4,
\end{align*}
as required.  However, using the explicit form of $\Psi$ we see that this
implies
\begin{align*}
 (u_{tt} - \gL) \N^2 u - \N u_t \otimes \N u_t \geq 0,
\end{align*}
which since $\N^2 u > 0$ implies
\begin{align*}
 u_{tt} \N^2 u - \N u_t \otimes \N u_t \geq 0.
\end{align*}
Hence $E_u \geq u_{tt} A$, and then we obtain using Lemma
\ref{newtonprops} that
\begin{align*}
u_{tt}^{2-k} \IP{T_{k-1}(E_u),A} =&\ u_{tt}^{2-k} \Sigma(E_u,\dots,E_u,A)\\
\geq&\ u_{tt}^{2-k} \Sigma(u_{tt} A,\dots,u_{tt} A, A)\\
=&\ u_{tt} \gs_k(A)\\\
\geq&\ 0.
\end{align*}
We can also simplify
\begin{align*}
u_{tt}^{1-k} \IP{T_{k-1}(E_u^{\ge}), u_{tt}^{-1} \N u_t \otimes \N u_t} =&\
(1+\ge)^{k-1} u_{tt}^{-1} \IP{T_{k-1}(A_u), \N u_t \otimes \N u_t}\\
=&\ - (1+\ge)^{k-1} u_{tt}^{-1} f + (1+\ge)^{k-1}\gs_k(A_u)
\end{align*}
Combining these observations yields, at the interior minimum,
\begin{align*}
\LL \Psi \leq&\ (k+1)(1+\ge)^{k-1} f - 2 \gL (1+\ge)^{k-1} \gs_k(A_u)\\
\leq&\ (k+1)(1+\ge)^{k-1} f - 2 \gL (1+\ge)^{k-1} \gs_k(A)\\
\leq&\ C f - 2 \gd \gL,
\end{align*}
for some constants $C$ and $\gd$ depending only on the background data and maximum of $f$.
Choosing $\gL$ sufficiently large with respect to these constants yields $\LL
\Psi < 0$.  Hence $\Psi$ cannot have an interior minimum, and the result
follows.
\end{proof}
\end{prop}

\begin{rmk}  In the following estimates, all bounds on solutions be understood to depend on
\begin{align*}
\max_M \big\{ f + \frac{|f_t|}{f} + \frac{|\N f|}{f} + \frac{|f_{tt}|}{f} + \frac{|\N^2 f|}{f} \big\},
\end{align*}
but this dependence will be suppressed to simplify the exposition.  \end{rmk}

\subsection{\texorpdfstring{$C^1$ estimates}{C1 estimates}}

\begin{prop} \label{utestimate} Given $u$ an admissible
solution to $(\star_{\ge,f})$, one has
 \begin{align*}
 \sup_{M \times [0,1]} \brs{u_t} \leq C.
 \end{align*}
\begin{proof} First we observe that, since $u_{tt} \geq 0$, it follows that
there is a constant such that
$u_t(0) \leq C$ by direct integration.  Now fix constants $\gL_1, \gL_2$ and
consider
\begin{align*}
\Phi(x,t) = u(x,t) - u(x,0) - \gL_1 t^2 +  \gL_2 t,
\end{align*}
where $\gL_1$ is chosen large below, and $\gL_2$ is chosen still larger so that
$\Phi(x,1) \geq 0$.  First note using (\ref{linop}) that
\begin{align*}
\LL u_0 =&\ u_{tt}^{1-k} \left<T_{k-1}(E_u^{\ge}), (1 + \ge) u_{tt} \left( \N^2
u_0 + \N u_0 \otimes \N u + \N u
\otimes \N u_0 - \IP{\N u_0,\N u} g \right) \right>.
\end{align*}
Combining this with Lemmas \ref{Ltt} and \ref{Lucalc} we obtain
\begin{align*}
\LL \Phi =&\ \LL u - \LL u_0 - \gL_1 \LL t^2\\
=&\ (1+\ge) u^{2-k}_{tt} \IP{T_{k-1}(E_u), - A - \N^2 u_0 + \N u \otimes \N u -
2 \N u \otimes \N u_0 - \frac{1}{2} \brs{\N u}^2 g + \IP{\N u_0, \N u} g}\\
&\ (k+1) (1+\ge)^{k-1} f - 2 \gL_1 (1+\ge)^{k-1} u_{tt}^{-1} f - 2 \gL_1
u_{tt}^{-k} \IP{T_{k-1}(E_u^{\ge}), \N u_t \otimes \N u_t}.
\end{align*}
Also we have $\N u = \N u_0$ at the minimum, so we can simplify to
\begin{align*}
\LL \Phi =&\ - u_{tt}^{2-k}
\IP{T_{k-1}(E),  A +
\N^2 u_0 + \N u_0 \otimes \N u_0 - \frac{1}{2}
\brs{\N u_0}^2 g}\\
&\ + (k+1) (1+\ge)^{k-1} f - 2 \gL_1 (1+\ge)^{k} \gs_k(A_u)\\
=&\ - u_{tt}^{2-k}
\IP{T_{k-1}(E),  A_{u_0}} + (k+1) (1+\ge)^{k-1} f - 2 \gL_1 (1+\ge)^{k}
\gs_k(A_u).
\end{align*}
At a spacetime minimum for $\Phi$ we have $\N^2(u - u_0) \geq 0$, and hence
\begin{align*}
0 \leq&\ \Phi_{tt} \N^2 \Phi - \N \Phi_t \otimes \N \Phi_t\\
=&\ (u_{tt} - 2 \gL_1) \N^2 (u - u_0) - \N u_t \otimes \N u_t\\
\leq&\ u_{tt} \N^2 (u - u_0) - \N u_t \otimes \N u_t.
\end{align*}
Using this yields
\begin{align*}
E_u =&\  \left[ (1+\ge) u_{tt} A_u - \N u_t \otimes \N u_t \right]\\
=&\ \left[ (1+\ge) u_{tt} \left( A + \N^2 u + \N u \otimes \N u -
\frac{1}{2} \brs{\N u}^2 g \right) - \N u_t \otimes \N u_t\right]\\
\geq&\ \left[ (1+\ge) u_{tt} \left( A + \N^2 u_0 + \N u \otimes \N u -
\frac{1}{2} \brs{\N u}^2 g \right) \right]\\
=&\ \left[ (1+\ge) u_{tt} \left( A + \N^2 u_0 + \N u_0 \otimes \N u_0 -
\frac{1}{2} \brs{\N u_0}^2 g \right) \right].
\end{align*}
It follows from Lemma \ref{newtonmonotone} that
\begin{align*}
\IP{T_{k-1}(E), A} \geq 0.
\end{align*}
A similar calculation shows that at the minimum point under consideration we
have
\begin{align*}
\gs_k(A_u) \geq \gs_k(A_{u_0}).
\end{align*}
Putting these estimates together yields
\begin{align*}
\LL \Phi \leq&\ (k+1) (1+\ge)^{k-1} f - 2 \gL_1 (1+\ge)^k \gs_k(A_{u_0}).
\end{align*}
If we choose $\gL_1$ sufficiently large with respect to the positive lower bound
for $\gs_k(A_{u_0})$ and the maximum of $f$ we obtain
$L \Phi < 0$, and hence $\Phi$ cannot have an interior
minimum.  Thus it follows that
$\Phi_t(x,0) \geq 0$ for all $x$, and thus the lower bound for $u_t(0)$ follows.
 A very similar estimate yields a two sided bound for $u_t(1)$.  Since
$u_{tt} \geq 0$ everywhere we have a two sided bound for $u_t$ everywhere.
\end{proof}
\end{prop}

We next proceed to obtain the interior spatial gradient estimate.  To do this we need two preliminary calculations.

\begin{lemma} \label{Leucalc} Let $u$ be an admissible solution to
($\star_{\ge,f}$).  Then
\begin{align*}
  \LL e^{-\gl u} \geq&\ - \gl e^{- \gl u} \LL u + \frac{1}{2} \gl^2 e^{-\gl u} u_{tt}^{2-k} \IP{T_{k-1}(E_u^{\ge}) , \N u \otimes \N u} - C \gl^2 e^{-\gl u} \gs_k(A_u) u_t^2.
\end{align*}
\begin{proof} To begin we directly compute using (\ref{linop}) that
\begin{align*}
 \LL e^{-\gl u} =&\ (1+\ge)^{k-1} u_{tt}^{-1} f (e^{-\gl u})_{tt}\\
 &\ + u_{tt}^{1-k} \left<T_{k-1}(E_u^{\ge}), (1+\ge)
u_{tt} \left( \N^2 e^{-\gl u} + \N e^{-\gl u} \otimes \N u + \N u \otimes \N e^{-\gl u} - \IP{\N e^{-\gl u},\N u} \right) \right.\\
&\ \qquad \left. - \N (e^{-\gl u})_t \otimes \N u_t - \N u_t \otimes \N (e^{-\gl u})_t + u_{tt}^{-1} \N u_t \otimes \N u_t (e^{-\gl u})_{tt} \right>\\
=&\ - \gl e^{- \gl u} \LL u + (1+\ge)^{k-1} u_{tt}^{-1} f \gl^2 e^{- \gl u} u_t^2\\
&\ + \gl^2 e^{-\gl u} u_{tt}^{1-k} \IP{T_{k-1}(E_u^{\ge}), (1+\ge) u_{tt} \N u \otimes \N u - u_t \N u \otimes \N u_t - u_t \N u_t \otimes \N u + u_{tt}^{-1} u_t^2 \N u_t \otimes \N u_t}.
\end{align*}
Next we observe using the Cauchy-Schwarz inequality and equation ($\star_{\ge,f}$) that
\begin{align*}
u_{tt}^{1-k} & \IP{T_{k-1}(E_u^{\ge}), (1+\ge) u_{tt} \N u \otimes \N u - u_t \N u \otimes \N u_t - u_t \N u_t \otimes \N u + u_{tt}^{-1} u_t^2 \N u_t \otimes \N u_t}\\
 =&\ \gs_k(A_u) u_t^2 - 2 u_t
u_{tt}^{1-k} \IP{T_{k-1}(E_u), \N
u_t \otimes \N u} + u_{tt}^{2-k} \IP{T_{k-1}(E_u),\N u \otimes \N u}\\
\geq&\ - C \gs_k(A_u) u_t^2 + \frac{1}{2} u_{tt}^{2-k} \IP{T_{k-1}(E_u),\N u
\otimes \N u}.
\end{align*}
Combining these calculations yields the result.
\end{proof}
\end{lemma}

\begin{lemma} \label{Lut} Given $u$ an admissible
solution to $(\star_{\ge,f})$, one has
 \begin{align*}
  \LL u_t^2 =&\ 2 u_t f_t + 2 (1+\ge)^{k-1} f u_{tt} + 2 \ge u_{tt}^{2-k}
T_{k-1}(E)^{jk} \N_j u_t \N_k u_t.
 \end{align*}
\begin{proof} It follows directly from the definition of $\LL$ that $\LL u_t =
f_t$.  It follows that
\begin{align*}
\LL u_t^2 =&\ 2 u_t \LL u_t + 2 (1+\ge)^{k-1} f u_{tt}\\
&\ + 2 u_{tt}^{1-k} T_{k-1}(E)^{jk} \left\{ (1+\ge) u_{tt} \N_j u_t \N_k u_t - 2
\N_j u_t \N_k u_t u_{tt} + \N_j u_t \N_k u_t u_{tt} \right\}\\
=&\ 2 u_t f_t + 2 (1+\ge)^{k-1} f u_{tt} + 2 \ge u_{tt}^{2-k} T_{k-1}(E)^{jk}
\N_j u_t \N_k u_t,
\end{align*}
as required.
\end{proof}
\end{lemma}

\begin{lemma} \label{Lgradu} Given $u$ an admissible
solution to $(\star_{\ge,f})$, one has
 \begin{align*}
\LL \brs{\N u}^2 =&\ 2 u_{tt}^{1-k} T_{k-1}(E)^{jk} \left\{ (1+\ge) u_{tt} \N_i
\N_j u \N_i \N_k u - 2 \N_i \N_j u \N_k u_t \N_i u_t + u_{tt}^{-1} \N_j u_t \N_k
u_t \brs{\N u_t}^2 \right\}\\
&\ + 2(1+\ge)^{k-1} u_{tt}^{-1} f \brs{\N u_t}^2 + 2 \IP{\N f, \N u} - 2 (1+\ge)
u_{tt}^{2-k} \IP{T_{k-1}(E), \N^i u \N_i A + R_{ijk}^l \N^i u \N_l u}.
 \end{align*}
 \begin{proof} To begin we take the gradient of the geodesic equation to yield
 \begin{align*}
\N_i f =&\ \N_i \left[ u_{tt}^{1-k} \gs_k(E^{\ge}_u) \right]\\
=&\ (1-k) u_{tt}^{-k} \N_i u_{tt} \gs_k(E^{\ge}_u) + u_{tt}^{1-k}
\IP{T_{k-1}(E^{\ge}_u),
\N_i E_u^{\ge}}\\
=&\ (1-k) u_{tt}^{-k} \N_i u_{tt} \gs_k(E_u)\\
&\ + u_{tt}^{1-k} \IP{T_{k-1}(E^{\ge}_u),
(1+\ge) \N_i u_{tt} A_u + (1+\ge) u_{tt} \N_i A_u - \N_i \N u_t \otimes \N u_t -
\N u_t \otimes
\N_i \N u_t}.
 \end{align*}
A calculation similar to (\ref{lin20}) shows that
 \begin{align*}
 (1-k) u_{tt}^{-k}&  \N_i u_{tt} \gs_k(E_u) + u_{tt}^{1-k} \IP{T_{k-1}(E_u),
(1+\ge) \N_i
u_{tt} A_u}\\
 =&\ (1+\ge)^{k-1} u_{tt}^{-1} f \N_i u_{tt} + u_{tt}^{-k} \IP{T_{k-1}(E), \N
u_t \otimes \N u_t} \N_i u_{tt}
 \end{align*}
Next we simplify
\begin{align*}
 \N_i (A_u)_{jk} =&\ \N_i \left[ A_{jk} + \N_j \N_k u + \N_j u \N_k u -
\frac{1}{2} \brs{\N u}^2 g_{jk} \right]\\
=&\ \N_i A_{jk} + \N_i \N_j \N_k u + \N_i \N_j u \N_k u + \N_j u \N_i \N_k u -
\frac{1}{2} \N_i
\brs{\N u}^2 g_{jk}\\
=&\ \N_i A_{jk} + \N_j \N_k \N_i u + R_{ij k}^l \N_l u + \N_i \N_j u \N_k u +
\N_j u \N_i \N_k u - \frac{1}{2} \N_i \brs{\N u}^2 g_{jk}.
\end{align*}
Hence we obtain the identity
\begin{align} \label{DPhi}
 \LL \N_i u =  \N_i f -  (1+\ge) u_{tt}^{2-k} T_{k-1}(E)^{jk} \big\{ \N_i A_{jk} + R_{ijk}^l
\N_l u \big\}.
\end{align}
On the other hand using (\ref{linop}) we have
\begin{align*}
\LL \brs{\N u}^2 =&\ 2 \IP{\LL \N u, \N u} + 2 (1+\ge)^{k-1} u_{tt}^{-1} f
\brs{\N u_t}^2\\
&\ + 2 u_{tt}^{1-k} T_{k-1}(E)^{jk} \left\{ (1+\ge) u_{tt} \N_i \N_j u \N_i \N_k
u - 2 \N_i \N_j u \N_k u_t \N_i u_t + u_{tt}^{-1} \N_j u_t \N_k u_t \brs{\N
u_t}^2 \right\}\\
=&\ 2 u_{tt}^{1-k} T_{k-1}(E)^{jk} \left\{ (1+\ge) u_{tt} \N_i \N_j u \N_i \N_k
u - 2 \N_i \N_j u \N_k u_t \N_i u_t + u_{tt}^{-1} \N_j u_t \N_k u_t \brs{\N
u_t}^2 \right\}\\
&\ + 2(1+\ge)^{k-1} u_{tt}^{-1} f \brs{\N u_t}^2 + 2 \IP{\N f, \N u} - 2 (1+\ge)
u_{tt}^{2-k} \IP{T_{k-1}(E), \N^i u \N_i A + R_{ijk}^l \N^i u \N_l u},
\end{align*}
as required.
\end{proof}
\end{lemma}

\begin{prop} \label{4dC1estimate} Given $u$ an admissible
solution to $(\star_{\ge,f})$, one has
 \begin{align*}
 \sup_{M \times [0,1]} \brs{\N u}^2 \leq C.
 \end{align*}
 \begin{proof}  Without loss of generality we can assume $u < 0$.  Choose
$\lambda_1, \lambda_2, \lambda_3 \in \mathbb R$ and let
 \begin{align*}
 \Phi = \brs{\N u}^2 + \gl_1 u_t^2 + e^{- \gl_2 u} + \gl_3 t(t-1).
 \end{align*}
  Lemmas \ref{Ltt}, \ref{Leucalc}, \ref{Lut}, \ref{Lgradu} show
that
 \begin{align*}
 \LL \Phi \geq&\ \LL \brs{\N u}^2 + 2 \gl_1 \left[ f_t u_t + (1+\ge)^{k-1} f u_{tt} + \ge u_{tt}^{2-k}
T_{k-1}(E)^{jk} \N_j u_t \N_k u_t \right]\\
&\ - \gl_2 \LL u e^{-\gl_2 u} + \frac{\gl_2^2}{2} e^{-\gl_2 u} u_{tt}^{2-k} \IP{T_{k-1}(E_u^{\ge}) , \N u \otimes \N u} - C \gl_2^2 e^{-\gl_2 u} \gs_k(A_u) u_t^2 + \gl_3 \gs_k(A_u)\\
 \geq&\ 2 \IP{\N f, \N u} + 2 \gs_k(A_u) \brs{\N u_t}^2 + 2 u_{tt}^{2-k}
\IP{T_{k-1}(E_u),
\N_j \N_i u \N_k \N_i u}\\
&\ - 4 u_{tt}^{1-k}
\IP{T_{k-1}(E_u)_{ij},\N_i u_t \N_k u_t \N_j \N_k u} - 2 u_{tt}^{2-k} \IP{T_{k-1}(E_u)_{jk} \N_i u, \N_i A_{jk} + R_{ij k}^l \N_l
u}\\
&\ - C \gl_1 f_t + \gl_1 f u_{tt}\\
&\ - \gl_2 e^{-\gl_2 u} \left[ f - u^{2-k}_{tt} \IP{T_{k-1}(E_u),A} +
u^{2-k}_{tt} \left[
\IP{T_{k-1}(E_u), \N u \otimes \N u} - \frac{1}{2} \tr T_{k-1}(E_u) \brs{\N u}^2
\right] \right]\\
&\ + \frac{\gl_2^2}{2} e^{-\gl_2 u} u_{tt}^{2-k} \IP{T_{k-1}(E_u^{\ge}) , \N u \otimes \N u} - C \gl_2^2 e^{-\gl_2 u} \gs_k(A_u) u_t^2 + \gl_3 \gs_k(A_u).
 \end{align*}
 First we observe that, using the Cauchy-Schwarz inequality and Lemma
\ref{LAlemma}
 \begin{align*}
4 u_{tt}^{1-k}& \IP{T_{k-1}(E_u)_{ij}, \N_i u_t \N_k u_t \N_j \N_k u}\\
=&\ 4 u_{tt}^{1-k} \left[ \IP{ T_{k-1}(E_u)^{\frac{1}{2}} \cdot \N^2 u,
T_{k-1}(E_u)^{\frac{1}{2}} \cdot \N u_t \otimes \N u_t} \right]\\
\leq&\ 2 u_{tt}^{2-k} \IP{T_{k-1}(E_u), \N^2 u \cdot \N^2 u} + 2 u_{tt}^{-k}
\IP{T_{k-1}(E_u), \N u_t \otimes \N u_t} \brs{\N u_t}^2\\
=&\ 2 u_{tt}^{2-k} \IP{T_{k-1}(E_u), \N^2 u \cdot \N^2 u} + 2 u_{tt}^{-1}
\IP{T_{k-1}(A_u), \N u_t \otimes \N u_t} \brs{\N u_t}^2\\
=&\ 2 u_{tt}^{2-k} \IP{T_{k-1}(E_u), \N^2 u \cdot \N^2 u} + 2 \left[ \gs_k(A_u)
-f u_{tt}^{-1} \right]
\brs{\N u_t}^2.
 \end{align*}
Observe the preliminary inequality
 \begin{align*}
 u_{tt}^{2-k} \tr T_{k-1}(E_u) =&\ u_{tt}^{2-k} \gs_{k-1}(E_u)\\
 \geq&\ u_{tt}^{2-k} \left[ \gs_k(E_u)^{\frac{k-1}{k}} \right]\\
 =&\ u_{tt}^{2-k} \left[ f u_{tt}^{k-1} \right]^{\frac{k-1}{k}}\\
 =&\ f^{\frac{k-1}{k}} u_{tt}^{2-k + \frac{k^2 - 2 k + 1}{k}}\\
 =&\ f^{\frac{k-1}{k}} u_{tt}^{\frac{1}{k}}.
 \end{align*}
Next observe the estimate
 \begin{align*}
 \IP{\N f, \N u} \leq&\ C f u_{tt}^{-\frac{1}{k}} + C f u_{tt}^{\frac{1}{k}}
\brs{\N u}^2\\
 \leq&\ C f u_{tt}^{-1} + C f u_{tt} + C f u_{tt}^{\frac{1}{k}} \brs{\N u}^2\\
 \leq&\ C f u_{tt}^{-1} + C f u_{tt} + C f^{\frac{k-1}{k}} u_{tt}^{\frac{1}{k}}
\brs{\N u}^2.
 \end{align*}
Next observe that
\begin{align*}
-2 u_{tt}^{2-k} \IP{T_{k-1}(E_u)_{jk} \N_i u, \N_i A_{jk} + R_{ijk}^l \N_l u}
\geq&\ - C u_{tt}^{2-k} \tr T_{k-1} (E_u)
\left[1 + \brs{\N u}^2 \right].
\end{align*}
Combining these preliminary observations and using Proposition \ref{utestimate} yields
\begin{align*}
L \Phi \geq&\ \left(\gl_3 - C \right) f u_{tt}^{-1} + \left( \gl_1 - C \right) f
u_{tt} + \left( \frac{\gl_2}{4} e^{-\gl_2 u} - C \right) f^{\frac{k-1}{k}}
u_{tt}^{\frac{1}{k}} \brs{\N u}^2 - C \gl_1 f\\
&\ + e^{-\gl_2 u} u_{tt}^{2-k} \left[ \frac{1}{2} \gl_2^2 - \gl_2 \right]
\IP{T_{k-1}(E_u), \N u \otimes \N u}\\
&\ + u_{tt}^{2-k} \tr T_{k-1}(E_u) \left[ - C - C \brs{\N u}^2 + \frac{\gl_2}{4}
e^{-\gl_2 u} \brs{\N u}^2 \right]\\
&\ + \gs_k(A_u) \left[ \gl_3 - C \gl^2_2 \right]\\
\geq&\ \frac{\gl_3}{2} f u_{tt}^{-1} + \frac{\gl_1}{2} f u_{tt} - C \gl_1 f +
u_{tt}^{2-k} \tr T_{k-1}(E_u) \left[ - C + \brs{\N u}^2 \right]\\
\geq&\ u_{tt}^{2-k} \tr T_{k-1}(E_u) \left[ - C + \brs{\N u}^2 \right],
\end{align*}
where the second inequality follows by choosing $\gl_1,\gl_2$ large with
respect to universal constants and noting that $e^{-\gl_2 u} > 1$ for every choice of $\gl_2$, and then choosing $\gl_3$ large with respect to these choices.  The third inequality follows by
choosing $\gl_3$ large with respect to $\gl_1$.  Using the previously
establishing a priori estimates for $u$ and $u_t$, at a sufficiently large
maximum of $\Phi$ we will have $\brs{\N u}^2 \geq C$, and hence we see that $L
\Phi > 0$ at a sufficiently large maximum, a contradiction.  The a priori
estimate for $\brs{\N u}^2$ follows.
 \end{proof}
\end{prop}

\subsection{\texorpdfstring{$C^2$ estimates}{C2 estimates}}

\begin{lemma} \label{Lutt} Given $u$ an admissible solution of $(\star_{\ge,f})$
we
have
\begin{align*}
\LL u_{tt} =&\ - k f^{\frac{k}{k-1}} u_{tt} (1+\ge)^k \FF^{ij,kl} \left[ (E_u)_t
\right]_{ij} \left[ (E_u)_t \right]_{kl}\\
&\ + u_{tt}^{1-k} \left< T_{k-1}(E), 2 (1+\ge) u_{tt}^{-2} u_{ttt}^2 \N u_t
\otimes \N u_t - 4 u_{tt}^{-1} u_{ttt} \N u_{tt} \otimes \N u_t + 2 \N u_{tt}
\otimes \N u_{tt}\right.\\
&\ \qquad \qquad \left. - 2 u_{tt} \N u_t \otimes \N u_t + (1+\ge) u_{tt}
\brs{\N u_t}^2 g  \right>\\
&\ + (1+\ge)^{k-1} k f^{\frac{k-1}{k}} (f^{\frac{1}{k}})_{tt} + 2 (k-1)
(1+\ge)^{k-1} u_{tt}^{-1} f^{\frac{k-1}{k}} (f^{\frac{1}{k}})_t u_{ttt} -
2(1+\ge)^{k-1} u_{tt}^{-1} u_{ttt} f_t\\
&\ + (1+\ge)^{k-1} \frac{k+1}{k} f u_{tt}^{-2} u_{ttt}^2.
\end{align*}

\begin{proof}
First we compute using (\ref{linop}) that
\begin{gather} \label{lutt5}
\begin{split}
\LL u_{tt} =&\ (1+\ge)^{k-1} u_{tt}^{-1} f u_{tttt}\\
&\ + u_{tt}^{1-k} \left<T_{k-1}(E), (1+\ge) u_{tt} \left( \N^2 u_{tt} + \N
u_{tt} \otimes \N u + \N u \otimes \N u_{tt} - \IP{\N u_{tt} \N u} g \right)
\right.\\
&\ \qquad \qquad \qquad \left. - \N u_{ttt} \otimes \N u_t - \N u_t \otimes \N
u_{ttt} + u_{tt}^{-1} \N u_t \otimes \N u_t u_{tttt} \right>.
\end{split}
\end{gather}

To simplify notation we adopt the following (standard) conventions: for an $n \times n$ symmetric matrix $r = r_{ij}$ we denote
\begin{align*}
\FF(r) = \gs_k(r)^{1/k},
\end{align*}
and derivatives of $\FF$ with respect to the entries of $r$ by
\begin{align*}
\dfrac{\partial}{\partial r_{pq}}\FF(r) &= \FF(r)^{pq}, \\
 \dfrac{\partial^2}{\partial r_{pq} \partial r_{rs}}\FF(r) &= \FF(r)^{pq, rs}.
\end{align*}
We next need to differentiate the equation, which we can rewrite as
\begin{align*}
c_{\ge} f^{\frac{1}{k}} u_{tt}^{\frac{k-1}{k}} = \gs_k(E_u)^{\frac{1}{k}} =
\FF(E_u),
\end{align*}
where $c_{\ge} = (1+\ge)^{\frac{k-1}{k}}$.
Differentiating this yields
\begin{align*}
c_{\ge} \left( f^{\frac{1}{k}} \right)_t u_{tt}^{\frac{k-1}{k}} + c_{\ge}
\frac{k-1}{k}
f^{\frac{1}{k}} u_{tt}^{-\frac{1}{k}} u_{ttt} = \FF^{ij} \left[ \dt E_u
\right]_{ij} = \frac{1}{k} \gs_k(E_u)^{\frac{1-k}{k}} \IP{T_{k-1}(E_u),
(E_u)_t}.
\end{align*}
Differentiating again yields
\begin{gather} \label{lutt20}
\begin{split}
\FF^{ij} & \left[ (E_u)_{tt}
\right]_{ij} + \FF^{ij,kl} \left[ (E_u)_t \right]_{ij} \left[ (E_u)_t
\right]_{kl}\\
=&\ c_{\ge} \left[  \left( f^{\frac{1}{k}} \right)_{tt} u_{tt}^{\frac{k-1}{k}} +
2 \frac{k-1}{k}
\left( f^{\frac{1}{k}} \right)_t u_{tt}^{-\frac{1}{k}} u_{ttt}  - \frac{1}{k}
\left( \frac{k-1}{k} \right) f^{\frac{1}{k}} u_{tt}^{-\frac{1+k}{k}} u_{ttt}^2
\right.\\
&\  \left. + \frac{k-1}{k} f^{\frac{1}{k}} u_{tt}^{-\frac{1}{k}} u_{tttt}
\right].
\end{split}
\end{gather}
Next we want to get an explicit formula for $(E_u)_{tt}$, which we build up to
in stages.  We first observe the preliminary  computation
\begin{gather} \label{lutt10}
\begin{split}
(1+\ge) (A_u)_t =&\ \left[ u_{tt}^{-1} E_u + u_{tt}^{-1} \N u_t \otimes \N u_t
\right]_t\\
=&\ - u_{tt}^{-2} u_{ttt} E_u + u_{tt}^{-1} (E_u)_t - u_{tt}^{-2} u_{ttt} \N u_t
\otimes \N u_t + u_{tt}^{-1} \N u_{tt} \otimes \N u_t + u_{tt}^{-1} \N u_t
\otimes \N u_{tt}
\end{split}
\end{gather}
Next we compute that
\begin{align*}
\left[ (E_u)_t \right] =&\ (1+\ge) u_{ttt} A_u + (1+\ge) u_{tt} (A_u)_t - \N
u_{tt} \otimes \N
u_t - \N u_t \otimes \N u_{tt}\\
=&\ (1+\ge) u_{ttt} A_u + (1+\ge) u_{tt} \left[ \N^2 u_t + \N u_t \otimes \N u +
\N u \otimes \N
u_t - \IP{\N u_t, \N u} g \right]\\
&\ - \N u_{tt} \otimes \N u_t - \N u_t \otimes \N u_{tt}.
\end{align*}
Next we have, using (\ref{lutt10}),
\begin{align*}
\left[ (E_u)_{tt} \right] =&\ (1+\ge) u_{tttt} A_u + 2 (1+\ge) u_{ttt} (A_u)_t +
(1+\ge) u_{tt}
(A_u)_{tt}\\
&\ - \N u_{ttt} \otimes \N u_t - 2 \N u_{tt} \otimes \N u_{tt} - \N u_t
\otimes \N u_{ttt}\\
=&\ (1+\ge) u_{tttt} A_u + 2 (1+\ge) u_{ttt} (A_u)_t\\
&\  + (1+\ge) u_{tt} \left[ \N^2 u_{tt} + \N u_{tt} \otimes \N u + 2 \N u_t
\otimes \N
u_t + \N u \otimes \N u_{tt} - \brs{\N u_t}^2 g - \IP{\N u, \N u_{tt}} g
\right]\\
&\ - \N u_{ttt} \otimes \N u_t - 2 \N u_{tt} \otimes \N u_{tt} - \N u_t \otimes
\N u_{ttt}\\
=&\ (1+\ge) u_{tttt} A_u\\
&\ + 2 u_{ttt} \left[ - u_{tt}^{-2} u_{ttt} E_u + u_{tt}^{-1}
(E_u)_t - u_{tt}^{-2} u_{ttt} \N u_t \otimes \N u_t + u_{tt}^{-1} \N u_{tt}
\otimes \N u_t + u_{tt}^{-1} \N u_t \otimes \N u_{tt} \right]\\
&\  + (1+\ge) u_{tt} \left[ \N^2 u_{tt} + \N u_{tt} \otimes \N u + 2 \N u_t
\otimes \N
u_t + \N u \otimes \N u_{tt} - \brs{\N u_t}^2 g - \IP{\N u, \N u_{tt}} g
\right]\\
&\ - \N u_{ttt} \otimes \N u_t - 2 \N u_{tt} \otimes \N u_{tt} - \N u_t \otimes
\N u_{ttt}.
\end{align*}
Hence
\begin{align*}
k & \gs_k(E_u)^{\frac{k-1}{k}} u_{tt}^{1-k} \FF^{ij} \left[ (E_u)_{tt}
\right]_{ij}\\
=&\ u_{tt}^{1-k} \left< T_{k-1}(E_u), (1+\ge) u_{tttt} A_u - 2 u_{tt}^{-2}
u_{ttt}^2 E_u + 2 u_{tt}^{-1} u_{ttt} (E_u)_t - 2 u_{tt}^{-2} u_{ttt}^2 \N u_t
\otimes \N u_t \right.\\
&\ \qquad \qquad + 4 u_{tt}^{-1} u_{ttt} \N u_{tt} \otimes \N u_t + (1+\ge)
\left\{ u_{tt} \N^2 u_{tt} + 2
u_{tt} \N u_{tt} \otimes \N u + 2 u_{tt} \N u_t \otimes \N u_t \right. \\
&\ \qquad \qquad \left. \left. - u_{tt} \brs{\N u_t}^2 g - u_{tt} \IP{\N u, \N
u_{tt}} g \right\} - 2 \N
u_{ttt} \otimes \N u_t - 2  \N u_{tt} \otimes \N u_{tt} \right>\\
=&\ \sum_{i=1}^{12} A_i.
\end{align*}
Comparing against (\ref{lutt5}) yields
\begin{align*}
\LL u_{tt} =&\ A_1 + A_6 + A_7 + A_{10} + A_{11}\\
&\ + u_{tttt} u_{tt}^{1-k} \IP{T_{k-1}(E), - (1+\ge) A_u + u_{tt}^{-1} \N u_t
\otimes \N u_t}  + (1+\ge)^{k-1} u_{tt}^{-1} f u_{tttt}\\
=&\ A_1 + A_6 + A_7 + A_{10} + A_{11}\\
&\ + u_{tttt} \left[ u_{tt}^{-k} \IP{T_{k-1}(E), - E} + (1+\ge)^{k-1}
u_{tt}^{-1} f \right]\\
=&\ A_1 + A_6 + A_7 + A_{10} + A_{11} + u_{tttt} \left[ k u_{tt}^{-k} \gs_k(E) +
(1+\ge)^{k-1} u_{tt}^{-1} f \right]\\
=&\ A_1 + A_6 + A_7 + A_{10} + A_{11} + f (1+\ge)^{k-1} (1-k) u_{tt}^{-1}
u_{tttt}.
\end{align*}
Hence we obtain
\begin{gather} \label{lutt15}
\begin{split}
\LL u_{tt}
=&\ k \gs_k(E_u)^{\frac{k-1}{k}} u_{tt}^{1-k} \FF^{ij} \left[ (E_u)_{tt}
\right]_{ij} - u_{tt}^{1-k} \left< T_{k-1}(E_u), - 2 u_{tt}^{-2}
u_{ttt}^2 E_u + 2 u_{tt}^{-1} u_{ttt} (E_u)_t \right.\\
&\ \left. - 2 (1+\ge) u_{tt}^{-2} u_{ttt}^2 \N u_t
\otimes \N u_t + 4 u_{tt}^{-1} u_{ttt} \N u_{tt} \otimes \N u_t + 2 u_{tt} \N
u_t \otimes \N u_t \right.\\
&\ \left. - (1+\ge) u_{tt} \brs{\N u_t}^2 g - 2 \N u_{tt} \otimes \N u_{tt}
\right> + f (1+\ge)^{k-1} (1-k) u_{tt}^{-1} u_{tttt}\\
=&\ k \gs_k(E_u)^{\frac{k-1}{k}} u_{tt}^{1-k} \left[ - \FF^{ij,kl} \left[
(E_u)_t \right]_{ij} \left[ (E_u)_t \right]_{kl} \right.\\
&\ \left. + c_{\ge} \left[  \left( f^{\frac{1}{k}} \right)_{tt}
u_{tt}^{\frac{k-1}{k}} + 2 \frac{k-1}{k}
\left( f^{\frac{1}{k}} \right)_t u_{tt}^{-\frac{1}{k}} u_{ttt}  - \frac{1}{k}
\left( \frac{k-1}{k} \right) f^{\frac{1}{k}} u_{tt}^{-\frac{1+k}{k}} u_{ttt}^2 +
\frac{k-1}{k} f^{\frac{1}{k}} u_{tt}^{-\frac{1}{k}} u_{tttt} \right] \right]\\
&\ + u_{tt}^{1-k} \left< T_{k-1}(E), 2 u_{tt}^{-2} u_{ttt}^2 E_u - 2 u_{tt}^{-1}
u_{ttt}(E_u)_t + 2 (1+\ge) u_{tt}^{-2} u_{ttt}^2 \N u_t \otimes \N u_t \right.\\
&\ \left. - 4 u_{tt}^{-1} u_{ttt} \N u_{tt} \otimes \N u_t - 2 u_{tt} \N
u_t \otimes \N u_t + (1+\ge) u_{tt} \brs{\N u_t}^2 g + 2 \N u_{tt} \otimes \N
u_{tt} \right>\\
&\ + f (1+\ge)^{k-1} (1-k) u_{tt}^{-1} u_{tttt}\\
=&\ \sum_{i=1}^{13} A_i.
\end{split}
\end{gather}
We now clean up some of the lower order terms.  In particular we express
\begin{align*}
k \gs_k(E)^{\frac{k-1}{k}} u_{tt}^{1-k} =&\ k \left[ f u_{tt}^{k-1}
(1+\ge)^{k-1} \right]^{\frac{k-1}{k}} u_{tt}^{1-k} = k f^{\frac{k-1}{k}}
u_{tt}^{\frac{1-k}{k}} (1+\ge)^{\frac{(k-1)^2}{k}}.
\end{align*}
Then observe
\begin{align*}
A_2 =&\ \left(k \gs_k(E)^{\frac{k-1}{k}} u_{tt}^{1-k} \right) \left(
(1+\ge)^{\frac{k-1}{k}} (f^{\frac{1}{k}})_{tt} u_{tt}^{\frac{k-1}{k}} \right)\\
=&\ \left( k f^{\frac{k-1}{k}} u_{tt}^{\frac{1-k}{k}}
(1+\ge)^{\frac{(k-1)^2}{k}} \right) \left( (1+\ge)^{\frac{k-1}{k}}
(f^{\frac{1}{k}})_{tt} u_{tt}^{\frac{k-1}{k}} \right)\\
=&\ (1+\ge)^{k-1} k f^{\frac{k-1}{k}} (f^{\frac{1}{k}})_{tt}
\end{align*}
Next
\begin{align*}
A_3 =&\ \left( k \gs_k(E)^{\frac{k-1}{k}} u_{tt}^{1-k} \right) \left(
(1+\ge)^{\frac{k-1}{k}} 2 \frac{k-1}{k} (f^{\frac{1}{k}})_t
u_{tt}^{-\frac{1}{k}} u_{ttt} \right)\\
=&\ \left( k f^{\frac{k-1}{k}} u_{tt}^{\frac{1-k}{k}}
(1+\ge)^{\frac{(k-1)^2}{k}} \right) \left( (1+\ge)^{\frac{k-1}{k}} 2
\frac{k-1}{k} (f^{\frac{1}{k}})_t u_{tt}^{-\frac{1}{k}} u_{ttt} \right)\\
=&\ 2 (k-1) (1+\ge)^{k-1} u_{tt}^{-1} f^{\frac{k-1}{k}} (f^{\frac{1}{k}})_t
u_{ttt}
\end{align*}
Next
\begin{align*}
A_4 =&\ \left( k f^{\frac{k-1}{k}} u_{tt}^{\frac{1-k}{k}}
(1+\ge)^{\frac{(k-1)^2}{k}} \right) \left( - (1+\ge)^{\frac{k-1}{k}} \frac{1}{k}
\left( \frac{k-1}{k} \right) f^{\frac{1}{k}} u_{tt}^{-\frac{1+k}{k}} u_{ttt}^2
\right)\\
=&\ - (1+\ge)^{k-1} \left( \frac{k-1}{k} \right) f u_{tt}^{-2} u_{ttt}^2.
\end{align*}
Next note that
\begin{align*}
A_{5} = k \gs_k(E_u)^{\frac{k-1}{k}} u_{tt}^{1-k} c_{\ge} \frac{k-1}{k}
f^{\frac{1}{k}} u_{tt}^{-\frac{1}{k}} u_{tttt} =&\ (k-1) (1+\ge)^{k-1} f
u_{tt}^{-1}  u_{tttt} = - A_{13}.
\end{align*}
Also observe
\begin{align*}
A_{6} =&\ u_{tt}^{1-k} \IP{T_{k-1}(E), 2 u^{-2}_{tt} u_{ttt}^2 E_u}\\
=&\ 2 k u^{-1-k}_{tt} u_{ttt}^2 \gs_k(E)\\
=&\ 2 k u_{tt}^{-1-k} u_{ttt}^2 \left[ f u_{tt}^{k-1} (1+\ge)^{k-1} \right]\\
=&\ 2 k (1+\ge)^{k-1} u_{tt}^{-2} u_{ttt}^2 f.
\end{align*}
Lastly
\begin{align*}
A_7 =&\ -2 u_{tt}^{1-k} \IP{T_{k-1}(E), u_{tt}^{-1} u_{ttt} (E_u)_t}\\
=&\ -2 u_{tt}^{-k} u_{ttt} \left[ \gs_k(E) \right]_t\\
=&\ - 2(1+\ge)^{k-1} u_{tt}^{-k} u_{ttt} \left[ f u_{tt}^{k-1}\right]_t\\
=&\ - 2(1+\ge)^{k-1} u_{tt}^{-k} u_{ttt} \left[ f_t u_{tt}^{k-1} + (k-1) f
u_{tt}^{k-2} u_{ttt} \right]\\
=&\ - 2(1+\ge)^{k-1} u_{tt}^{-1} u_{ttt} \left[ f_t + (k-1) f u_{tt}^{-1}
u_{ttt} \right].
\end{align*}
Inserting these simplifications into (\ref{lutt10}) yields the result.
\end{proof}
\end{lemma}

\begin{prop} \label{uttest} Given $u$ an admissible
solution to $(\star_{\ge,f})$, one has
 \begin{align*}
 \sup_{M \times [0,1]} u_{tt} \leq C \ge^{-1}.
 \end{align*}
 \begin{proof} Let's begin with a preliminary estimate for $\LL u_{tt}$.
Returning to Lemma \ref{Lutt} and considering the terms in order, one first
observes by convexity of $\FF$ that
 \begin{align*}
 - k f^{\frac{k}{k-1}} u_{tt} (1+\ge)^k \FF^{ij,kl} \left[ (E_u)_t \right]_{ij}
\left[ (E_u)_t \right]_{kl} \geq 0.
 \end{align*}
 Also, by an application of the Cauchy Schwarz inequality one has the matrix
inequality
 \begin{align*}
 2  u_{tt}^{-2} u_{ttt}^2 \N u_t \otimes \N u_t - 4 u_{tt}^{-1} u_{ttt} \N
u_{tt} \otimes \N u_t + 2 \N u_{tt} \otimes \N u_{tt} \geq 0.
 \end{align*}
 Also, since $u$ is an admissible solution we have
 \begin{align*}
 u_{tt}^{1-k} \IP{T_{k-1}(E), u_{tt} \brs{\N u_t}^2 g} =&\ u_{tt}^{2-k} \brs{\N
u_t}^2 \tr T_{k-1}(E) \geq 0.
\end{align*}
Also we observe
\begin{align*}
(1+\ge)^{k-1} k f^{\frac{k-1}{k}} (f^{\frac{1}{k}})_{tt} \leq&\ C
f^{\frac{k-1}{k}} \left[ f^{\frac{1}{k} - 1} f_{tt} + f^{\frac{1}{k} - 2} f_t^2
\right] \leq C f.
\end{align*}
Next
\begin{align*}
2 (k-1)(1+\ge)^{k-1} u_{tt}^{-1} f^{\frac{k-1}{k}} (f^{\frac{1}{k}})_t u_{ttt}
\leq&\ C f^{\frac{k-1}{k}}(f^{\frac{1}{k} - 1} f_t) u_{tt}^{-1} u_{ttt}\\
\leq&\ C f u_{tt}^{-1} u_{ttt}\\
\leq&\ C \gd^{-1} f + C \gd f u_{tt}^{-2} u_{ttt}^2.
\end{align*}
Also
\begin{align*}
- 2 (1+\ge)^{k-1} u_{tt}^{-1} u_{ttt} f_t \leq&\ C f u_{tt}^{-1} u_{ttt}\\
\leq&\ C \gd^{-1} f + C \gd f u_{tt}^{-2} u_{ttt}^2.
\end{align*}
Combining these estimates and choosing $\gd$ sufficiently small leads to the
preliminary estimate
\begin{gather} \label{uttest10}
\begin{split}
\LL u_{tt} \geq&\ - 2 u_{tt}^{2-k} \IP{T_{k-1}(E), \N u_t \otimes \N u_t} - C f.
\end{split}
\end{gather}
Similar considerations with the result of Lemma \ref{Lut} lead to the
preliminary estimate
\begin{gather} \label{uttest20}
\LL u_t^2 \geq - C f + 2 f u_{tt} + 2 \ge u_{tt}^{2-k} \IP{T_{k-1}(E), \N u_t
\otimes \N u_t}.
\end{gather}
 Now fix constants $\gl_i$ and let
 \begin{align*}
 \Phi = u_{tt} + \gl_1 \ge^{-1} u_t^2 + \gl_2 t(t-1).
 \end{align*}
Choosing $\gl_1 \geq 1$, combining Lemma \ref{Ltt} with (\ref{uttest10}) and
(\ref{uttest20}) yields
 \begin{align*}
 \LL \Phi \geq&\ 2 u_{tt}^{2-k} \IP{ T_{k-1}(E), (\gl_1 - 1) \N u_t \otimes \N
u_t} - f (C + C \gl_1 \ge^{-1}) + 2 \gl_1 \ge^{-1} f u_{tt} + \gl_2 f
u_{tt}^{-1}\\
 \geq&\ f \left[ \left( 2 \gl_1 \ge^{-1} - \gd \left( C + C \gl_1 \ge^{-1}
\right) \right) u_{tt} + \left( \gl_2 - \gd^{-1} \left(C + C \gl_1 \ge^{-1}
\right) \right) \right].
 \end{align*}
If we now choose $\gd$ small above with respect to universal constants and then
choose $\gl_2$ large with respect to $\gd$ we conclude
 \begin{align*}
 \LL \Phi > 0,
 \end{align*}
 and hence $\Phi$ cannot have an interior maximum.  The proposition follows.
 \end{proof}
\end{prop}

\begin{lemma} \label{Llapu}
Given $u$ an admissible solution of $(\star_{\ge,f})$ we
have
\begin{align*}
\LL(\gD u) =&\ - k \gs_k(E)^{\frac{k-1}{k}} u_{tt}^{1-k} \FF^{(ij),(kl)} \N_p
(E_u)_{ij} \N_p (E_u)_{kl}\\
&\ + u_{tt}^{1-k} T_{k-1}(E)^{ij} \Big\{  2 u_{tt}^{-2} \brs{\N
u_{tt}}^2 \N_i u_t \otimes \N_j u_t  - 4 u_{tt}^{-1} \N_p u_{tt} \N_i \N_p u_t
\N_j u_t  + 2 \N_i \N_p u_t \N_j
\N_p u_t \\
&\  - 2 (1+\ge) u_{tt} \N_i \N_p u \N_j \N_p u + (1+\ge) u_{tt} \brs{\N^2
u}^2 g_{ij} + u_{tt}
\OO(\brs{\N^2 u} + \brs{\N u}^2 + 1) \Big\} \\
&\ + k (1+\ge)^{k-1} f^{\frac{k-1}{k}} \gD (f^{\frac{1}{k}}) - (1+\ge)^{k-1}
\frac{2}{k} u_{tt}^{-1} \IP{\N f, \N u_{tt}}\\
&\ + (1+\ge)^{k-1} \left( \frac{k+1}{k} \right) f u_{tt}^{-2} \brs{\N u_{tt}}^2.
\end{align*}

\begin{proof}
To begin we compute using (\ref{linop})
\begin{gather} \label{Llap10}
\begin{split}
\LL(\gD u) =&\ (1 + \ge)^{k-1} u_{tt}^{-1} f \gD u_{tt}\\
&\ + u_{tt}^{1-k} \left< T_{k-1}(E), (1+\ge) u_{tt} \left( \N^2 \gD u + \N \gD u
\otimes \N u + \N u \otimes \N \gD u - \IP{\N \gD u, \N u} g \right) \right.\\
&\ \left. \qquad \qquad - \N \gD u_t \otimes \N u_t - \N u_t \otimes \N \gD u_t
+ u_{tt}^{-1} \N u_t \otimes \N u_t \gD u_{tt} \right>.
\end{split}
\end{gather}
Next we differentiate the equation, which we rewrite as
\begin{align*}
c_{\ge} f^{\frac{1}{k}} u_{tt}^{\frac{k-1}{k}} = \gs_k(E_u)^{\frac{1}{k}} =:
\FF(E_u),
\end{align*}
Differentiating yields
\begin{align*}
c_{\ge} \N_p (f^{\frac{1}{k}}) u_{tt}^{\frac{k-1}{k}} + c_{\ge} \left(
\frac{k-1}{k} \right)
f^{\frac{1}{k}} u_{tt}^{-\frac{1}{k}} \N_p u_{tt} = \FF^{ij} \N_p (E_u)_{ij}.
\end{align*}
Differentiating again yields
\begin{align*}
\FF^{ij} &(\gD E_u)_{ij} + \FF^{(ij),(kl)} \N_p (E_u)_{ij} \N_p (E_u)_{kl}\\
=& c_{\ge} \left[ \gD (f^{\frac{1}{k}})
u_{tt}^{\frac{k-1}{k}} + 2 \left( \frac{k-1}{k} \right) \IP{ \N
(f^{\frac{1}{k}}), \N u_{tt}} u_{tt}^{-\frac{1}{k}} \right.\\
&\ \left. - \frac{1}{k} \left( \frac{k-1}{k} \right) f^{\frac{1}{k}} u_{tt}^{-
\frac{1+k}{k}} \brs{\N u_{tt}}^2 + \left( \frac{k-1}{k} \right) f^{\frac{1}{k}}
u_{tt}^{-\frac{1}{k}} \gD u_{tt} \right].
\end{align*}
Next we have
\begin{align*}
\N_p (E_u)_{ij} =&\ \N_p \left[ (1+\ge) u_{tt} (A_u)_{ij} - \N_i u_t \N_j u_t
\right]\\
=&\ (1+\ge) \N_p u_{tt} (A_u)_{ij} + (1+\ge) u_{tt} \N_p (A_u)_{ij} - \N_p \N_i
u_t \N_j u_t -
\N_i u_t \N_p \N_j u_t.
\end{align*}
Differentiating again and commuting derivatives yields
\begin{align*}
(\gD E_u)_{ij} =&\ (1+\ge) \gD u_{tt} (A_u)_{ij} + 2 (1+\ge) \N_p u_{tt} \N_p
(A_u)_{ij} +
(1+\ge) u_{tt} \gD (A_u)_{ij}\\
&\ - \N_i \gD u_t \N_j u_t - \N_i u_t \N_j \gD u_t - 2 \N_i \N_p u_t \N_j \N_p
u_t\\
&\ - R_{ip} \N_p u_t \N_j u_t - R_{jp} \N_p u_t \N_i u_t.
\end{align*}
Differentiating the equation for the Schouten tensor yields
\begin{align*}
\N_p (A_u)_{ij} =&\ \N_p A_{ij} + \N_p \N_i \N_j u + \N_i \N_p u \N_j u + \N_i u
\N_j \N_p u - \frac{1}{2} \N_p \brs{\N u}^2 g.
\end{align*}
This implies
\begin{gather} \label{lapA}
 \begin{split}
\gD (A_u)_{ij} =&\ \gD A_{ij} + \N_i \N_j \gD u + \N_i \gD u \N_j u + \N_i u
\N_j \gD u\\
&\ + 2 \N_i \N_p u \N_j \N_p u - \brs{\N^2 u}^2 g_{ij} - \IP{\N u, \N \gD u}
g_{ij} + \OO(\brs{\N^2 u} + \brs{\N u}^2 + 1).
 \end{split}
\end{gather}
On the other hand it is also useful to express
\begin{align*}
(1+\ge) \N_p (A_u)_{ij} =&\ \N_p \left[ u_{tt}^{-1} (E_u)_{ij} + u_{tt}^{-1}
\N_i u_t
\N_j u_t \right]\\
=&\ u_{tt}^{-1} \N_p (E_u)_{ij} - u_{tt}^{-2} (E_u)_{ij} \N_p u_{tt} -
u_{tt}^{-2} \N_p u_{tt} \N_i u_t \N_j u_t\\
&\ + u_{tt}^{-1} \N_i \N_p u_t \N_j u_t + u_{tt}^{-1} \N_i u_t \N_j \N_p u_t.
\end{align*}
Combining the above calculations yields
\begin{align*}
\gD(E_u)_{ij} =&\ (1+\ge) \gD u_{tt} (A_u)_{ij} + 2 \N_p u_{tt} \left[
u_{tt}^{-1} \N_p
(E_u)_{ij} - u_{tt}^{-2} (E_u)_{ij} \N_p u_{tt} - u_{tt}^{-2} \N_p u_{tt} \N_i
u_t \N_j u_t \right.\\
&\ \left. + u_{tt}^{-1} \N_i \N_p u_t \N_j u_t + u_{tt}^{-1} \N_i u_t \N_j \N_p
u_t \right]\\
&\ + (1+\ge) u_{tt} \left[ \N_i \N_j \gD u + \N_i \gD u \N_j u + \N_i u \N_j \gD
u
\right.\\
&\ \left. + 2 \N_i \N_p u \N_j \N_p u - \brs{\N^2 u}^2 g_{ij} - \IP{\N u, \N \gD
u} g_{ij} + \OO(\brs{\N^2 u} + \brs{\N u}^2 + 1) \right]\\
&\ - \N_i \gD u_t \N_j u_t - \N_i u_t \N_j \gD u_t - 2 \N_i \N_p u_t \N_j \N_p
u_t\\
=&\ (1+\ge) \gD u_{tt} (A_u)_{ij} + 2 u_{tt}^{-1} \N_p u_{tt} \N_p (E_u)_{ij} -
2
u_{tt}^{-2}\brs{\N u_{tt}}^2 (E_u)_{ij} - 2 u_{tt}^{-2} \brs{\N u_{tt}}^2 \N_i
u_t \otimes \N_j u_t\\
&\ + 2 u_{tt}^{-1} \N_p u_{tt} \N_i \N_p u_t \N_j u_t + 2 u_{tt}^{-1} \N_p
u_{tt} \N_j \N_p u_t \N_i u_t\\
&\ + (1+\ge) u_{tt} \N_i \N_j \gD u + (1+\ge) u_{tt} \N_i \gD u \N_j u + (1+\ge)
u_{tt} \N_i u \N_j \gD
u + 2 (1+\ge) u_{tt} \N_i \N_p u \N_j \N_p u\\
&\  - (1+\ge) u_{tt} \brs{\N^2 u}^2 g_{ij} - (1+\ge) u_{tt}  \IP{\N u, \N \gD u}
g_{ij} + u_{tt}
\OO(\brs{\N^2 u} + \brs{\N u}^2 + 1)\\
&\ - \N_i \gD u_t \N_j u_t - \N_i u_t \N_j \gD u_t - 2 \N_i \N_p u_t \N_j \N_p
u_t.
\end{align*}
Thus
\begin{align*}
k & \gs_k(E)^{\frac{k-1}{k}} u_{tt}^{1-k} \FF^{ij} (\gD E_u)_{ij}\\
=&\ u_{tt}^{1-k} \left< T_{k-1}(E), (1+\ge) \gD u_{tt} (A_u)_{ij} + 2
u_{tt}^{-1} \N_p u_{tt} \N_p
(E_u)_{ij} - 2 u_{tt}^{-2}\brs{\N u_{tt}}^2 (E_u)_{ij} - 2 u_{tt}^{-2} \brs{\N
u_{tt}}^2 \N_i u_t \otimes \N_j u_t \right.\\
&\ + 4 u_{tt}^{-1} \N_p u_{tt} \N_i \N_p u_t \N_j u_t + (1+\ge) u_{tt} \N_i \N_j
\gD u + 2 (1+\ge) u_{tt} \N_i \gD u \N_j u\\
&\  + 2 (1+\ge) u_{tt} \N_i \N_p u \N_j \N_p u  - (1+\ge) u_{tt} \brs{\N^2 u}^2
g_{ij} - (1+\ge) u_{tt}  \IP{\N u, \N \gD u} g_{ij}\\
&\ \left. - 2 \N_i \gD u_t \N_j u_t - 2 \N_i \N_p u_t \N_j
\N_p u_t + u_{tt}
\OO(\brs{\N^2 u} + \brs{\N u}^2 + 1) \right>\\
=&\ \sum_{i=1}^{13} A_i.
\end{align*}
Comparing this against (\ref{Llap10}) yields
\begin{gather}
\begin{split}
\LL (\gD u) =&\ A_1 + A_6 + A_7 + A_{10} + A_{11}\\
&\ + u_{tt}^{1-k} \gD u_{tt} \IP{T_{k-1}(E), - (1+\ge) A_u + \N u_t \otimes \N
u_t} + (1+\ge)^{k-1} u_{tt}^{-1} f \gD u_{tt}\\
=&\ A_1 + A_6 + A_7 + A_{10} + A_{11}\\
&\ + u_{tt}^{-k} \gD u_{tt} \IP{T_{k-1}(E), -E} + (1+\ge)^{k-1} u_{tt}^{-1} f
\gD u_{tt}\\
=&\ A_1 + A_6 + A_7 + A_{10} + A_{11} + \gD u_{tt} \left[ - k u_{tt}^{-k}
\gs_k(E) + (1+\ge)^{k-1} u_{tt}^{-1} f \right]\\
=&\ A_1 + A_6 + A_7 + A_{10} + A_{11} + (1-k) (1+\ge)^{k-1} u_{tt}^{-1} f \gD
u_{tt}.
\end{split}
\end{gather}
Hence, collecting these calculations yields
\begin{align*}
\LL (\gD u) =&\ k \gs_k(E)^{\frac{k-1}{k}} u_{tt}^{1-k} \FF (\gD E_u)_{ij}\\
&\ - u_{tt}^{1-k} \left< T_{k-1}(E), 2 u_{tt}^{-1} \N_p u_{tt} \N_p
(E_u)_{ij} - 2 u_{tt}^{-2}\brs{\N u_{tt}}^2 (E_u)_{ij} - 2 u_{tt}^{-2} \brs{\N
u_{tt}}^2 \N_i u_t \otimes \N_j u_t \right.\\
&\ + 4 u_{tt}^{-1} \N_p u_{tt} \N_i \N_p u_t \N_j u_t  + 2 (1+\ge) u_{tt} \N_i
\N_p u \N_j \N_p u  - (1+\ge) u_{tt} \brs{\N^2 u}^2 g_{ij}\\
&\ \left. - 2 \N_i \N_p u_t \N_j
\N_p u_t + u_{tt}
\OO(\brs{\N^2 u} + \brs{\N u}^2 + 1) \right> + (1-k) (1+\ge)^{k-1} u_{tt}^{-1} f
\gD u_{tt}\\
=&\ - k \gs_k(E)^{\frac{k-1}{k}} u_{tt}^{1-k} \FF^{(ij),(kl)} \N_p (E_u)_{ij}
\N_p (E_u)_{kl}\\
&\ + c_{\ge} k \gs_k(E)^{\frac{k-1}{k}} u_{tt}^{1-k} \left[ \gD
(f^{\frac{1}{k}})
u_{tt}^{\frac{k-1}{k}} + 2 \left( \frac{k-1}{k} \right) \IP{ \N
(f^{\frac{1}{k}}), \N u_{tt}} u_{tt}^{-\frac{1}{k}} \right.\\
&\ \qquad \qquad \left. - \frac{1}{k} \left( \frac{k-1}{k} \right)
f^{\frac{1}{k}} u_{tt}^{-
\frac{1+k}{k}} \brs{\N u_{tt}}^2 + \left( \frac{k-1}{k} \right) f^{\frac{1}{k}}
u_{tt}^{-\frac{1}{k}} \gD u_{tt} \right]\\
&\ + u_{tt}^{1-k} \left< T_{k-1}(E),-  2 u_{tt}^{-1} \N_p u_{tt} \N_p
(E_u)_{ij} + 2 u_{tt}^{-2}\brs{\N u_{tt}}^2 (E_u)_{ij} + 2 u_{tt}^{-2} \brs{\N
u_{tt}}^2 \N_i u_t \otimes \N_j u_t \right.\\
&\ - 4 u_{tt}^{-1} \N_p u_{tt} \N_i \N_p u_t \N_j u_t  + 2 \N_i \N_p u_t \N_j
\N_p u_t - 2 (1+\ge) u_{tt} \N_i \N_p u \N_j \N_p u\\
&\ \left. + (1+\ge) u_{tt} \brs{\N^2 u}^2 g_{ij} + u_{tt}
\OO(\brs{\N^2 u} + \brs{\N u}^2 + 1) \right> + (1-k) (1+\ge)^{k-1} u_{tt}^{-1} f
\gD u_{tt}\\
=&\ \sum_{i=1}^{14} A_i.
\end{align*}
Now we simplify
\begin{align*}
A_2 =&\ \left( k \gs_k(E)^{\frac{k-1}{k}} u_{tt}^{1-k} \right) \left( c_{\ge}
u_{tt}^{\frac{k-1}{k}} \gD (f^{\frac{1}{k}})\right)\\
=&\ \left( k f^{\frac{k-1}{k}} u_{tt}^{\frac{1-k}{k}}
(1+\ge)^{\frac{(k-1)^2}{k}} \right) \left( (1+\ge)^{\frac{k-1}{k}}
u_{tt}^{\frac{k-1}{k}} \gD (f^{\frac{1}{k}}) \right)\\
=&\ k (1+\ge)^{k-1} f^{\frac{k-1}{k}} \gD (f^{\frac{1}{k}}).
\end{align*}
Next
\begin{align*}
A_3 =&\ \left( k \gs_k(E)^{\frac{k-1}{k}} u_{tt}^{1-k} \right) \left( 2 c_{\ge}
\left( \frac{k-1}{k} \right) \IP{\N (f^{\frac{1}{k}}), \N u_{tt}}
u_{tt}^{-\frac{1}{k}} \right)\\
=&\ \left( k f^{\frac{k-1}{k}} u_{tt}^{\frac{1-k}{k}}
(1+\ge)^{\frac{(k-1)^2}{k}} \right) \left( 2 (1+\ge)^{\frac{k-1}{k}} \left(
\frac{k-1}{k} \right) \IP{\N (f^{\frac{1}{k}}), \N u_{tt}} u_{tt}^{-\frac{1}{k}}
\right)\\
=&\ 2 (1+\ge)^{k-1} (k-1) f^{\frac{k-1}{k}} u_{tt}^{-1} \IP{\N
(f^{\frac{1}{k}}), \N u_{tt}}\\
=&\ (1+\ge)^{k-1} \left( 2 - \frac{2}{k} \right) u_{tt}^{-1} \IP{\N f, \N
u_{tt}}.
\end{align*}
Next
\begin{align*}
A_4 =&\ - \left( k \gs_k(E)^{\frac{k-1}{k}} u_{tt}^{1-k} \right) \left( c_{\ge}
\frac{1}{k} \left( \frac{k-1}{k} \right) f^{\frac{1}{k}} u_{tt}^{-\frac{1+k}{k}}
\brs{\N u_{tt}}^2 \right)\\
=&\ - \left( k f^{\frac{k-1}{k}} u_{tt}^{\frac{1-k}{k}}
(1+\ge)^{\frac{(k-1)^2}{k}} \right)  \left( (1+\ge)^{\frac{k-1}{k}} \frac{1}{k}
\left( \frac{k-1}{k} \right) f^{\frac{1}{k}} u_{tt}^{-\frac{1+k}{k}} \brs{\N
u_{tt}}^2 \right)\\
=&\ - (1+\ge)^{k-1} \left( \frac{k-1}{k} \right) f u_{tt}^{-2} \brs{\N
u_{tt}}^2.
\end{align*}
Next
\begin{align*}
A_5 =&\ \left( k \gs_k(E)^{\frac{k-1}{k}} u_{tt}^{1-k} \right) \left( c_{\ge}
\left( \frac{k-1}{k} \right) f^{\frac{1}{k}}
u_{tt}^{-\frac{1}{k}} \gD u_{tt} \right)\\
=&\ \left( k f^{\frac{k-1}{k}} u_{tt}^{\frac{1-k}{k}}
(1+\ge)^{\frac{(k-1)^2}{k}} \right) \left( (1+\ge)^{\frac{k-1}{k}} \left(
\frac{k-1}{k} \right) f^{\frac{1}{k}}
u_{tt}^{-\frac{1}{k}} \gD u_{tt} \right)\\
=&\ (k-1) (1+\ge)^{k-1} f u_{tt}^{-1} \gD u_{tt}\\
=&\ - A_{14}.
\end{align*}
Next
\begin{align*}
A_6 =&\ - 2 u_{tt}^{1-k} \N_p u_{tt} \IP{T_{k-1}(E), u_{tt}^{-1} \N_p
(E_u)_{ij}}\\
=&\ - 2 u_{tt}^{-k} \N_p u_{tt} \N_p \gs_k(E)\\
=&\ - 2 (1+\ge)^{k-1} u_{tt}^{-k} \N_p u_{tt} \N_p \left[ f u_{tt}^{k-1}
\right]\\
=&\ - 2(1+\ge)^{k-1} u_{tt}^{-1} \IP{\N f, \N u_{tt}} - 2(1+\ge)^{k-1} (k-1)  f
u_{tt}^{-2} \brs{\N u_{tt}}^2.
\end{align*}
Lastly
\begin{align*}
A_7 =&\ 2 u_{tt}^{1-k} u_{tt}^{-2} \brs{\N u_{tt}}^2 \IP{T_{k-1}(E), E}\\
=&\ 2 k u_{tt}^{1-k} u_{tt}^{-2} \brs{\N u_{tt}}^2 \gs_k(E)\\
=&\ 2 k (1+\ge)^{k-1} f u_{tt}^{-2} \brs{\N u_{tt}}^2.
\end{align*}
Collecting these simplifications yields the result.
\end{proof}
\end{lemma}

\begin{prop} \label{lapuest}
Given $u$ an admissible
solution to $(\star_{\ge,f})$, one has
 \begin{align*}
 \sup_{M \times [0,1]} \gD u \leq C \ge^{-1}.
 \end{align*}
 \begin{proof} We begin with a preliminary estimate for $\LL \Delta u$.
Returning to Lemma \ref{Llapu} and considering the terms in order, one first
observes by convexity of $\FF$ that
 \begin{align*}
 - k f^{\frac{k}{k-1}} u_{tt} (1+\ge)^k \FF^{ij,kl} \left[ \N_p (E_u)
\right]_{ij} \left[ \N_p (E_u) \right]_{kl} \geq 0.
 \end{align*}
 Also, by an application of the Cauchy Schwarz inequality one has the matrix
inequality
 \begin{align*}
 2  u_{tt}^{-2} \brs{\N u_{tt}}^2 \N_i u_t \N_j u_t - 4 u_{tt}^{-1} \N_p u_{tt}
\N_i \N_p u_t \N_j u_t + 2 \N_i \N_p u_{t} \N_j \N_p u_t \geq 0.
 \end{align*}
Also we observe
\begin{align*}
(1+\ge)^{k-1} k f^{\frac{k-1}{k}} \gD (f^{\frac{1}{k}}) \leq&\ C
f^{\frac{k-1}{k}} \left[ f^{\frac{1}{k} - 1} \gD f + f^{\frac{1}{k} - 2} \brs{\N
f}^2 \right] \leq C f.
\end{align*}
Next
\begin{align*}
-\frac{2}{k} (1+\ge)^{k-1} u_{tt}^{-1} \IP{\N f, \N u_{tt}} \leq&\ C f
u_{tt}^{-1} \brs{\N u_{tt}}\\
\leq&\ C \gd^{-1} f  + C \gd u_{tt}^{-2} \brs{\N u_{tt}}^2.
\end{align*}
Combining these estimates and choosing $\gd$ sufficiently small leads to the
preliminary estimate
\begin{gather} \label{lapuest10}
\begin{split}
\LL \gD u \geq&\ - 2 (1+\ge) u_{tt}^{2-k} \IP{T_{k-1}(E), \N_i \N_p u \N_j \N_p
u}\\
&\ + u_{tt}^{2-k} \IP{T_{k-1}(E), \brs{\N^2 u}^2 g + \mathcal
O(\brs{\N^2 u} + \brs{\N u}^2 + 1)} - C f.
\end{split}
\end{gather}
Similar considerations applied to Lemma \ref{Lgradu} yield
\begin{align} \label{lapuest20}
\LL \brs{\N u}^2 \geq&\ 2 \ge u_{tt}^{2-k} T_{k-1}(E)^{jk} \N_i \N_j u \N_i \N_k
u - C f - u_{tt}^{2-k} \IP{T_{k-1}(E), \OO(1)}.
\end{align}
Now fix a constant $\gl \in \mathbb R$ and consider
\begin{align*}
 \Phi = \gD u + \ge^{-1} \left[ (1+\ge) \brs{\N u}^2 + u_t^2 + \gl t(t-1)
\right]
\end{align*}
Combining Lemma \ref{Ltt} with lines (\ref{uttest20}), (\ref{lapuest10}), and
(\ref{lapuest20}) yields
\begin{align*}
 \LL \Phi \geq&\ u_{tt}^{2-k} \IP{T_{k-1}(E), \brs{\N^2 u}^2 g + \mathcal
O(\brs{\N^2 u} + \brs{\N u}^2 + 1) + \ge^{-1} \OO(1)}\\
&\ - C \ge^{-1} f + 2 \ge^{-1} f u_{tt} + \gl \ge^{-1} f u_{tt}^{-1}.
\end{align*}
First we observe that at a sufficiently large maximum of $\Phi$, the existing a
priori estimates imply that $\gD u$ is also large.  In particular, at a maximum
for $\Phi$ where $\brs{\N^2 u} \geq C \ge^{-\frac{1}{2}}$ we obtain
\begin{align*}
 \brs{\N^2 u}^2 g + \mathcal
O(\brs{\N^2 u} + \brs{\N u}^2 + 1) + \ge^{-1} \OO(1) \geq \frac{1}{2} \brs{\N^2
u}^2 g,
\end{align*}
and hence since $u$ is an admissible solution we have
 \begin{align*}
 u_{tt}^{2-k} \IP{T_{k-1}(E), \brs{\N^2 u}^2 g + \mathcal
O(\brs{\N^2 u} + \brs{\N u}^2 + 1) + \ge^{-1} \OO(1) } \geq&\ \frac{1}{2}
u_{tt}^{2-k}
\brs{\N^2 u}^2 \tr T_{k-1}(E) \geq 0.
\end{align*}
But then we can estimate
\begin{align*}
 C \ge^{-1} f \leq&\ \ge^{-1} f u_{tt} + C \ge^{-1} f u_{tt}^{-1}.
\end{align*}
hence choosing $\gl$ sufficiently large we obtain, at a sufficiently large
maximum for $\Phi$ which satisfies $\gD u \geq C \ge^{-\frac{1}{2}}$, one has
\begin{align*}
 \LL \Phi > 0,
\end{align*}
a contradiction.  The a priori estimate for $\gD u$ follows directly.
\end{proof}
\end{prop}

\subsection{Boundary estimates}

By Proposition \ref{utestimate} we already have the boundary estimate
\begin{align*}
\sup_{M \times \{ 0,1 \} } \left[ |u| + |u_{t}| + |\N u| \right] \leq C.
 \end{align*}
In this section we prove boundary estimates for second order derivatives:

\begin{prop} \label{C2bdyProp}  Given $u$ an admissible
solution to $(\star_{\ge,f})$, one has
 \begin{align*}
 \sup_{M \times \{ 0,1 \} } \left[ |u_{tt}| + |\N u_t| + |\N^2 u| \right] \leq C.
 \end{align*}
\end{prop}

\begin{proof} A bound for $|\N^2 u|$ is immediate.   If we can prove a bound for the `mixed' term $|\N u_t|$, then restricting the equation for $u$ to $t = 0$ we have
\begin{align*}
(1+\ge) u_{tt}(\cdot,0) \gs_k(A_{u(\cdot,0)}) &= \IP{T_{k-1}(A_{u(\cdot,0)}), \N u_t(\cdot,0)  \otimes \N  u_t(\cdot,0)} + f \\
&\leq C_1 \left( 1 + |\N u_0|^2 + |\N^2 u_0| \right) |\N u_t(\cdot,0)|^2 + C_2.
\end{align*}
Since $u_0$ is admissible,
\begin{align*}
\gs_k(A_{u(\cdot,0)}) = \gs_k(A_{u_0}) \geq \delta_0 > 0,
\end{align*}
and it follows that
\begin{align*}
\sup_M u_{tt}(\cdot,0) \leq C_0 ( 1 + \sup_{M}  |\N u_t(\cdot,0)|^2 ),
\end{align*}
where $C_0$ depends on the second-order spacial derivatives of $u_0$.  The same argument gives a corresponding bound for $u_{tt}(\cdot,1)$ in terms of the mixed derivative $|\N u_t(\cdot,1)|$.

To prove a bound on $\nabla u_t$ we consider the following auxiliary function
$\Psi : M \times [0,\tau] \rightarrow \mathbb{R}$, where $0 < \tau < 1$ will be chosen later:
\begin{align*}
\Psi = |\nabla(u-u_0)| + \left[ e^{\gl \left( u_0 - u + \gU
\right)}  - e^{\gl \gU} \right] + \gL t(t-1),
\end{align*}
where $\gl,\gL$ and $\gU$ are constants yet to be determined.  By making an appropriate choice of these constants, we claim that $\Psi$ attains a non-positive maximum on the boundary of of $M \times [0,\tau]$.  Assuming for
the moment this is true, let us see how a bound for $\N u_t$ follows.

Choose a point $x_0 \in M$, and a unit tangent vector $X \in T_{x_0}M$.  Let $\{ x^i \}$ be a local coordinate system
with $X = \frac{\partial}{\partial x^1}$ at $x_0$.  Then
\begin{align*}
&\frac{\partial}{\partial x^1}\big( u(x,t) - u_0(x) \big) + \left[ e^{\gl \left( u_0 - u + \gU
\right)}  - e^{\gl \gU} \right] + \gL t(t-1) \\
& \quad \leq  |\nabla(u-u_0)(x,t)| + \left[ e^{\gl \left( u_0 - u + \gU
\right)}  - e^{\gl \gU} \right] + \gL t(t-1) \\
& \quad \leq 0.
\end{align*}
Therefore,
\begin{align*}
0 &\geq \lim_{t \to 0+} \frac{1}{t} \Big\{ \frac{\partial}{\partial x^1}u(x,t) - \frac{\partial}{\partial x^1}u_0(x) + \left[ e^{\gl \left( u_0 - u + \gU
\right)}  - e^{\gl \gU} \right] + \gL t(t-1)  \Big\} \\
&= \frac{\partial}{\partial x^1}u_t(x_0,0) + \frac{1}{t} \left[ e^{\gl \left( u_0 - u + \gU
\right)}  - e^{\gl \gU} \right] + \gL (t-1).
\end{align*}
Since $u_t$ is bounded, an upper bound on $\frac{\partial}{\partial x^1}u_t$ follows.  Since $X = \frac{\partial }{\partial x^1}$ was arbitrary, we obtain a bound on $|\N u_t(x,0)|$.

To see that such a choice of $\gl,\gL, \gU$ and $\tau$ are possible, we first note that
\begin{align*}
\Psi(x,0) = 0.
\end{align*}
Since $|\N u|$ is bounded,
\begin{align*}
\Psi(x,\tau) &= |\nabla  u(x,\tau) - \nabla  u_0(x)| +  \left[ e^{\gl \left( u_0(x) - u(x,\tau) + \gU
\right)}  - e^{\gl \gU} \right] + \gL \tau (\tau-1)  \\
&\leq C_1 +  \big\vert e^{\gl \left( u_0(x) - u(x,\tau) + \gU
\right)}  - e^{\gl \gU} \big\vert + \gL \tau (\tau-1).
\end{align*}
Since $|u_t|$ is also bounded,
\begin{align*}
\big\vert e^{\gl \left( u_0(x) - u(x,\tau) + \gU
\right)}  - e^{\gl \gU} \big\vert \leq C_2 \lambda e^{ C_2 \lambda \tau + \gU},
\end{align*}
hence if $0 < \tau < 1/2$,
\begin{align*}
\Psi(x,\tau) &\leq C_1 + C_2 \tau  \lambda e^{ C_2 \lambda \tau + \gU}  -  \Lambda \tau (1-\tau) \\
&\leq C_1 +  \big( C_2   \lambda e^{ \frac{1}{2} C_2 \lambda  + \gU} - \Lambda/2 \big) \tau.
\end{align*}
Therefore, if $\Lambda$ is chosen large enough (depending on $\tau, C_1, C_2, \lambda$, and $\gU$), then
\begin{align*}
\Psi(x,\tau) \leq 0.
\end{align*}
We conclude that $\Psi \leq 0$ on $\partial \big( M \times [0,\tau] \big)$.

Assume the maximum of $\Psi$ is attained at a point $(x_0,t_0)$ which is interior (i.e., $0 < t_0 < \tau$).  Let
\begin{align*}
\eta = \frac{\displaystyle \N (u-u_0)(x_0,t_0)}{ \displaystyle | \N (u-u_0)(x_0,t_0)|}.
\end{align*}
We can extend $\eta$ locally via parallel transport along radial geodesics based at $x_0$. By construction,
\begin{align} \label{pareta} \begin{split}
\N \eta (x_0) &= 0, \\
|\N^2 \eta (x_0)| &\leq C(g).
\end{split}
\end{align}
By using a cut-off function, we can assume $\eta$ is globally defined and satisfies
\begin{align*}
|\eta| \leq 1,
\end{align*}
with $|\eta| = 1$ in a neighborhood of $x_0$.

Define
\begin{align*}
H = \eta^{\alpha} \N_{\alpha} (u-u_0)  + \left[ e^{\gl \left( u_0 - u + \gU
\right)}  - e^{\gl \gU} \right] + \gL t(t-1).
\end{align*}
Since $|\eta| \leq 1$,
\begin{align*}
H(x,t) \leq \Psi(x,t),
\end{align*}
and the max of $H$ is attained at $(x_0,t_0)$.  Therefore,
\begin{align*}
\mathcal{L} H (x_0,t_0) \leq 0.
\end{align*}

To compute $\mathcal{L} H (x_0,t_0)$, let $\phi = \eta^{\alpha} \N_{\alpha} (u-u_0).$  Using (\ref{pareta}), at $(x_0,t_0)$ we have
\begin{align*} \begin{split}
\phi_t &= \eta^{\alpha} \N_{\alpha} u_t, \\
\phi_{tt} &= \eta^{\alpha} \N_{\alpha} u_{tt}, \\
\nabla_k \phi_t &= \eta^{\alpha} \N_k \N_{\alpha} u_t.
\end{split}
\end{align*}
Also at $(x_0,t_0)$,
\begin{align*} \begin{split}
\N_k \phi &= \eta^{\alpha} \N_k \N_{\alpha} (u - u_0) = \eta^{\alpha} \N_k \N_{\alpha} u + O(1), \\
\N_k \N_{\ell} \phi &= \N_k \N_{\ell} \eta^{\alpha} \N_{\alpha} (u-u_0) + \eta^{\alpha} \N_k \N_{\ell} \N_{\alpha} (u-u_0 ) \\
&= \eta^{\alpha} \N_k \N_{\ell} \N_{\alpha} u + O(1).
\end{split}
\end{align*}
Therefore, by the formula in (\ref{linop}), at $(x_0,t_0)$ we have
\begin{align*}
\begin{split}
\LL & \phi =\ (1+\ge)^{k-1} u_{tt}^{-1} f \eta^{\alpha} \N_{\alpha} u_{tt} + u_{tt}^{1-k} T_{k-1}(E_u^{\ge})_{k \ell}  \Big\{  (1 + \ge) u_{tt} \big[  \eta^{\alpha} \N_k \N_{\ell} \N_{\alpha} u \\
 & \qquad + \eta^{\alpha} \N_k \N_{\alpha} u  \N_{\ell} u + \eta^{\alpha} \N_k u \N_{\ell} \N_{\alpha} u  - ( \eta^{\alpha}  \N_m \N_{\alpha} v \N_m u ) g_{k \ell} + O(1) g_{k \ell} \big] \\
&\ \qquad \qquad  -  \eta^{\alpha} \N_k \N_{\alpha} u_t \N_{\ell} u_t -  \eta^{\alpha} \N_{k} u_t \N_{\ell} \N_{\alpha} u_t + \frac{\eta^{\alpha} \N_{\alpha} u_{tt}}{ u_{tt}} \N_k u_t  \N_{\ell} u_t  \Big\} \\
&\ \geq \eta^{\alpha} \LL \N_{\alpha} u  - C  u_{tt}^{2-k} \tr T_{k-1} (E_u^{\ge}).
\end{split}
\end{align*}
Using the identity (\ref{DPhi}), we conclude
\begin{align*} \begin{split}
\LL \phi &\geq  \IP{ \N f, \eta}  - C u_{tt}^{2-k} \tr T_{k-1} (E_u^{\ge}) \\
&\geq - C f - C u_{tt}^{2-k} \tr T_{k-1} (E_u^{\ge}),
\end{split}
\end{align*}
where the constants depend on  $\max_M |\N f|/f.$

Next, we use Lemma \ref{Lucalc} to calculate
\begin{align} \label{Luu0} \begin{split}
\LL (u - u_0) =&\ (k+1) (1 + \epsilon)^{k-1} f + (1 + \epsilon) u^{2-k}_{tt} \IP{T_{k-1}(E_u^{\ge}),- A + \N u \otimes \N u - \frac{1}{2}|\N u|^2 g } \\
&\ \ - (1 + \epsilon) u_{tt}^{2-k} \IP{ T_{k-1}(E_u^{\ge}), \N^2 u_0 + \N u_0 \otimes \N u + \N u \otimes \N u_0 - \langle \N u_0, \N u \rangle g } \\
=&\ (k+1) (1 + \epsilon)^{k-1} f - (1 + \epsilon) u^{2-k}_{tt} \IP{T_{k-1}(E_u^{\ge}),A + \N^2 u_0} \\
&\ \  + (1 + \epsilon) u^{2-k}_{tt} \left[
\IP{T_{k-1}(E_u^{\ge}), \N u \otimes \N u} - \frac{1}{2} \tr T_{k-1}(E_u^{\ge}) \brs{\N u}^2
\right]  \\
&\ \ \  - (1 + \epsilon) u_{tt}^{2-k} \left[ 2
\IP{T_{k-1}(E_u^{\ge}), \N u \otimes \N u_0} - \tr T_{k-1}(E_u^{\ge}) \IP{\N u, \N u_0}
\right]\\
=&\ (k+1) (1 + \epsilon)^{k-1} f + (1 + \epsilon) u_{tt}^{2-k}  \IP{T_{k-1}(E_u^{\ge}), - A_{u_0} + \N {u_0} \otimes \N
u_0 - \frac{1}{2} \brs{\N u_0}^2 g}   \\
&\ \ + (1 + \epsilon) u_{tt}^{2-k} \left[ \IP{T_{k-1}(E_u^{\ge}), \N u \otimes \N u - 2 \N u \otimes \N u_0} + \tr
T_{k-1}(E_u^{\ge}) \left( - \frac{1}{2} \brs{\N u}^2 + \IP{\N u, \N u_0} \right)
\right] \\
=&  (k+1) (1 + \epsilon)^{k-1} f + (1+\epsilon) u_{tt}^{2-k} \big[ - \IP{T_{k-1}(E_u^{\ge}), A_{u_0}} + \IP{T_{k-1}(E_u^{\ge}), \N
(u - u_0) \otimes \N (u - u_0)} \\
& \ \ \ - \frac{1}{2} \tr T_{k-1}(E_u^{\ge}) \brs{\N (u - u_0)}^2 \big].
\end{split}
\end{align}
Taking $v = e^{\gl \left( u_0 - u + \gU \right)}  - e^{\gl \gU}$ in Lemma \ref{linearizedgeod}, we also have
\begin{align*}  \begin{split}
\LL   \left(   e^{\gl \left( u_0 - u + \gU \right)}  - e^{\gl \gU} \right) &= e^{\gl \left( u_0 - u + \gU \right)} \Big\{  (1+ \epsilon)^{k-1} f u_{tt}^{-1} \big[ - \lambda u_{tt} +  \lambda^2 u_t^2 \big] \\
 & + u_{tt}^{1-k}  \big\langle T_{k-1}(E^{\epsilon}_{u}), (1+\epsilon) u_{tt} \big[ \lambda \N^2 (u_0 - u ) + \lambda^2 \N (u_0 - u ) \otimes \N (u_0 - u) \\
 & + \lambda \N (u_0 - u) \otimes \N u + \lambda \N u \otimes \N (u_0 - u) - \lambda \langle \N(u_0 - u), \N u \rangle g \big] \\
  & + \lambda \N u_t \otimes \N u_t + \lambda^2 u_t \N(u_0 - u) \otimes \N u_t + \lambda^2 u_t \N u_t \otimes \N(u_0 - u) + \lambda^2 \dfrac{u_t^2}{u_{tt}} \N u_t \otimes \N u_t \Big\} \\
  &= - \lambda e^{\gl \left( u_0 - u + \gU \right)}  \LL (u - u_0) \\
  & \ \  + \lambda^2 e^{\gl \left( u_0 - u + \gU \right)} \Big\{ (1+\epsilon)^{k-1}f \frac{u_t^2}{u_{tt}} +
   u_{tt}^{2-k} \Big\langle T_{k-1}(E^{\epsilon}_u ), (1 + \epsilon) \N (u - u_0) \otimes \N (u - u_0)  \\
  & \ \ \ + \frac{u_t}{u_{tt}} \N(u_0 - u) \otimes \N u_t + \frac{u_t}{u_{tt}} \N u_t \otimes \N (u_0 - u) + \frac{u_t^2}{u_{tt}^2} \N u_t \otimes \N u_t \Big\rangle \Big\}
\end{split}
\end{align*}
We can estimate the term in braces as follows:
\begin{align*} \begin{split}
& (1+\epsilon)^{k-1}f \frac{u_t^2}{u_{tt}} + u_{tt}^{2-k} \Big\langle T_{k-1}(E^{\epsilon}_u ), (1 + \epsilon) \N (u - u_0) \otimes \N (u - u_0)  \\
  & \ \ \ + \frac{u_t}{u_{tt}} \N(u_0 - u) \otimes \N u_t + \frac{u_t}{u_{tt}} \N u_t \otimes \N (u_0 - u) + \frac{u_t^2}{u_{tt}^2} \N u_t \otimes \N u_t \Big\rangle \\
  &\geq (1+\epsilon)^{k-1}f \frac{u_t^2}{u_{tt}} + u_{tt}^{2-k} \Big\langle T_{k-1}(E^{\epsilon}_u ), \frac{(1+\epsilon)}{2} \N (u - u_0) \otimes \N (u - u_0) - \frac{u_t^2}{u_{tt}^2} \N u_t \otimes \N u_t \Big\rangle
\end{split}
\end{align*}
Using Lemma \ref{LAlemma} and the regularized equation, the final (negative) term above can be rewritten:
\begin{align*} \begin{split}
u_{tt}^{2-k} \big\langle T_{k-1}(E^{\epsilon}_u ), - \frac{u_t^2}{u_{tt}^2} \N u_t \otimes \N u_t \big\rangle &=  - u_{tt}^{-k} u_t^2 \big\langle T_{k-1}\big( (1+\epsilon) u_{tt} A_u - \N u_t \otimes \N u_t\big) ,  \N u_t \otimes \N u_t \big \rangle \\
&=  - u_{tt}^{-k} u_t^2 \big\langle T_{k-1}\big( (1+\epsilon) u_{tt} A_u  \big) ,  \N u_t \otimes \N u_t \big \rangle \\
&=  - (1 + \epsilon)^{k-1} u_{tt}^{-1} u_t^2 \big\langle T_{k-1}(A_u) ,  \N u_t \otimes \N u_t \big \rangle \\
&=  - (1 + \epsilon)^{k-1} u_{tt}^{-1} u_t^2 \big\{ (1+\epsilon) u_{tt} \sigma_k(A_u) - f \big\} \\
&= - (1+\epsilon)^k u_t^2 \sigma_k(A_u) + ( 1 + \epsilon)^{k-1} f \frac{u_t^2}{u_{tt}}.
\end{split}
\end{align*}
Therefore,
\begin{align} \label{Le3} \begin{split}
\LL   \left(   e^{\gl \left( u_0 - u + \gU \right)}  - e^{\gl \gU} \right) &\geq - \lambda e^{\gl \left( u_0 - u + \gU \right)}  \LL (u - u_0) + \lambda^2 e^{\gl \left( u_0 - u + \gU \right)} \Big\{ 2 (1+\epsilon)^{k-1}f \frac{u_t^2}{u_{tt}} \\
  &  \hskip-.25in - (1+\epsilon)^k u_t^2 \sigma_k(A_u) + u_{tt}^{2-k} \Big\langle T_{k-1}(E^{\epsilon}_u ), \frac{(1+\epsilon)}{2} \N (u - u_0) \otimes \N (u - u_0) \Big\rangle \Big\}.
  \end{split}
  \end{align}
  Also, by (\ref{Luu0}),
\begin{align} \label{Luz}  \begin{split}
-\lambda  \LL (u - u_0) &= - \lambda (k+1) (1 + \epsilon)^{k-1} f + \lambda (1+\epsilon) u_{tt}^{2-k} \IP{T_{k-1}(E_u^{\ge}), A_{u_0}} \\
& \hskip-.5in + u_{tt}^{2-k} \big\langle T_{k-1}(E_u^{\ge}), - \lambda (1+\epsilon) \N (u - u_0) \otimes \N (u - u_0) \big\rangle +  \frac{1}{2} (1+\epsilon) \lambda u_{tt}^{2-k} \tr T_{k-1}(E_u^{\ge}) \brs{\N (u - u_0)}^2.
\end{split}
\end{align}
Combining (\ref{Le3}) and (\ref{Luz}), we get
\begin{align*} \begin{split}
\LL   \left(   e^{\gl \left( u_0 - u + \gU \right)}  - e^{\gl \gU} \right) &\geq e^{\gl \left( u_0 - u + \gU \right)} \Big\{ - \lambda (k+1) (1 + \epsilon)^{k-1} f + 2 \lambda^2 (1+\epsilon)^{k-1}f \frac{u_t^2}{u_{tt}} - \lambda^2 (1+\epsilon)^k u_t^2 \sigma_k(A_u) \\
& + u_{tt}^{2-k} \Big\langle T_{k-1}(E^{\epsilon}_u ), (1+\epsilon) \big( \frac{1}{2} \lambda^2 - \lambda \big) \N (u - u_0) \otimes \N (u - u_0) \Big\rangle \\
& +  \frac{1}{2} (1+\epsilon) \lambda u_{tt}^{2-k} \tr T_{k-1}(E_u^{\ge}) \brs{\N (u - u_0)}^2 + \lambda (1+\epsilon) u_{tt}^{2-k} \IP{T_{k-1}(E_u^{\ge}), A_{u_0}}  \Big\}.
  \end{split}
  \end{align*}

Next, using Lemma \ref{Ltt}, we have
\begin{align*}
\LL \big( \Lambda t(1-t) \big) = 2 \Lambda (1+\epsilon)^k \sigma_k(A_u).
\end{align*}
Combing the above, we conclude that at an interior maximum of $H$,
\begin{align*} \begin{split}
\LL H &\geq - C f - C u_{tt}^{2-k} \tr T_{k-1} (E_u^{\ge}) + 2 \Lambda (1+\epsilon)^k \sigma_k(A_u) \\
&\ \  + e^{\gl \left( u_0 - u + \gU \right)} \Big\{ - \lambda (k+1) (1 + \epsilon)^{k-1} f + 2 \lambda^2 (1+\epsilon)^{k-1}f \frac{u_t^2}{u_{tt}} - \lambda^2 (1+\epsilon)^k u_t^2 \sigma_k(A_u) \\
&\ \ \ + u_{tt}^{2-k} \Big\langle T_{k-1}(E^{\epsilon}_u ), (1+\epsilon) \big( \frac{1}{2} \lambda^2 - \lambda \big) \N (u - u_0) \otimes \N (u - u_0) \Big\rangle \\
&\ \ \ +  \frac{1}{2} (1+\epsilon) \lambda u_{tt}^{2-k} \tr T_{k-1}(E_u^{\ge}) \brs{\N (u - u_0)}^2 + \lambda (1+\epsilon) u_{tt}^{2-k} \IP{T_{k-1}(E_u^{\ge}), A_{u_0}} \Big\}.
\end{split}
\end{align*}

Now note that since the cone $\gG_k^+$ is open and $M$ is compact there exists
$\gd > 0$ depending only on $u_0$ so that $A_{u_0} - \gd g \in \gG_k^+$.  It
follows from Lemma \ref{newtonprops} that
\begin{align*}
\gd \tr T_{k-1}(E_u^{\ge}) = \Sigma(E_u^{\ge},\dots,E_u^{\ge}, \gd g) <
\Sigma(E_u^{\ge},\dots,E_u^{\ge},A_{u_0}) = \IP{T_{k-1}(E_u^{\ge}), A_{u_0}}.
\end{align*}
Therefore, if $\lambda >> 2$ we have
\begin{align} \label{LH2} \begin{split}
\LL H &\geq \big\{ -C - \lambda (k+1) (1 + \epsilon)^{k-1} e^{\gl \left( u_0 - u + \gU \right)}  \big\} f \\
&\ \ + \big\{  2 \Lambda (1+\epsilon)^k  - \lambda^2 (1+\epsilon)^k u_t^2 e^{\gl \left( u_0 - u + \gU \right)} \big\} \sigma_k(A_u) + \big\{ - C + \lambda(1+\epsilon) \delta   \big\} u_{tt}^{2-k} \tr T_{k-1}(E_u^{\ge}).
\end{split}
\end{align}
Observe that by choosing $\lambda = \lambda(\delta)$ large enough, we can assume the last term in (\ref{LH2}) is bounded below by
\begin{align} \label{Trf}
\frac{\lambda}{2} \delta u_{tt}^{2-k} \tr T_{k-1}(E_u^{\ge}).
\end{align}
By the Newton-Maclaurin inequality,
\begin{align*}
u_{tt}^{2-k} \tr T_{k-1}(E_u^{\ge}) =&\ (k-1) u_{tt}^{2-k} \gs_{k-1}(E_u^{\ge})\\
\geq&\ (k-1) u_{tt}^{2-k} \gs_k(E_u^{\ge})^{\frac{k-1}{k}}\\
=&\ (k-1) f^{\frac{k-1}{k}} u_{tt}^{\frac{1}{k}}\\
\geq&\ C f u_{tt}^{\frac{1}{k}}.
\end{align*}
Combining this with (\ref{Trf}) and substituting into (\ref{LH2}), we get
\begin{align*} \begin{split}
\LL H &\geq \big\{ -C - \lambda (k+1) (1 + \epsilon)^{k-1} e^{\gl \left( u_0 - u + \gU \right)} + C \lambda \delta  u_{tt}^{\frac{1}{k}} \big\} f \\
&\ \ + \big\{  2 \Lambda (1+\epsilon)^k  - \lambda^2 (1+\epsilon)^k u_t^2 e^{\gl \left( u_0 - u + \gU \right)} \big\} \sigma_k(A_u).
\end{split}
\end{align*}
Let us fix the constant $\gU$ so that
\begin{align*}
0 \leq u_0 - u + \gU \leq C,
\end{align*}
then
\begin{align*}
\LL H \geq \big\{ -C - C \lambda (k+1)  + C \lambda \delta  u_{tt}^{\frac{1}{k}} \big\} f + \big\{  2 \Lambda (1+\epsilon)^k  - C \lambda^2  u_t^2   \big\} \sigma_k(A_u).
\end{align*}
Next, we assume $\Lambda = \Lambda(\lambda, \max u_t^2)$ is chosen large enough so that the coefficient of the second term above is
\begin{align*}
2 \Lambda (1+\epsilon)^k  - C \lambda^2  u_t^2  \geq \frac{1}{2}\lambda^2.
\end{align*}
By the regularized equation,
\begin{align*}
\sigma_k(A_u) \geq \dfrac{f}{(1+\epsilon) u_{tt}}.
\end{align*}
Therefore,
\begin{align*}
\LL H \geq \big\{ -C - C \lambda (k+1)  + C \lambda \delta  u_{tt}^{\frac{1}{k}} + \frac{1}{2(1+\epsilon)} \lambda^2 u_{tt}^{-1} \big\} f.
\end{align*}

If $u_{tt} > C(\delta)$ is large then the left-hand side is positive, which would be a contradiction at an interior maximum. On the other hand, if $u_{tt}$ is small then as long as
$\lambda$ is chosen large enough, the last term in the braces will dominate and once again we conclude $\LL H > 0$.  It follows that $H$ attains its maximum on the boundary, as
claimed.

\end{proof}

\subsection{Existence of approximate and regularizable geodesics}
In this subsection we use the a priori estimates of the previous subsections to establish the existence of weak geodesics.

\begin{thm} \label{approxgeodthm} Given $u_0, u_1 \in \gG_k^+$, there exists $f
\in C^{\infty}(M \times [0,1])$ with $f > 0$ and a smooth
solution $u(x,t,s,\ge) : M \times [0,1] \times
[0,1] \times (0,\ge_0] \to \mathbb R$ of $\mathcal{G}_{s f}^{\ge} (u_{\epsilon}) = 0$ such that
\begin{enumerate}
\item For each $\ge \in (0,\ge_0]$ $u_{\epsilon} =
u(\cdot,\cdot,\cdot,\epsilon)$ satisfies
\begin{align*}
u_{\epsilon}(x,0,s) = u_0(x), \qquad u_{\epsilon}(x,1,s) = u_1(x).
\end{align*}
\item There is a constant $C > 0$, independent of $\epsilon$, such that
\begin{align*}
|u_{\epsilon}| + |\N u_{\epsilon}| + |(u_{\epsilon})_t| + \ge \left\{ |\N^2
u_{\epsilon}| + |\N (u_{\epsilon})_t| + |(u_{\epsilon})_{tt}| \right\} \leq C.
\end{align*}
\end{enumerate}
\begin{proof} As the argument follows standard lines we provide only a sketch.  Fix some $0 < \ge_0 < 1$, then choose an arbitrary $0 < \ge < \ge_0$.  First we observe that it follows from (\cite{JeffVEstimates} Proposition 3) that the path $u_t := t u_1 + (1-t) u_0$ lies in $\gG_k^+$.  Moreover, there exists some constant $\gL$ for which $w_t := u_t + \gL t(t-1)$ satisfies $E_u^{\ge} \in \gG_k^+$.  Let $f := \Phi_{\ge}(w)$, and set
\begin{align*}
\mathcal{I} = \big\{ s \in [0,1]  :\ \exists u \in C^{4,\alpha} \cap \Gamma_k^+,\ u \mbox{ solves $(\star_{\epsilon,sf})$} \big\}.
\end{align*}
By construction, $1 \in \mathcal{I}$.

To verify that $\mathcal{I}$ is open, it suffices to study the linearized equation; i.e., given $\psi \in C^{\infty}(M \times [0,1])$, we need to solve for some $s \in \mathcal I$ then equation
\begin{align*}
\mathcal{L}_{u_{\ge}(\cdot,\cdot,s)} \varphi = \psi
\end{align*}
with $\varphi$ satisfying Dirichlet boundary conditions.  The solvability of this linear problem follows from \cite{GT}, Theorem 6.13.

We claim that $\mathcal{I}$ is closed: let $\{ u_i = u_{s_i} \}$ be a sequence of admissible solutions with $s_i \geq s_0$.  The preceding {\em a priori }
estimates imply there is a constant $C$ (independent of $\epsilon$) such that
\begin{align*}
|u_i| + |\N u_i| + |(u_i)_t| + \ge \left\{ |(u_i)_{tt}| + |(\N u_i)_t| + |\N^2 u_i| \right\} \leq C.
\end{align*}
To obtain higher order regularity, we need to verify the concavity of the operator.  Observe that the equation can be rewritten as
\begin{align*}
\gs_k^{\frac{1}{k}} \left( u_{tt}^{\frac{1-k}{k}} E_u^{\ge} \right) = f^{\frac{1}{k}}.
\end{align*}
Since $\gs_k^{\frac{1}{k}}$ is a concave operator, the equation is convex, and so by Evans-Krylov \cite{Evans} \cite{Krylov} we conclude there is a constant $C = C(\epsilon, f)$ such that
\begin{align*}
\| u_i \|_{C^{2,\alpha}} \leq C.
\end{align*}
Applying the Schauder estimates we obtain bounds on derivatives of all orders, and it follows that the set $\mathcal{I}$ is closed.  Since $\mathcal{I}$ is open, closed, and non-empty, it follows that $\mathcal{I} = [0,1]$.  The theorem follows.
\end{proof}
\end{thm}

\begin{defn} \label{eG} Given $u_0, u_1 \in \gG_k^+$, we say a one parameter
family of $C^{1,1}$ functions $u_{\ge}(x,t) : M \times [0,1] \to \mathbb R$ is
an \emph{$\ge$-geodesic from $u_0$ to $u_1$} if
\begin{align*}
u_{\epsilon}(x,0) = u_0(x), \qquad u_{\epsilon}(x,1,s) = u_1(x), \qquad
\mathcal{G}_{0}^{\ge}(u_{\epsilon}) = 0.
\end{align*}
We furthermore will say that it is a \emph{regularizable $\ge$-geodesic} if
there exists $f_0 \in
C^{\infty}(M \times [0,1])$ with $f_0 > 0$ and a smooth
function $u(x,t,s) : M \times [0,1] \times [0,1] \times
 \to \mathbb R$ with the following properties:

\noindent $(i)$ For each $s \in [0,1]$, $u(\cdot,\cdot,s)$ satisfies
\begin{align*}
u(x,0,s) = u_0(x), \qquad u(x,1,s) = u_1(x), \qquad
\mathcal{G}_{s f_0}(u) = 0.
\end{align*}

\noindent $(ii)$ There is a constant $C > 0$, independent of $\epsilon$, such
that
\begin{align*}
|u_{\epsilon}| + |\N u_{\epsilon}| + |(u_{\epsilon})_t| + \ge \left\{ |\N^2 u_{\epsilon}| +
|\N (u_{\epsilon})_t| + |(u_{\epsilon})_{tt}| \right\} \leq C.
\end{align*}

\noindent $(iii)$  One has that $u(x,t,s) \to u(x,t)$ in the weak $C^{1,1}$
topology as $s \to 0$.
\end{defn}

\begin{defn} \label{RG} Given $u_0, u_1 \in \gG_k^+$, we say a one parameter
family of $C^1$ functions $u(x,t)$ is
a \emph{regularizable geodesic from $u_0$ to $u_1$} if there exists $f_0 \in
C^{\infty}(M \times [0,1])$ with $f_0 > 0$ and a smooth
function $u(x,t,s,\ge) : M \times [0,1] \times [0,1] \times
[0,\ge_0] \to \mathbb R$ with the following properties:  \\

\noindent $(i)$ For each $\ge \in [0,\ge_0)$ $u_{\epsilon} =
u(\cdot,\cdot,\cdot,\epsilon)$ satisfies
\begin{align*}
u_{\epsilon}(x,0,s) = u_0(x), \qquad u_{\epsilon}(x,1,s) = u_1(x), \qquad
\mathcal{G}_{s f_0}(u_{\epsilon}) = 0.
\end{align*}

\noindent $(ii)$ There is a constant $C > 0$, independent of $\epsilon$, such
that
\begin{align*}
|u_{\epsilon}| + |\N u_{\epsilon}| + |(u_{\epsilon})_t| + \ge \left\{ |\N^2 u_{\epsilon}| +
|\N (u_{\epsilon})_t| + |(u_{\epsilon})_{tt}| \right\} \leq C.
\end{align*}

\noindent $(iii)$  For each $0 < \alpha < 1$, $u_{\epsilon} \rightarrow u$ in
$C^{0,\alpha}$ as $\ge,s \to 0$.
\end{defn}

We can now show existence and uniqueness of a regularizable geodesic connecting
any two points in $\gG^+$.  The key issue for uniqueness is a comparison lemma.

\begin{lemma} \label{compare} Suppose $u, \tilde{u} \in C^{\infty}$ are
admissible and satisfy
\begin{align*} \begin{split}
\mathcal{G}^{\ge}_{f_1}(u) &= 0, \\
\mathcal{G}^{\ge}_{f_2}(\tilde{u}) &= 0,
\end{split}
\end{align*}
where $f_1 \leq f_2$.  Assume further that on the boundary,
\begin{align*} \begin{split}
u(x,0) &= \tilde{u}(x,0), \\
u(x,1) &= \tilde{u}(x,1).
\end{split}
\end{align*}
Then on $M \times [0,1]$,
\begin{align*}
u(x,t) \geq \tilde{u}(x,t).
\end{align*}
\end{lemma}

We remark here also that the Lemma \ref{compare} can be used to exhibit
uniqueness for solutions of the equation $\mathcal{G}^{\ge}_0(u) = 0$.

\begin{cor} Given $u_0, u_1 \in \gG^+_k$, there exists a unique $\ge$-geodesic
from $u_0$ to $u_1$.
\begin{proof}Let $u(x,t,\ge)$ and $f$ be the data guaranteed by Theorem
\ref{approxgeodthm}.  Due to the a priori estimates, by Arzela-Ascoli there
exists a $C^{1,1}$ limit as $s \to 0$.  By definition this is an
$\ge$-geodesic.  Now suppose $\til{u}$ is another regularizable geodesic
connecting $u_0$ to $u_1$, with regularization $\til{u}(x,t,\ge)$ and auxiliary
function $\til{f}$.  Fixing some $\gd > 0$, for sufficiently small $\ge > 0$
Lemma \ref{compare} implies that $u(x,t,\ge) \geq \til{u}(x,t,\gd)$.  Since the
convergence is in $C^{0,\ga}$, sending $\ge \to 0$ yields $u(x,t) \geq
\til{u}(x,t,\gd)$.  We can now send $\gd \to 0$ to obtain $u(x,t) \geq
\til{u}(x,t)$.  Since the roles of $u$ and $\til{u}$ are interchangeable in that
argument, it follows that $u(x,t) = \til{u}(x,t)$.
\end{proof}
\end{cor}

\begin{cor} Given $u_0, u_1 \in \gG^+_k$, there exists a unique regularizable
geodesic from $u_0$ to $u_1$.
\begin{proof} Let $u(x,t,\ge)$ and $f$ be the data guaranteed by Theorem
\ref{approxgeodthm}.  Due to the a priori estimates, by Arzela-Ascoli there
exists a $C^{0,\ga}$ limit as both $\ge \to 0$ and $s \to 0$.  By definition
this is a
regularizable geodesic.  Now suppose $\til{u}$ is another regularizable geodesic
connecting $u_0$ to $u_1$, with regularization $\til{u}(x,t,\ge)$ and auxiliary
function $\til{f}$.  Fixing some $\gd > 0$, for sufficiently small $\ge > 0$
Lemma \ref{compare} implies that $u(x,t,\ge) \geq \til{u}(x,t,\gd)$.  Since the
convergence is in $C^{0,\ga}$, sending $\ge \to 0$ yields $u(x,t) \geq
\til{u}(x,t,\gd)$.  We can now send $\gd \to 0$ to obtain $u(x,t) \geq
\til{u}(x,t)$.  Since the roles of $u$ and $\til{u}$ are interchangeable in that
argument, it follows that $u(x,t) = \til{u}(x,t)$.
\end{proof}
\end{cor}

\section{Smoothing via Guan-Wang flow} \label{smoothingsec}

In this section we develop a sharper picture (Theorem \ref{gwsmoothing}) of the
short-time smoothing properties of a parabolic flow introduced by Guan-Wang in \cite{GW}.
This is used in the proof of Theorem \ref{unique} to smooth the approximate
geodesics so that we can take strong limits to obtain a curve of critical points
for $F$ connecting any two given critical points.

In first subsection we will derive a series of formulas for the evolution of various quantities.  Since we
will be quoting some of the formulas from the previous section, we will state these formulas for general
dimensions.  In the second subsection, where we derive some short-time estimates, we will specialize to the case
$n=4$ and $k=2$.

First, we recall the
definition of the flow introduced in \cite{GW}:
\begin{align} \label{gwflow}
\dt u =&\ \log \gs_k(g_u^{-1} A_u) - V_u^{-1} \int_M \log \gs_k(g_u^{-1} A_u) dV_g.
\end{align}
For technical simplicity we will instead study an unnormalized flow
\begin{align} \label{unnormgwflow} \begin{split}
\dt u =&\ \log \gs_k(g_u^{-1} A_u) \\
=&\ \log \gs_k( A_u) + 2ku.
\end{split}
\end{align}
As we will be able to control the size of $u$ along this flow, the renormalizing
term will only change $u$ by a controlled constant, and have no effect on the
estimates.

\subsection{Evolution equations}
We remark that when the dimension $n > 4$, Guan-Wang assumed the manifold was locally conformally flat.  For the evolutionary formulas
we are interested in this assumption will not be necessary.

\begin{defn} Given $u$ an admissible solution to (\ref{unnormgwflow}), define
\begin{align*}
 L f =&\ \gs_k(A_u)^{-1} \IP{T_{k-1}(A_u), \N^2 f + \N u \otimes \N f + \N f
\otimes \N u -
\IP{\N u, \N f} g},\\
H =&\ \dt - L.
\end{align*}
where the derivatives and inner products are with respect to $g$ (the fixed background metric).
\end{defn}

\begin{lemma} \label{gwHu} Let $u$ be a solution to (\ref{unnormgwflow}).  Then
\begin{align*}
H u =&\ \log \gs_k(A_u) - k + \gs_k(A_u)^{-1} \IP{T_{k-1}(A_u), A - \N u \otimes
\N u + \frac{1}{2} \brs{\N u}^2 g} + 2ku.
\end{align*}
\begin{proof} We directly compute
\begin{align*}
L u =&\ \gs_k(A_u)^{-1} \IP{T_{k-1}(A_u), \N^2 u + 2 \N u \otimes \N u -
\brs{\N u}^2 g} \\
=&\  \gs_k(A_u)^{-1} \IP{T_{k-1}(A_u), A_u - A + \N u \otimes \N u - \frac{1}{2}
\brs{\N u}^2 g} \\
=&\ k + \gs_k(A_u)^{-1} \IP{T_{k-1}(A_u), - A + \N u \otimes \N u - \frac{1}{2}
\brs{\N u}^2 g}.
\end{align*}
Combining this with (\ref{unnormgwflow}) yields the result.
\end{proof}
\end{lemma}

\begin{lemma} \label{gwHexpu} Let $u$ be a solution to (\ref{unnormgwflow}) and $\lambda \in \mathbb{R}$.
Then
\begin{align*}
H e^{\gl u} =&\ \gl e^{\gl u} \left[ \log \gs_k(A_u) + 2 k u  -  k +
\gs_k(A_u)^{-1}
\IP{T_{k-1}(A_u), A - (1 + \gl) \N u \otimes \N u + \frac{1}{2} \brs{\N u}^2 g}
\right].
\end{align*}
\begin{proof} Note
\begin{align*}
\frac{\partial}{\partial t}( e^{\lambda u}) = \lambda e^{\lambda u} \big( \log \sigma_k(A_u) + 2k u \big).
\end{align*}
Also,
\begin{align*}
L e^{\lambda u} &= \sigma_k(A_u)^{-1}\big\langle T_{k-1}(A_u), \lambda e^{\lambda u} \nabla^2 u + \lambda^2 e^{\lambda u} \nabla u \otimes \nabla u + 2 \lambda e^{\lambda u} \nabla u \otimes \N u
- \lambda e^{\lambda u} |\N u|^2 g \big\rangle  \\
&= \sigma_k(A_u)^{-1}\big\langle T_{k-1}(A_u), \lambda e^{\lambda u} \big[ A_u - A - \nabla u \otimes \N u + \frac{1}{2}|\N u|^2 g \big] +   \lambda (\lambda + 2) e^{\lambda u} \nabla u \otimes \N u
- \lambda e^{\lambda u} |\N u|^2 g \big\rangle \\
&= \lambda e^{\lambda u} \sigma_k(A_u)^{-1} \big\langle T_{k-1}(A_u), A_u - A + (\lambda + 1) \nabla u \otimes \N u - \frac{1}{2}|\N u|^2 g \big] \big\rangle  \\
&= \lambda e^{\lambda u} \sigma_k(A_u)^{-1} \big\langle T_{k-1}(A_u), - A + (\lambda + 1) \nabla u \otimes \N u - \frac{1}{2}|\N u|^2 g \big] \big\rangle + \lambda  k e^{\lambda u}.
\end{align*}
Therefore,
 \begin{align*}
H e^{\gl u} =&\ \frac{\partial}{\partial t}( e^{\lambda u})- L e^{\lambda u} \\
=&\  \gl e^{\gl u} \left[ \log \gs_k(A_u) + 2  k u  -   k +
\gs_k(A_u)^{-1}
\IP{T_{k-1}(A_u), A - (1 + \gl) \N u \otimes \N u + \frac{1}{2} \brs{\N u}^2 g}
\right].
\end{align*}
\end{proof}
\end{lemma}

\begin{lemma} \label{gwflowgradu} Given $u$ a solution to (\ref{unnormgwflow}),
one has
\begin{align*}
H \brs{\N u}^2 = 2 \gs_k(A_u)^{-1}  T_{k-1}(A_u)^{pq}\big\{ - \N_i \N_p u \N_i
\N_q u +  \OO(\brs{\N u}^2 + 1) \big\} + 4k |\N u|^2.
\end{align*}
\begin{proof} We compute
\begin{align*}
 \dt \N_i u =&\ \N_i \log \gs_k(A_u) + 2 k \N_i u \\
 =&\ \gs_k(A_u)^{-1} \IP{T_{k-1}(A_u), \N_i A_u} + 2 k \N_i u \\
 =&\ \gs_k(A_u)^{-1} T_{k-1}(A_u)^{pq} \big\{  \N_i A_{pq} + \N_i \N_p \N_q u + 2 \N_i \N_p u \N_q u - \N_i \N_j u \N_j u g_{pq} \big\} + 2 k \N_i u \\
=&\ \gs_k(A_u)^{-1} T_{k-1}(A_u)^{pq} \big\{ \N_p \N_q \N_i u + 2 \N_i \N_p u \N_q u - \N_i \N_j u \N_j u g_{pq} + (\N A + \Rm * \N u)_{ipq} \big\} \\
& \ \ \ + 2 k \N_i u,
\end{align*}
hence
\begin{align*}
\frac{\partial}{\partial t}|\N u|^2 &= 2 \gs_k(A_u)^{-1} T_{k-1}(A_u)^{pq} \big\{  \N_p \N_q \N_i u \N_i u + 2 \N_i \N_p u \N_q u \N_i u - \N_i \N_j u \N_j u \N_i u  g_{pq} \\
 & \ \ + \left[ (\N A + \Rm * \N u) * \N u \right]_{pq} \big\} + 4 k |\N u|^2.
\end{align*}
Also,
\begin{align*}
L |\N u|^2 &= 2 \gs_k(A_u)^{-1} T_{k-1}(A_u)^{pq} \big\{  \N_p \N_i u \N_q \N_i u + \N_p \N_q \N_i u \N_i u + 2 \N_i \N_p u \N_q u \N_i u - \N_i \N_j u \N_j u \N_i u  g_{pq} \big\}.
\end{align*}
It follows that
\begin{align*}
 \dt \brs{\N u}^2 =&\ L \brs{\N u}^2 + 2 \gs_k(A_u)^{-1}  T_{k-1}(A_u)^{pq}\big\{  -\N_p \N_i u \N_q \N_i u  + \left[ (\N A + \Rm * \N u) * \N u \right]_{pq} \big\} + 4 k |\N u|^2,
\end{align*}
which implies the result.
\end{proof}
\end{lemma}

\begin{cor} \label{gwflowgraduCor} Given $u$ a solution to (\ref{unnormgwflow}),
one has
\begin{align} \label{gflowCorsum}
H \big( e^{-4kt} \brs{\N u}^2\big)  = 2 e^{-4kt} \gs_k(A_u)^{-1}  T_{k-1}(A_u)^{pq}\big\{ - \N_i \N_p u \N_i
\N_q u +  \OO(\brs{\N u}^2 + 1) \big\}.
\end{align}
\end{cor}

For the following lemma, for an $n \times n$ symmetric matrix $r = r_{ij}$ we denote
\begin{align*}
\FF(r) = \log \gs_k(r),
\end{align*}
and derivatives of $\FF$ with respect to the entries of $r$ by
\begin{align*}
\dfrac{\partial}{\partial r_{pq}}\FF(r) &= \FF(r)^{pq}, \\
 \dfrac{\partial^2}{\partial r_{pq} \partial r_{rs}}\FF(r) &= \FF(r)^{pq, rs}.
\end{align*}

\begin{lemma} \label{gwflowlapu} Given $u$ a solution to (\ref{unnormgwflow}),
one has
 \begin{align*}
  H \gD u =&\ \FF^{pq,rs} \N_i (A_u)_{pq} \N_i (A_u)_{rs}\\
  &\ + \gs_k(A_u)^{-1} \left<T_{k-1}(A_u)_{ij}, 2 \N_i \N_p u \N_j \N_p u -
\brs{\N^2 u}^2 g_{ij} + \OO(\brs{\N^2 u} + \brs{\N u}^2 + 1) \right>.
 \end{align*}
\begin{proof} We compute
\begin{align*}
\gD \log \gs_k(A_u) =&\ \N_i \left[ \FF^{pq}  \N_i (A_u)_{pq} \right] =  \FF^{pq,rs}
\N_i (A_u)_{pq} \N_i (A_u)_{rs} + \FF^{pq} \gD (A_u)_{pq}.
\end{align*}
Combining this with our prior calculation of $\gD A_u$ (\ref{lapA}) yields
\begin{align*}
 \dt \gD u =&\ \gD \log \gs_k(A_u) + 2k \Delta u  \\
 =&\ \FF^{pq,rs} \N_i (A_u)_{pq} \N_i (A_u)_{rs} + \gs_k(A_u)^{-1}
T_{k-1}(A_u)^{pq} (\gD A_u)_{pq} + 2k \Delta u \\
 =&\ \FF^{pq,rs} \N_i (A_u)_{pq} \N_i (A_u)_{rs} + \gs_k(A_u)^{-1}
 T_{k-1}(A_u)^{pq} \big\{  \N_p \N_q (\gD u) + \N_p \gD u \N_q u + \N_p u
\N_q \gD u \\
&\  + 2 \N_p \N_{\ell} u \N_q \N_{\ell} u - \brs{\N^2 u}^2 g_{pq} - \IP{\N u, \N \gD
u} g_{pq} + \OO(\brs{\N^2 u} + \brs{\N u}^2 + 1) \big\} + 2k \Delta u \\
=&\ L(\Delta u) + \FF^{pq,rs} \N_i (A_u)_{pq} \N_i (A_u)_{rs}  \\
& \ \ + \gs_k(A_u)^{-1} T_{k-1}(A_u)^{pq} \big\{ 2 \N_p \N_{\ell} u \N_q \N_{\ell} u - \brs{\N^2 u}^2 g_{pq} + \OO(\brs{\N^2 u} + \brs{\N u}^2 + 1) \big\},
\end{align*}
and the result follows.
\end{proof}
\end{lemma}

\subsection{Estimates}
In this section we specialize to the case $n=4$ and $k=2$, and use the evolutionary formulas from the
preceding subsection to derive some short-time smoothing estimates.

\begin{lemma} \label{gwFdecrease} Given $u$ a solution to (\ref{gwflow}), one
has
\begin{align*}
 \frac{d}{dt} F[u] \leq 0.
\end{align*}
\begin{proof} This is immediate from the flow equation (\ref{gwflow}) and the formula (\ref{Fdot}).
\end{proof}
\end{lemma}

\begin{prop} \label{gwflowest10} Let u be a solution to (\ref{gwflow}) with initial value $u(\cdot,0) = u_0$, where $u_0$ is admissible. Then
there are constants $C_1 = C_1(g), \ge = \ge(\brs{u_0}_{C^0})$, such that $u$ exists for all $0 \leq t \leq \ge$, and
\begin{align*}
\brs{u}_{C^0} \leq C_1 ( 1 + \brs{u_0}_{C^0}).
\end{align*}
for all $0 \leq t \leq \ge$.

\begin{proof} At a maximum for $u$, one has $A_u \leq A$, and hence $\gs_2(A_u) \leq
\gs_2(A) < C$.  By (\ref{unnormgwflow}),
\begin{align*}
\frac{d}{dt} \max u \leq C + 4 \max u.
\end{align*}
Integrating this inequality we get an upper bound for $u$.  Applying a similar argument at a minimum of $u$, we
obtain a lower bound.
\end{proof}
\end{prop}

\begin{prop} \label{gwflowest40} Given $u$ as in the previous proposition,
there exists constants $C_1$ and $\ge$ depending on $\brs{u_0}_{C^0}$ such that for all $0 \leq t \leq \ge \leq 1$, one has
\begin{align*}
\brs{\N u}_{C^0} \leq C_1(\brs{u_0}, \brs{\N u_0}).
\end{align*}
\begin{proof} Let
\begin{align*}
\Phi = e^{-8t} \brs{\N u}^2 + \gL e^{- 2 u} - \mu t,
\end{align*}
where $\Lambda, \mu > 0$ will be specified later. Combining Corollary \ref{gwflowgraduCor} and Lemma \ref{gwHexpu}, and using the fact that at a maximum of $\Phi$ we have $H \Phi \geq 0$, it follows  
\begin{align*}\begin{split}
0 \leq H \Phi &= 2 \sigma_2(A_u)^{-1} T_{1}(A_u)^{pq}\big\{  - e^{-8t} \N_i \N_p u \N_i \N_q u + e^{-8t} \OO(1 + |\N u|^2) - \Lambda e^{-2u} \big[  \N_p u \N_q u + \frac{1}{2}|\N u|^2 g_{pq} \big] \big\}
 \\
 & \ \ - 2 \Lambda e^{-2u} \big[ \log \sigma_2(A_u) + 4u - 2 + \sigma_2(A_u)^{-1}\langle T_{1}(A_u),A\rangle \big]  - \mu \\
 &= I_1 + I_2 - \mu.
\end{split}
\end{align*}
We can estimate the terms in braces in $I_1$ by
\begin{align*}
- e^{-8t} \N_i \N_p u \N_i \N_q u + e^{-8t} \OO(1 + |\N u|^2) - \Lambda e^{-2u} \big[  \N_p u \N_q u + \frac{1}{2}|\N u|^2 g_{pq} \big] \leq \big\{ C + ( C  - \frac{\gL}{2} e^{-2u}) \brs{\N u}^2 \big\}g_{pq}.
\end{align*}
By Proposition \ref{gwflowest10}, for $0 \leq t \leq \ge \leq 1$ we have a uniform bound on $|u|$ depending only on the initial data, hence if $\Lambda >> 1$ is chosen large enough,
\begin{align*}
C + ( C  - \frac{\gL}{2} e^{-2u}) \brs{\N u}^2  \leq C -  \brs{\N u}^2.
\end{align*}
If $|\N u|$ remains uniformly bounded we have nothing to prove, so we may assume that at the maximum of $\Phi$ the gradient of $u$ is large, hence at a maximum of $\Phi$ we have
\begin{align*}
I_1 \leq 0.
\end{align*}
To estimate $I_2$, we first consider the case where $\sigma_2(A_u) \geq 1$.  Then $\log \sigma_2(A_u) \geq 0$ and the remaining terms
in brackets are either bounded or non-negative, hence
\begin{align} \label{mud} \begin{split}
I_2 - \mu &\leq C(\Lambda, \max |u|) - \mu \\
&\leq 0,
\end{split}
\end{align}
if $\mu$ is chosen large enough.  On the other hand, using Lemma \ref{convexityinequality} we see that
\begin{align*}
\sigma_2(A_u)^{-1} \langle T_1(A_u), A \rangle \geq \sigma_2(A_u)^{-1}\sigma_2(A_u)^{\frac{1}{2}}\sigma_2(A)^{\frac{1}{2}} = \frac{\sigma_2(A)^{\frac{1}{2}}}{\sigma_2(A_u)^{\frac{1}{2}}}.
\end{align*}
It follows there is a small constant $\delta = \delta(\sigma_2(A))$ such that if $0 < \sigma_2(A_u) \leq \delta$, then
\begin{align*}
\log \sigma_2(A_u) + \sigma_2(A_u)^{-1} \langle T_1 (A_u) , A \rangle \geq 0.
\end{align*}
Then arguing as we did in the case where $\sigma_2(A_u) \geq 1$, we can choose $\mu$ large enough to achieve (\ref{mud}) again.  Finally, in the intermediate range $\delta \leq \sigma_2(A_u) \leq 1$, all the terms in the brackets in $I_2$ are bounded are non-positive, and we again conclude that (\ref{mud}) holds once $\mu$ is chosen large enough.  It follows
that $H \Phi \leq 0$, and the result follows from the maximum principle.
\end{proof}
\end{prop}

\begin{prop} \label{gwflowest30} Suppose $u$ is a solution to
(\ref{unnormgwflow}) with $n=4$ on
$[0,T], T \leq 1$, such that
\begin{align} \label{Nbig}
\sup_{M \times [0,T]} \brs{u} \leq N.
\end{align}
There exists a constant $C = C(\gL)$ such that for all $t \in [0,T]$, one has
\begin{align*}
t \brs{\log \gs_2(A_u)} \leq C.
\end{align*}
\begin{proof} First, note that
\begin{align*}
\frac{\partial}{\partial t}\log \sigma_2(A_u) &= \sigma_2(A_u)^{-1} \langle T_1(A_u), \frac{\partial}{\partial t} A_u \rangle \\
&= \sigma_2(A_u)^{-1} \Big \langle T_1(A_u), \nabla^2 \log \sigma_2(A_u) + \N u \otimes \N \log \sigma_2(A_u) + \N \log \sigma_2(A_u) \otimes \N u \\
& \ \ \ - \langle \N u, \N \log \sigma_2(A_u) \rangle g + 4 \N^2 u + 8 \N u \otimes \N u - 4|\N u|^2 g \Big\rangle \\
&= L (\log \sigma_2(A_u)) + 4 \sigma_2(A_u)^{-1} \big \langle T_1(A_u), \N^2 u + 2 \N u \otimes \N u - |\N u|^2 g \big\rangle  \\
&=  L (\log \sigma_2(A_u)) + 4 \sigma_2(A_u)^{-1} \big \langle T_1(A_u), A_u - A + \N u \otimes \N u - \frac{1}{2} |\N u|^2 g \big\rangle  \\
&=  L (\log \sigma_2(A_u)) + 8 + 4 \sigma_2(A_u)^{-1} \big \langle T_1(A_u), - A + \N u \otimes \N u - \frac{1}{2} |\N u|^2 g \big\rangle,  \\
\end{align*}
hence
\begin{align} \label{HLog}
H (\log \sigma_2(A_u)) =  8 + 4 \sigma_2(A_u)^{-1} \big \langle T_1(A_u), - A + \N u \otimes \N u - \frac{1}{2} |\N u|^2 g \big\rangle.
\end{align}
Set
\begin{align*}
 \Phi :=  t \log \gs_2(A_u) + \Lambda e^{- 2 u} - \mu t.
\end{align*}
We will show that by choosing $\Lambda, \mu > > 1$ sufficiently large (depending on $N$), $H \Phi \leq 0$.  This will give an upper bound on $\Phi$ depending only
 on the initial $C^0$-norm of $u$.

 To begin, we combine (\ref{HLog}) with \ref{gwHexpu} to get
\begin{align} \label{Log1} \begin{split}
H \Phi &= - \mu + 8 t + 4 \Lambda ( 1 - 2u) e^{-2u} + (1 - 2 \Lambda e^{-2u}) \log \sigma_2(A_u) \\
& \ \ + \sigma_2(A_u)^{-1} \big \langle T_1(A_u), - ( 4t + 2 \Lambda e^{-2u}) A
+ ( 4t - 2 \Lambda e^{-2u}) \N u \otimes \N u - (2t +  \Lambda e^{-2u}) |\N u|^2 g \big\rangle.
\end{split}
\end{align}
By choosing $\Lambda$ large enough (depending on the constant $N$ in (\ref{Nbig})) we may assume the coefficient of the log-term
 \begin{align} \label{Lmin1}
 1 - 2 \Lambda e^{-2u} \leq -1.
 \end{align}
For $t$ small (depending on $N$ and $\Lambda$) the coefficients of the gradient terms in (\ref{Log1}) are also non-positive, so we have
\begin{align} \label{Log1a} \begin{split}
H \Phi &\leq - \mu + 8 t + 4 \Lambda ( 1 - 2u) e^{-2u} + (1 - 2 \Lambda e^{-2u}) \log \sigma_2(A_u)  \\
& \ \ \ - ( 4t + 2 \Lambda  e^{-2u})\sigma_2(A_u)^{-1} \big \langle T_1(A_u),  A \big\rangle.
\end{split}
\end{align}
If $\mu >> 1$ is chosen large enough, the first three terms
on the RHS of (\ref{Log1}) can be bounded above by $-\mu/2$, and we conclude
\begin{align} \label{Log2}
H \Phi \leq - \mu/2 + (1 - 2 \Lambda e^{-2u}) \log \sigma_2(A_u) - ( 4t + 2 \Lambda  e^{-2u})\sigma_2(A_u)^{-1} \big \langle T_1(A_u),  A \big\rangle.
\end{align}
By Lemma \ref{convexityinequality} we have
\begin{align*}
 \gs_2(A_u)^{-1} \IP{T_1(A_u),A} \geq \gs_2(A_u)^{-1} \left[ 4
\gs_2(A_u)^{\frac{1}{2}} \gs_2(A)^{\frac{1}{2}} \right] \geq \gd
\gs_2(A_u)^{-\frac{1}{2}} > 0,
\end{align*}
hence
\begin{align} \label{CVLog}
- ( 4t + 2 \Lambda  e^{-2u})\sigma_2(A_u)^{-1} \big \langle T_1(A_u),  A \big\rangle \leq -C_1 \gs_2(A_u)^{-\frac{1}{2}}.
\end{align}

If $\sigma_2(A_u) \geq 1$, it follows from (\ref{Lmin1}), (\ref{Log2}), and (\ref{CVLog}) that $H \Phi \leq 0$. On the other hand, if $\sigma_2(A_u) < 1$, then
\begin{align*}
H \Phi \leq - \mu/2 - \log \sigma_2(A_u) - C_1 \gs_2(A_u)^{-\frac{1}{2}},
\end{align*}
and by choosing $\mu >> 1$ large enough (depending only on $C_1$) once again we have $H \Phi \leq 0$.

To obtain a lower bound for $\log \sigma_2(A_u)$, we consider
\begin{align*}
\tilde{\Phi} :=  - t \log \gs_2(A_u) + \Lambda e^{- 2 u} - \mu t,
\end{align*}
and apply a similar argument.  We will omit the details.
\end{proof}
\end{prop}

\begin{prop} \label{gwflowest20} Suppose $u$ is a solution to
(\ref{unnormgwflow}) with $n=4$ on
$[0,T], T \leq 1$, such that
\begin{align*}
\sup_{M \times [0,T]} \left\{ \brs{\N u}^2 + \brs{u} \right\} \leq A.
\end{align*}
There exists a constant $C = C(A)$ such that for all $t \in [0,T]$, one has
\begin{align*}
t \gD u \leq C.
\end{align*}
\begin{proof} Let
\begin{align*}
\Phi = t \gD u + \brs{\N u}^2,
\end{align*}
where $\Lambda >> 1$ will be chosen later.
A direct calculation using Lemmas \ref{gwflowgradu} and \ref{gwflowlapu} and
some elementary estimates yields
\begin{align} \label{Hdel} \begin{split}
H \Phi =&\ \gD u + t \FF^{pq,rs} \N_i (A_u)_{pq} \N_i (A_u)_{rs}\\
  &\ + \gs_2(A_u)^{-1} T_1(A_u)^{pq} \big\{ 2 \left(t - 1 \right) \N_i
\N_p u \N_i \N_q u - t \brs{\N^2 u}^2 g_{pq} + \OO( t \brs{\N^2 u} +\brs{\N u}^2
+ 1)\}.
\end{split}
\end{align}
If $\Phi$ attains a large space-time maximum, say $\Phi \geq B \geq 2A$, then
\begin{align*}
t \Delta u \geq B - A \geq \frac{1}{2} B,
\end{align*}
hence
\begin{align*}
t |\N^2 u|^2 \geq \frac{B^2}{16t}.
\end{align*}
Therefore, if $t \leq 1$, the terms in braces in (\ref{Hdel}) can be estimated as
\begin{align*}
2 \left(t - 1 \right) \N_i
\N_p u \N_i \N_q u - t \brs{\N^2 u}^2 g_{pq} + & \OO( t \brs{\N^2 u} +\brs{\N u}^2
+ 1) \leq  \big\{ - t \brs{\N^2 u}^2  + C t |\N^2 u| + C(A) \big\} g _{pq}  \\
&\leq  \big\{ - \frac{t}{2} \brs{\N^2 u}^2  +  C' \big\} g _{pq}   \\
&\leq \big\{  - \frac{B^2}{32t} + C' \big\} g_{pg} \\
&\leq 0,
\end{align*}
if $B$ is large enough.  Thus we conclude $H \Phi < 0$ at a sufficiently large maximum, proving the
result.
\end{proof}
\end{prop}

\begin{thm} \label{gwsmoothing} Let $(M^4, g)$ be a compact Riemannian manifold
such that $g \in \gG_2^+$.  Given $u_0 \in \gG_2^+$ there exists $\ge =
\ge(\brs{u_0}, \brs{\N u_0})$ and $C = C(\brs{u_0}, \brs{\N u_0})$ such that the
solution to (\ref{gwflow}) with initial condition $u_0$ exists on $[0,\ge]$ and
moreover satisfies
\begin{align} \label{gwsmoothing10}
-C \leq \gD u \leq \frac{C}{t}, \qquad -C \leq t \log \gs_2(A_u) \leq C.
\end{align}
Furthermore, choosing $l \in \mathbb N, 0 < \ga < 1$ there exists $C_2 =
C(\brs{u_0},\brs{\N u_0},l,\ga)$ such that
\begin{align*}
\brs{u_{\ge}}_{C^{l,\ga}} \leq C.
\end{align*}
\begin{proof} The equation (\ref{unnormgwflow}) is strictly parabolic for $u_0
\in \gG_2^+$, and so there exists a solution on some small time interval
$[0,\eta)$.  By Propositions \ref{gwflowest10} and \ref{gwflowest40}, as long as
the solution exists there is a uniform upper bound on $\brs{u}_{C^1}$ on
$[0,\ge]$ where $\ge$ depends only on $\brs{u_0}_{C^1}$.  The estimates of
(\ref{gwsmoothing10}) follow from Propositions \ref{gwflowest20} and
\ref{gwflowest30}.  Given these it follows that equation (\ref{unnormgwflow}) is
uniformly parabolic on $[0,\ge]$, and hence by the Evans-Krylov estimates
\cite{Evans,Krylov} there is a uniform $C^{2,\ga}$ estimate for $u$ on
$[0,\ge]$.  Schauder
estimates now imply that for any $l,\ga$ there are uniform $C^{l,\ga}$
bounds on $u$ on $[0,\ge]$, which in particular proves that the solution
actually exists for this whole time interval as well. Given these estimates, one
relates the solution to (\ref{unnormgwflow}) to the solution to (\ref{gwflow})
by adding a time dependent constant to $u$ which fixed the volume to be
$V_{u_0}$.  Since $u$ is a priori bounded and this has no effect on any of the
derivative estimates the result follows.
\end{proof}
\end{thm}

\section{\texorpdfstring{Uniqueness of solutions to $\gs_2$-Yamabe
problem}{Uniqueness of solutions to sigma-2-Yamabe problem}} \label{uniquesec}

	In this section we combine the previous results to establish Theorem
\ref{unique}.  As described in the introduction, the proof consists of a few
main steps.  In particular, we use Theorem \ref{approxgeodthm} to connect any
two critical points for $F$ by an $\ge$-geodesic.  Applying the geodesic
convexity of $F$ we obtain that the curve must consist of near-minimizers for
$F$.  We then smooth this approximate geodesic via Theorem \ref{gwsmoothing}.
Taking the limit as $\ge \to 0$ of these smoothed paths yields a nontrivial
one-parameter family of minimizers of $F$.  Using our knowledge of the geodesic
convexity of $F$ we can show that this can only happen if the background
conformal class is $[g_{S^4}]$, and the endpoints of the path are round metrics.
 Note that, unlike the K\"ahler setting, we are unable to show that the
approximate geodesics converge directly to a nontrivial smooth geodesic due to
the lack of stronger regularity results for the geodesics.

 \begin{lemma} \label{uniquelemma1} Given $u_0,u_1$ two admissible critical points of $F$,
one has $F[u_0] = F[u_1]$, and $F[u] \geq F[u_0]$ for all admissible $u$.
Moreover, given $f$ and $u = u(x,t,s,\ge)$ the
approximate geodesics given by Theorem \ref{approxgeodthm}, one has for any $t
\in [0,1]$,
\begin{align*}
\lim_{s,\ge \to 0} F[u(\cdot,t,s,\ge)] = F[u_0].
\end{align*}
 \begin{proof} Fix $f$, and let $u = u(x,t,s,\ge)$ be the approximate geodesics guaranteed by Theorem \ref{approxgeodthm}, connecting $u_0$ and $u_1$.  To begin
we repeat the
calculation of
Proposition \ref{CYgeodconv} for these paths.  Fix some $s,\ge$ and
compute:
\begin{gather*}
\begin{split}
 \frac{d^2}{dt^2} F[u] =&\ \frac{d}{dt} \int_M u_t \left[ - \gs_2(g_u^{-1} A_u) +
\bar{\gs} \right] dV_u\\
 =&\ - \int_M \left[ u_{tt} \gs_2(g_u^{-1} A_u) + u_t \IP{T_1(g_u^{-1} A_u), \N^2 u_t} \right]
dV_u\\
 &\ + \gs \int_M \left[ u_{tt} V_u^{-1} + V_u^{-2} u_t \left( \int_M 4 u_t dV_u
\right) - 4 V_u^{-1} u_t^2 \right] dV_u\\
=&\ \int_M \left[ \ge u_{tt} \gs_2(g_u^{-1} A_u) - s f \right] dV_u + \gs V_u^{-1} \int_M
\left[ \frac{1}{\gs_2(g_u^{-1} A_u)} s f - \ge u_{tt} \right] dV_u\\
&\ + \gs V_u^{-1} \int_M \left[ \frac{1}{\gs_2(A)} \IP{T_1(g_u^{-1} A_u), \N u_t
\otimes \N u_t} - 4 \left( \int_M u_t^2 dV_u - V_u^{-1} \left( \int_M u_t dV_u
\right)^2 \right) \right] dV_u.
\end{split}
\end{gather*}
Applying Corollary \ref{n4andrews} to the above equation yields
\begin{align} \label{uniquelemma120}
 \frac{d^2}{dt^2} F \geq&\ - \int_M s f dV_u - \gs V_u^{-1} \ge \int_M u_{tt}.
\end{align}
Now let us estimate using the uniform $C^1$ estimate
\begin{align*}
\int_0^1 \int_M u_{tt} dV_u =&\ \int_0^1 \left[ \frac{\del}{\del t} \int_M u_t
dV_u - \int_M 4 u_t^2 \right]dt\\
=&\ \left.\int_M u_t dV_u \right|_{t=0}^{t=1} - \int_0^1 \int_M 4 u_t^2 dV_u dt\\
\leq&\ C.
\end{align*}
Hence, integrating the inequality (\ref{uniquelemma120}) and using that $u_0$ is a critical point yields
\begin{align*}
\frac{d}{dt} F[u] (t) =&\ \frac{d}{dt} F[u] (t) - \frac{d}{dt} F[u] (0) = \int_0^t \frac{d^2}{dt^2} F dt \geq - C (s + \ge).
\end{align*}
Integrating this in time and sending $s,\ge \to 0$ yields
\begin{align*}
 F[u_1] \geq F[u_0].
\end{align*}
But since the roles of $u_0$ and $u_1$ are interchangeable, we obtain $F[u_0] =
F[u_1]$.
\end{proof}
\end{lemma}

\begin{lemma} \label{Fisolated} Fix $(M^4, g)$ with $A_g \in \gG_2^+$, and
suppose $u \in C^{\infty}(M)$ is an admissible critical point of $F$.  Then either $u$ is an isolated critical point for $F$ or $(M^4,
g_u)$ is isometric to $(S^4, g_{S^4})$.

\begin{proof} Suppose $u$ is not an isolated critical point, so that there
exists a sequence of admissible conformal factors $\{u_i\}$, $u_i \neq u$,
converging in $C^{\infty}$ to $u$, normalized so that $\int_M (u - u_i )dV_u =
0$.  We aim to use the convexity properties to show that the minimum eigenvalue
of the linear operator
\begin{align*}
L(\phi) =&\ - \IP{T_1(g_u^{-1} A_u), \N_{g_u}^2 \phi}_{g_u} - 4 \bar{\gs} \phi
\end{align*}
is zero.  Since $u$ satisfies $\gs_2(A_u) \equiv \bar{\gs}$ and has unit volume, this
lowest eigenvalue is characterized variationally
as
\begin{align*}
\gl_1 = \inf_{\{\phi | \int_M \phi dV_u = 0\}}\bar{\gs} \int_M \left[
\gs_2(A_u)^{-1} \IP{T_1(A_u), \N \phi \otimes \N \phi} - 4 \phi^2 \right]
dV_u.
\end{align*}
It follows from Corollary \ref{n4andrews} that $\gl_1 \geq 0$, with equality if
and only if $(M^4, g_u)$ is isometric to $(S^4, g_{S^4})$.  We suppose that
$\gl_1 > 0$ and derive a contradiction.

Fix a sufficiently large $i$ so
that the path
\begin{align*}
 w(x,t) = (1-t) u + t u_i
\end{align*}
consists of admissible functions.  Note that $w_{tt} = 0$, and by construction
$\frac{d F(w(\cdot,t))}{dt}(0) =
\frac{dF(w(\cdot,t))}{dt}(1) = 0$.  It follows that for any $i$ there exists
$t_{i} \in [0,1]$ such that $\frac{d^2 F(w(\cdot,t))}{dt^2}(t_{i}) = 0$.  We aim
to derive a contradiction from this setup.  First we make a second variation
calculation along this path using (\ref{CYvar}) and (\ref{densep}),
yielding
\begin{align*}
\frac{d^2}{dt^2} F[w(\cdot,t)] =&\ \frac{d}{dt} \int_M w_t \left( -
\sigma_2(g_w^{-1} A_w) +
\bar{\gs} \right) dV_{w}\\
=&\ \int_M w_{tt} \left( - \sigma_2(g_w^{-1} A_w) + \bar{\gs} \right) dV_w + \int_M
\left[- w_t \IP{T_1(g_w^{-1} A_w), \N^2 w_t} - n \bar{\gs} w_t^2 \right]
dV_w\\
=&\ \int_M \left[ \IP{T_1(g_w^{-1} A_w), \N w_t \otimes \N w_t} - n \bar{\gs}
w_t^2 \right] dV_w\\
=&\ \bar{\gs} \int_M \left[ \gs_k(g_w^{-1} A_w)^{-1} \IP{T_1(g_w^{-1} A_w), \N w_t \otimes
\N w_t} - n w_t^2 \right] dV_w.
\end{align*}
We next evaluate this at $t_i$.  Using that $w^i := w(\cdot,t_i)$ converges to $u$ as $i
\to \infty$ yields
\begin{align*}
 0 =&\ \int_M \left[ \IP{T_1(g_{w^{i}}^{-1} A_{w^i}), \N w_t \otimes \N w_t} -
n
\bar{\gs}
w_t^2 \right] dV_{w^i}\\
=&\ \int_M \left[ \IP{(1 - o(1)) T_1(g_{u_0}^{-1} A_{u_0}), \N w_t \otimes \N
w_t} - n \bar{\gs}
w_t^2 \right] (1 - o(1)) dV_{u_0}\\
=&\ \bar{\gs} \int_M \left[ \gs_k(g_{u_0}^{-1} A_{u_0})^{-1} \IP{T_1(g_{u_0}^{-1} A_{u_0}),
\N w_t \otimes \N w_t} - n w_t^2 \right] dV_{u_0} -o(1)\\
\geq&\ \bar{\gs} \gl_1 \int_M w_t^2 dV_{u_0} - o(1).
\end{align*}
If $\gl_1 > 0$ then for sufficiently large $i$ this implies that $w_{t} = u_i -
u = 0$, a contradiction.  It follows that $\gl_1 = 0$, and hence
by Corollary \ref{n4andrews} $(M^4, g_u)$ is isometric to $(S^4, g_{S^4})$.
\end{proof}
\end{lemma}

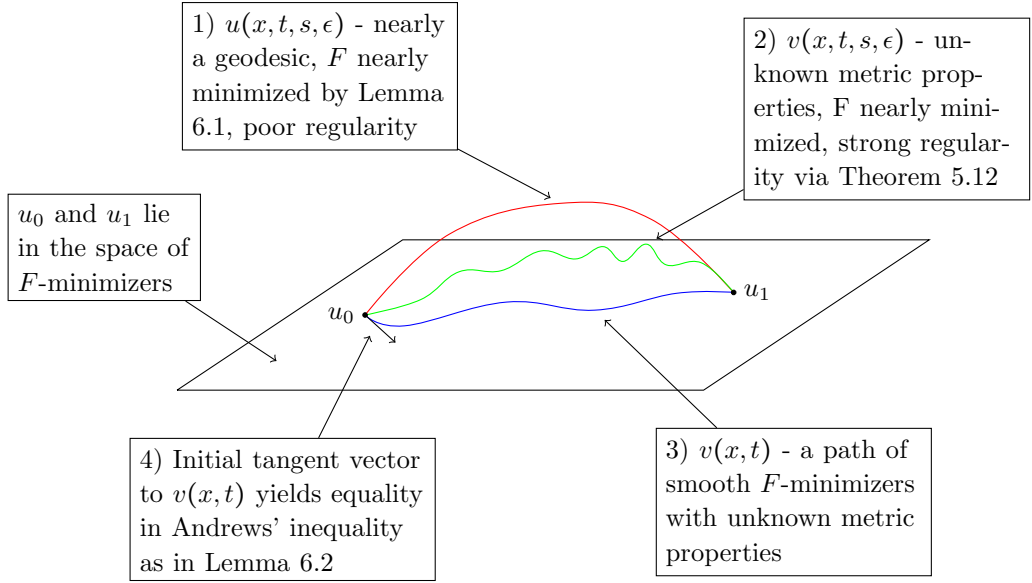
\begin{figure} \label{fig1}
\begin{tikzpicture}

\draw [black](-4,-1) -- (-1,1) -- (6,1) -- (3,-1) -- (-4,-1);

\node (v9) at (1.1005,1.4861) {};

\node (v10) at (1.9301,0.7128) {};
\node (v11) at (2.2394,0.9424) {};
\node (v14) at (-1.4366,-0.2777) {};
\node (v15) at (1.5504,0.0642) {};
\node (v17) at (-2.6787,-0.6116) {};

\draw  [red] plot[smooth, tension=.7] coordinates {(-1.5,0)(-0.3757,1.0783) (v9) (2.305,1.2724) (3.4,0.3)};
\draw  [green] plot[smooth, tension=.7] coordinates {(-1.5,0)(-0.76,0.2051)(-0.2492,0.5753)(0.3085,0.5894)(0.8381,0.8519)(1.2974,0.7456)(1.6629,0.9096) (v10) (v11) (2.5065,0.6659)(2.9705,0.694) (3.4,0.3)};
\draw  [blue] plot[smooth, tension=.7] coordinates {(-1.5,0)(-0.9756,-0.1441) (0.3835,0.172) (v15) (2.6424,0.2891)(3.4,0.3)};

\draw [fill=black] (-1.5,0) node (v2) [left]{\small{$u_0$}} circle (0.03);
\draw [fill=black] (3.4,0.3) node (v3) [right]{\small{$u_1$}} circle (0.03);

\node [draw,text width=3.4cm] (v5) at (-2.1004,3.1762) {\small{1) $u(x,t,s,\epsilon)$ - nearly a geodesic, $F$ nearly minimized by Lemma \ref{uniquelemma1}, poor regularity}};
\node[draw,text width=3.54cm] (v6) at (5.4368,2.7403) {\small{2) $v(x,t,s,\epsilon)$ - unknown metric properties, F nearly minimized, strong regularity via Theorem \ref{gwsmoothing}}};
\node[draw,text width=3.4cm] (v7) at (4.1996,-2.4886) {\small{3) $v(x,t)$ - a path of smooth $F$-minimizers with unknown metric properties}};
\node[draw,text width=3.8cm] (v13) at (-2.596,-2.587) {\small{4) Initial tangent vector to $v(x,t)$ yields equality in Andrews' inequality as in Lemma \ref{Fisolated}}};
\node[draw,text width=2.3cm] (v16) at (-4.9673,0.8998) {\small{$u_0$ and $u_1$ lie in the space of $F$-minimizers}};
\node (v8) at (-1.0965,-0.3663) {};
\draw [->] (v5) edge (v9);
\draw [->] (v6) edge (v11);
\draw [->] (v13) edge (v14.center);
\draw [->] (v7) edge (v15);
\draw [->](-1.5,0) -> (v8.center);
\draw[->] (v16)  -> (v17.center);
\end{tikzpicture}
\caption{Scheme of proof of Theorem \ref{unique}}
\end{figure}

\begin{proof}[Proof of Theorem \ref{unique}] See Figure \ref{fig1} for a schematic outline of the argument.  Suppose there exist two distinct solutions $u_0$ and $u_1$ to the $\gs_2$-Yamabe problem.  Let $u(x,t,s,\ge)$ be the family of
approximate geodesics connecting $u_0$ to $u_1$ guaranteed by Theorem \ref{approxgeodthm}.  Noting the a
priori estimates on $\brs{u}_{C^0}$ and $\brs{\N u}_{C^0}$ are independent of $s,\ge$
we have by Theorem \ref{gwsmoothing} that the solution to the flow equation (\ref{gwflow}) with
initial condition $u(\cdot,t,s,\ge)$ exists on some time interval $[0,\eta]$,
and moreover the solution at time $\eta$, call it $v(x,t,s,\ge)$ has uniform
$C^{k,\ga}$ estimates independent of $s,\ge$ and stays uniformly in the interior of $\gG_2^+$, in the
sense that $T_1(g_v^{-1} A_v)$ has uniform upper and lower bounds.
Due to these
estimates we can obtain one-parameter family of smooth functions $v(x,t) =
\lim_{s,\ge \to 0} v(x,t,s,\ge)$, which is continuous in $t$.  Moreover, by
Lemmas
\ref{gwFdecrease} and \ref{uniquelemma1} we see that $F[v(\cdot,t)] = F[u_0]$.
It follows that $v(\cdot,t)$ is a nontrivial path of critical points for $F$
through $u_0$, and hence by Lemma \ref{Fisolated} we conclude that $(M^4, g_u)$
is isometric to $(S^4, g_{S^4})$.
\end{proof}

\end{document}